\documentclass[12pt]{amsart}
\usepackage{amssymb,amsmath,amsthm,andy,fullpage}
\usepackage{pdfsync}
\usepackage{setspace}

\newtheorem{thm}{Theorem}[section]
\newtheorem{prop}[thm]{Proposition}
\newtheorem{lem}[thm]{Lemma}
\newtheorem{cor}[thm]{Corollary}

\theoremstyle{definition}
\newtheorem{defn}[thm]{Definition}

\theoremstyle{remark}
\newtheorem{rem}[thm]{Remark}
\newtheorem{remark}[thm]{Remark}

\newcommand{\eps}{\varepsilon}

\DeclareMathOperator{\dist}{dist}

\DeclareMathOperator{\esssup}{ess\,sup}
\newcommand{\imbed}{\hookrightarrow}

\renewcommand{\H}{\mathfrak H}
\newcommand{\Ta}{T}
\renewcommand{\T}{\text{tan}}

\DeclareMathOperator{\rea}{Reach}
\newcommand{\X}{\mathcal X}
\newcommand{\tnabla}{\nabla^{\tan}}
\newcommand{\oldb}{\tau}
\newcommand{\oldL}{Z}
\newcommand{\oldD}{\mathfrak{D}}

\begin{document}

\title{Sobolev Spaces and Elliptic Theory on Unbounded Domains in $\R^n$}

\author{Phillip S.\ Harrington and Andrew Raich}

\thanks{The first author is partially supported by NSF grant DMS-1002332 and
second author is partially supported by NSF grant DMS-0855822}

\address{Department of Mathematical Sciences, SCEN 301, 1 University of Arkansas, Fayetteville, AR 72701}
\email{psharrin@uark.edu \\ araich@uark.edu}



\begin{abstract}In this article, we develop the theory of weighted $L^2$ Sobolev spaces on unbounded domains in $\R^n$. As an application,
we establish the elliptic theory for elliptic operators and prove trace and extension results analogous to the bounded, unweighted case.
\end{abstract}

\maketitle

\setcounter{tocdepth}{1}
\tableofcontents

In this article, we develop  weighted $L^2$-Sobolev spaces and elliptic theory on unbounded domains in $\R^n$.
For spaces of functions with suitable regularity and domains with
regular boundaries, we show that traces exist and functions defined on the boundary extend in a
bounded manner. In the second part of the paper, we show that elliptic equations gain the full number of derivatives up to the
boundary and satisfying an elliptic equation is sufficient for taking traces
for functions as rough as $L^2$.

With this article, we are laying the groundwork to develop the $L^2$-theory for  the $\dbar$ and $\dbarb$ equations on unbounded domains and their boundaries in $\C^n$. Weighted
$L^2$ spaces are instrumental tools in several complex variables, and the analysis cannot proceed without them.

Unlike the bounded case,
however, unweighted Sobolev spaces on unbounded domains fail to have many critical features, such as the Rellich identity, and so we develop spaces with these properties in mind.
One method to solve $\dbarb$-problem involves using extension and trace operators, so we need fractional Besov and Sobolev spaces. As a result of our several complex
variables considerations (to be developed in later papers), both the types of weighted Sobolev spaces we study and the types of results we prove are quite different than what appears in the literature.
See, for example \cite{Kuf85}. The weights that authors typically study involve powers of the distance to the boundary \cite{MiTa06, Can03}. Although these weights are quite natural and reflect the geometry of
the boundary, they are not the only useful weights in several complex variables (see, e.g., \cite{Har09,Hor65,Sha85,Koh86,HaRa11,HaRa12,Rai10,Str10}). With respect to the literature on
elliptic theory on unbounded domains, authors seem
to be less concerned with proving trace results for solutions to elliptic equations and more interested in solvability, typically for the Dirichlet problem
(e.g., see \cite{BoMoTr08} and the references contained
within). Even when the author
solves an elliptic equation with a nonzero boundary condition (e.g., \cite{Kim08}), the derivatives are the standard derivatives, and not the weighted derivatives that we consider.
Consequently, we must build the theory from the most basic buiding blocks.

Our Sobolev space techniques mainly involve real interpolation, so they define Besov spaces.
The fractional Sobolev spaces and Besov spaces agree (see, for example, \cite{LiMa72} or \cite{BeLo76}),
and in the elliptic regularity section of the paper, we use the fractional Sobolev and Besov spaces interchangeably. In fact, even at the integer levels, the main
results hold for the interpolated spaces $B^{k;2,2}(\Omega,\vp; X)$ and Sobolev spaces $W^{k,2}(\Omega,\vp; X)$ (the former by interpolation and the latter by direct proof).

%
%
\section{Preliminaries}

Let $\Omega\subset \R^n$ be an open set and let $\vp:\Omega\to\R$ be $C^\infty$. Define the weighted $L^p$-space
\[
L^p(\Omega,\vp) = \{f:\Omega\to\C : \int_\Omega |f|^p e^{-\vp}\, dV<\infty \}
\]
where $dV$ is Lebesgue measure on $\C^n$. Let $\bd\Omega$ be the boundary of $\Omega$.  We will always assume that $\bd\Omega$ is at least Lipschitz, so that integration by parts is always justified.  For most results we will need additional boundary regularity, as indicated below.
\subsection{Hypotheses on $\Omega$, $\vp$, and $\rho$}
Let $A\subset\R^n$.

Let $\delta_A$ be the distance function from $A$, i.e., $\delta_A(x) = \inf_{y\in A} |x-y|$.
Let $U_A = \{ x\in\R^n : \text{ there exists a unique point }y\in A \text{ such that }\delta_A(x) = |y-x| \}$.
Define  $\pi_A:U_A \to A$ by $\pi_A(x) =y$. The following concepts were introduced in \cite{Fed59}.
\begin{defn}If $y\in A$, then define the \bfem{reach} of $A$ at $y$ by
\[
\rea(A,y) = \sup\{r\geq 0: B(y,r)\subset U_A\}
\]
and the \bfem{reach} of $A$ to be
\[
\rea(A) = \inf\{ \rea(A,y): y\in A\}.
\]
\end{defn}

The majority of our results use a subset of the following hypotheses.  Fix $m\in\mathbb{N}$, $m\geq 2$.
\begin{enumerate} \renewcommand{\labelenumi}{H\Roman{enumi}.}
\item The domain $\Omega$ has a $C^m$ boundary with positive reach. Moreover,
there exists $\ep>0$ and a defining function $\rho$ so that on $\Omega_\ep' = \{y : \delta_{\bd\Omega}(y)<\ep\}$, $\|\rho\|_{C^m(\Omega_\ep')}<\infty$ (i.e., $\Omega$ is uniformly $C^m$ in the sense of \cite{HaRa13d}).
\item There exists $\theta\in(0,1)$ so that
\[
\lim_{\atopp{|x|\to\infty}{x\in\Omega}} \big(\theta|\nabla\vp|^2 + \triangle \vp \big) = \infty
\]
where
\[
|\nabla\vp|^2 = \sum_{j=1}^n \Big|\frac{\p\vp}{\p x_k}\Big|^2 .
\]
\item There exists $\theta\in(0,1)$ so that
\[
\lim_{\atopp{|x|\to\infty}{x\in\Omega}} \big(\theta|\nabla\vp|^2 - \triangle \vp \big) = \infty
\]
\item There exists a constant $C_m>0 $ so that
\[
|\nabla^k\vp| \leq C_m(1+|\nabla\vp|)
\]
for $1\leq k \leq m$ and $x\in\Omega$.
\item Hypotheses (HII)-(HIV) can be extended to $\R^n$.

\item If $\frac{\partial}{\partial \nu}$ denotes the outward unit normal to $\bd\Omega$, we have
\[
  \inf_{r>0}\sup_{|x|>r,x\in\bd\Omega}|\nabla\varphi|^{-1}\left|\frac{\partial\varphi}{\partial\nu}\right|<1.
\]
\end{enumerate}
(HII) and (HIII) have their origin in \cite{Gan11,GaHa10} (who in turn adapt the ideas in \cite{KnMi94}).

The family of examples \emph{par excellence} of  weight functions is
\[
\vp(x) = t|x|^2
\]
for any nonzero $t\in\R$.  Such functions always satisfy (HII)-(HV) ((HI) is examined in detail in \cite{HaRa13d}). It is possible to construct domains for which (HVI) fails for this choice of $\vp$, but observe that if $\Omega$ satisfies (HVI) for $\vp(x)=t|x|^2$, then any isometry of $\mathbb{R}^n$ will map $\Omega$ to another domain which also satisfies (HVI) (because the composition of $|x|^2$ with any isometry will equal $|x|^2$ plus lower order terms).

\subsection{Weighted Sobolev spaces}
Set $D_j = \frac{\p}{\p x_j}$ and define the weighted differential operators
\[
X_j = \frac{\p}{\p x_j} - \frac{\p\vp}{\p x_j} = e^{\vp}\frac{\p}{\p x_j} e^{-\vp} ,\quad 1\leq j\leq n
\]
and
\[
\nabla_X = (X_1,\dots, X_n).
\]
\begin{defn} \label{defn:sobo space, weighted deriv}
Let $Y_j = X_j$ or $D_j$, $1\leq j \leq n$. For a nonnegative $k\in\Z$, let the weighted Sobolev space
\[
W^{k,p}(\Omega,\vp; Y) = \{ f\in L^p(\Omega,\vp) : Y^\alpha f \in L^p(\Omega,\vp) \text{ for } |\alpha|\leq k \}
\]
where $\alpha = (\alpha_1,\dots,\alpha_n)$ is an $n$-tuple of nonnegative integers and $Y^\alpha = Y_1^{\alpha_1}\cdots Y_n^{\alpha_n}$.
The space $W^{k,p}(\Omega,\vp; Y)$ has norm
\[
\| f \|_{W^{k,p}(\Omega,\vp; Y)}^p=  \sum_{|\alpha|\leq k} \| Y^\alpha f \|_{L^p(\Omega,\vp)}^p .
\]
Also, let
\begin{multline*}
W^{k,p}_0(\Omega,\vp; Y)
= \{ g\in W^{k,p}(\Omega,\vp;Y) :\\
\text{ there exists } \psi_\ell\in C^\infty_c(\Omega) \text{ satisfying } \| g-\psi_\ell \|_{W^{k,p}(\Omega,\vp; Y)} \to 0 \text{ as }\ell\to\infty \}.
\end{multline*}
In other words, $W^{k,p}_0(\Omega,\vp; Y)$ is the closure of $C^\infty_c(\Omega)$ in the $W^{k,p}(\Omega,\vp; Y)$-norm.
\end{defn}

\begin{remark} Our analysis focuses on the weighted spaces $W^{k,p}(\Omega,\vp; X)$ and we prove results on the spaces $W^{k,p}(\Omega,\vp; D)$ only where necessary. The choice
of which space to focus on is not central to the theory. We could have written the arguments with the roles of the two spaces reversed.
\end{remark}

\subsection{Weighted Sobolev spaces on $\bd\Omega$}
Let $\ep>0$ and set $M=\bd\Omega$. Recall that
\[
\Omega'_\ep = \{x\in\R^n : \dist(x,M)<\ep\}
\]
For discussions involving $M$, we always assume \emph{(HI)} and $m\geq 2$. Therefore, by \cite{HaRa13d}, there exists $\ep>0$ and a defining function $\rho$ so that
$\|\rho\|_{C^m(\Omega_\ep')}<\infty$ and $|d\rho|=1$ on $\bd\Omega$.
Let ${\oldL}_1,\dots,{\oldL}_{n-1}\in TM$ be an orthonormal basis near a point $x\in M$ and let  ${\oldL}_n = \frac{\partial}{\partial\nu}$ be the unit outward normal to $\Omega$. Moreover, ${\oldL}_n$
is also the unit normal to the level curves of $\rho$ (pointing in the direction in which $\rho$ increases).
For $1\leq j \leq n$, set
\[
\Ta_j = {\oldL}_j - {\oldL}_j(\vp).
\]
We call a first order differential operator $T$ \bfem{tangential} if the first order component of $T$ is tangential. For $1\leq j \leq n -1$, ${\oldL}_j$ is defined locally, and if $U$ is a neighborhood
on which ${\oldL}_1,\dots, {\oldL}_{n-1}$ form a basis of $T(M\cap U)$, we denote ${\oldL}_j$ by ${\oldL}_j^U$ to emphasize the dependence on $U$. In analogy to $\nabla_X$, we define
$\nabla_\Ta = (T_1,\dots, T_{n})$,
\[
\tnabla_\Ta = (T_1,\dots, T_{n-1}) \quad\text{and}\quad \tnabla_\oldL =  ({\oldL}_1,\dots, {\oldL}_{n-1})
\]

By \emph{(HI)}, we can construct an open cover $\left\{U_j\right\}$ of $\Omega'_{\epsilon}$ where
$U_j$ are of comparable surface area and admit local coordinates ${\oldL}_1,\ldots,{\oldL}_{n-1}$ with coefficients bounded uniformly in $C^{m-1}$. Let $\chi_j$ be a $C^m$ partition of unity
subordinate to $\{U_j\}$ where $\chi_j$ are uniformly bounded in the $C^m$ norm. With $\chi_j$ in hand, we set $v_j = v \chi_j$, so $v = \sum_{j=1}^\infty v_j$.
Observe that we have the following equivalent norms on $W^{k,2}(\Omega_\ep',\vp;X)$ and $W^{k,2}(\Omega_\ep',\vp;D)$, respectively:
\begin{align*}
\| v \|_{W^{k,2}(\Omega_\ep',\vp;X)} \sim \sum_{j=1}^\infty \sum_{|\alpha|\leq k} \| \Ta_{U_j}^\alpha v_j \|_{L^2(\Omega_\ep',\vp)}
& \text{ and } &
\| v \|_{W^{k,2}(\Omega_\ep',\vp;D)} \sim \sum_{j=1}^\infty \sum_{|\alpha|\leq k} \| \oldL^\alpha_{U_j} v_j \|_{L^2(\Omega_\ep',\vp)}
\end{align*}
where $\Ta^U_\ell = \oldL^U_\ell - \oldL^U_\ell(\vp)$ and $\Ta_U^\alpha = \Ta^U_{\alpha_1}\cdots \Ta^U_{\alpha_{|\alpha|}}$ (and similarly for $\oldL^\alpha_U$).

For the boundary Sobolev space, set
\begin{multline*}
W^{k,p}(M,\vp;T) =\\
 \{ f \in L^p(M,\vp): \Ta^\alpha f \in L^p(M,\vp),\ |\alpha|\leq k \text{ and }T_{\alpha_j} \text{ is tangential for } 1\leq j \leq k\}.
\end{multline*}

\subsection{Notation for differential operators}
In the second part of the paper, we establish the elliptic theory for strongly elliptic operators $\Omega\subset\R^n$ that satisfy \emph{(HI)}-\emph{(HV)} (and sometimes \emph{(HVI)} as well). Much of our development follows the outline in \cite{Fol95}.
Let $L$ be a second order operator of the form
\begin{equation}\label{eqn:L def'n}
L = \sum_{j,k=1}^n X_j^* a_{jk} X_k + \sum_{j=1}^n \big( b_j X_j + X_j^* b_j'\big) + b
\end{equation}
where $a_{jk}$ and $b_j'$ are functions on a neighborhood of $\bar\Omega$ that are bounded in the $C^1$ norm,
and $b_j$ and $b$ are bounded functions on a neighborhood of $\bar\Omega$.

Note that the formal adjoint $(X^\alpha)^* = (-1)^{|\alpha|} D^\alpha$.

The \bfem{formal adjoint} of $L$ is the operator given by the formula
\[
(L^* v, u)_\vp = (v, Lu)_\vp
\]
so integration by parts yields that
\[
L^* = \sum_{j,k=1}^n X_k^* \overline{a_{jk}} X_j + \sum_{j=1}^n \big( \overline{b_j'} X_j + X_j^* \overline{b_j} \big) + \bar b.
\]
We say that the operator $L$ is \bfem{strongly elliptic} on $\bar\Omega$ if there exists a constant
$\theta>0$ so that
\begin{equation}\label{eqn:A is positive definite}
\Rre\Big( \sum_{j,k=1}^n \overline{a_{jk}} \xi_j\bar\xi_k\Big) \geq \theta |\xi|^2.
\end{equation}

Associated to $L$ is a (nonunique) sesquilinear form $\oldD$ called a \bfem{Dirichlet form} given by
\begin{equation}\label{eqn:Dirichlet form}
{\oldD}(v,u) = \sum_{j,k=1}^n (X_j v, a_{jk} X_k u)_\vp + \sum_{j=1}^n (v, b_j X_j  u)_\vp + \sum_{j=1}^n (X_j v, b_j'  u)_\vp + (v,bu)_\vp.
\end{equation}
$\oldD$ is called a \bfem{Dirichlet form for the operator $L$} if
\[
{\oldD}(v,u) = (v, Lu)_\vp \quad\text{for all }u,v\in C^\infty_c(\Omega).
\]
The Dirichlet form $\oldD$ given by \eqref{eqn:Dirichlet form} is called \bfem{strongly elliptic} on $\bar\Omega$ if \eqref{eqn:A is positive definite} holds.

\begin{defn}\label{defn: coercive estimate}The Dirichlet form $\oldD$ on $\Omega$ is called \bfem{coercive over $\X$} if
$W^{1,2}_0(\Omega,\vp; X) \subset\X\subset W^{1,2}(\Omega,\vp; X)$, $\X$ is closed in $W^{1,2}(\Omega,\vp;X)$,  and there exist $C>0$ and $\lambda\geq 0$ such that
\begin{equation}\label{eqn:coercive estimate}
\Rre {\oldD}(u,u) \geq C \| u \|_{W^{1,2}(\Omega,\vp; X)}^2 - \lambda \|u\|_{L^2(\Omega,\vp)}^2
\quad\text{for all }u\in\X.
\end{equation}
$\oldD$ is called \bfem{strictly coercive} if we can take $\lambda =0$.
\end{defn}
If $\oldD$ is coercive, then ${\oldD}'(v,u) = {\oldD}(v,u)+\lambda(v,u)_\vp$ is strictly coercive.

We can also consider the adjoint Dirichlet form
\[
{\oldD}^*(v,u) = \overline{{\oldD}(u,v)} \qquad\text{for all }u,v\in W^{1,2}(\Omega,\vp; X).
\]
The form $\oldD$ is called \bfem{self-adjoint} if $\oldD = \oldD^*$.

%
%
\section{Main Results}\label{sec:main results}

\subsection{Sobolev space and trace theorems}

Let $\ep>0$ and set $M=\bd\Omega$ and
\[
\Omega'_\ep = \{x\in\R^n : \dist(x,M)<\ep\}.
\]
In Section \ref{sec:real_interpolation}, we will use interpolation to define the Besov space $B^{s;p,q}$.
The following theorem is the analog of the Trace Theorem \cite[Theorem 7.39]{AdFo03}

\begin{thm}\label{thm:Trace Theorem for integer m on M} Let $m\geq 2$ and assume that
$\Omega\subset\R^n$ satisfies \emph{(HI)}-\emph{(HVI)}. If $1\leq k\leq m-1$, then
the following two conditions  on a measurable function $u$ on $M$ are equivalent:
\begin{enumerate}\renewcommand{\labelenumi}{(\alph{enumi})}
\item There exists $U\in W^{k,2}(\Omega_\ep',\vp;X)$ supported in $\Omega_\ep'$ so that $u = \Tr U$;
\item $u \in B^{k-\frac 12;2,2}(M,\vp;T)$.
\end{enumerate}
\end{thm}

The proof of Theorem \ref{thm:Trace Theorem for integer m on M} is divided into two results, each of which is more general than one direction of Theorem \ref{thm:Trace Theorem for integer m on M}.  In Section \ref{sec:proof_traces exist for u in sobo spaces on M} we will show
\begin{lem} \label{lem:traces exist for u in sobo spaces on M}
Given the hypotheses of Theorem \ref{thm:Trace Theorem for integer m on M}, if $U\in W^{k,2}(\Omega_\ep'\cap\Omega,\vp;X)$, then
$\Tr U\in  B^{k-\frac 12;2,2}(M,\vp;T)$ and there exists a constant $K$ independent of $U$ so that
\[
\| \Tr U \|_{B^{k-\frac 12;2,2}(M,\vp;T)} \leq K \|U \|_{W^{k,2}(\Omega_\ep',\vp;X)}.
\]
\end{lem}

\begin{remark}The result also holds (by the same proof) if we replace $\Omega_\ep'\cap\Omega$ with $\Omega_\ep'\cap\Omega^c$.\end{remark}

The second half of the proof of Theorem \ref{thm:Trace Theorem for integer m on M} is proven in Section \ref{sec:proof_traces of normal derivatives}, as part of the general result:
\begin{thm}\label{thm:traces of normal derivatives}
Let $\ell,\ell'$ be integers so that $0\leq \ell+\ell' \leq m-2$.
If $u\in B^{\ell+\frac 12;2,2}(M,\vp;T)$, then
\[
u = \Tr \frac{\p^{\ell'}U}{\p\nu^{\ell'}}
\]
for some $U\in W^{\ell+\ell'+1,2}(\Omega_\ep',\vp;X)$ supported in $\Omega_\ep'$ satisfying
\[
\Tr U = \cdots = \Tr \frac{\p^{\ell'-1}U}{\p\nu^{\ell'-1}} =0
\]
and
\[
\| U \|_{W^{\ell+\ell'+1,2}(\Omega,\vp; X)} \leq C \| u\|_B
\]
for some $C$ independent of $U$ and $u$.
\end{thm}

The trace and extension theorems above allow us to prove the following result concerning the equality of the spaces with weighted and unweighted derivatives. We also prove the following Rellich identity in Section \ref{sec:Besov_spaces_and_Additional_Trace_Results}.
\begin{prop}\label{prop:weighted and unweighted derivs have equiv norm}
Let $\Omega$ satisfy (HI)-(HVI). Then for $0 \leq k \leq m$, $W^{k,2}(\Omega,\vp;X) = W^{k,2}(\Omega,\vp;D)$. Furthermore, if $m\geq 2$ and $1\leq k\leq m-1$, then
$W^{k,2}(\Omega,\vp;X)$ embeds compactly in $W^{k-1,2}(\Omega,\vp;X)$.
\end{prop}
The analog of Proposition \ref{prop:weighted and unweighted derivs have equiv norm} for $M$ is contained in Corollary \ref{cor:compactness of W^k into W^k-1}.
It is easier in this case since $C^\infty_c(M)$ is dense in $W^{\ell,2}(M,\vp;\cdot)$ where $\cdot$ is either $\oldL$ or $T$. Likewise, for $W^{1,2}_0(\Omega,\vp;X)$, the result
is easier (though not easy) and is contained Proposition \ref{prop:compactness of W^1_0 into L^2} and its corollaries.

A useful application of the trace and extension theorems is the construction of a simple $(k,2)$-extension operator for each $k$, $1\leq k\leq m-1$. Recall that a simple $(k,2)$-extension operator $E:W^{k,2}(\Omega,\vp;X)\to W^{k,2}(\R^n,\vp;X)$ is one that satisfies
$Eu(x) = u(x)$ for a.e.\ $x\in\Omega$ and there exists a constant $C = C(k)$ so that $\|Eu\|_{W^{k,2}(\R^n,\vp;X)} \leq C\|u\|_{W^{k,2}(\Omega,\vp;X)}$.  In Section \ref{sec:proof_traces of normal derivatives} we will show:
\begin{thm}\label{thm:simple extension operators exist}
Let $\Omega$ satisfy (HI)-(HVI). Then for $1 \leq k \leq m-1$, there exists a simple $(k,2)$-extension operator.
\end{thm}

Our final embedding result is proven in Section \ref{sec:Besov_spaces_and_Additional_Trace_Results}.
\begin{thm}\label{thm:trace theorem on M}
Let $M, \Omega$, and $\vp$ satisfy the hypotheses of Theorem \ref{thm:traces of normal derivatives}
If  $s>1/2$ and $1\leq q\leq \infty$, then
\[
B^{s;2,q}(\Omega,\vp; X) \hookrightarrow B^{s-1/2; 2,q}(M,\vp;T).
\]
\end{thm}

\subsection{Elliptic regularity -- solvability}
The Sobolev space theory that we develop is powerful enough that it allows us to adapt the proofs in the bounded, unweighted setting in a straight
forward manner and establish the following theorems, see \cite[Chapter 7]{Fol95}. In particular, we can establish that
 strong ellipticity is equivalent to G\aa rding's inequality and solve the $(\X,\oldD)$ Boundary Value
Problem (BVP): namely,
for a closed subspace $\X$ satisfying $W^{1,2}_0(\Omega,\vp; X) \subset \X \subset W^{1,2}(\Omega,\vp; X)$ and
$f\in L^2(\Omega)$, find $u\in \X$ so that ${\oldD}(v,u) = (v,f)_{\varphi}$ for all $v\in \X$. The case $\X  = W^{1,2}_0(\Omega,\vp; X)$
is the classical Dirichlet problem, but we also want
to include the case $\X = W^{1,2}(\Omega,\vp; X)$.

Note that $C^\infty_c(\Omega)\subset \X$, so a solution of the $(\X,\oldD)$ BVP
will satisfy $(L^* v, u)_\vp = (v,f)_\vp$ for all $v\in C^\infty_c(\Omega)$ and hence will be a distributional
solution to $Lu=f$. Furthermore, the requirement that ${\oldD}(v,u) = (v,f)_\vp$ for all $v\in \X$ leads to a free boundary condition, i.e., integration by parts
imposes a boundary condition on $u$.

\begin{thm}[G\aa rding's inequality] \label{thm: Garding's inequality}Let
\[
{\oldD}(v,u) = \sum_{j,k=1}^n (X_j v, a_{jk} X_k u)_\vp + \sum_{j=1}^n (v, b_j X_j  u)_\vp + \sum_{j=1}^n (X_j v, b_j'  u)_\vp + (v,bu)_\vp
\]
be a strongly elliptic Dirichlet form on $\Omega$ and suppose that $a_{jk}, b_j, b_j', b$ are bounded on $\Omega$.  Then
$\oldD$ is coercive over $W^{1,2}(\Omega,\vp; X)$ (and hence over any $\X\subset W^{1,2}(\Omega,\vp; X)$ that contains
$W^{1,2}_0(\Omega,\vp;X)$).
\end{thm}

The converse to G\aa rding's inequality holds as well.
\begin{thm}\label{thm:Garding converse}
If the Dirichlet form $\oldD$ is coercive over $W^{1,2}_0(\Omega,\vp; X)$ and $a_{jk}\in C(\bar\Omega)$, then $\oldD$ is strongly elliptic.
\end{thm}

We can prove existence and uniqueness of weak solutions for operators giving rise to strictly coercive Dirichlet forms.
\begin{thm}\label{thm:existence of weak solns -- strictly coercive}
Let $\X$ be a closed subspace of $W^{1,2}(\Omega,\vp; X)$ that contains $W^{1,2}_0(\Omega,\vp; X)$ and let $\oldD$ be a Dirichlet form that is strictly coercive
over $\X$. There is a bounded, injective operator $A:L^2(\Omega,\vp)\to \X$ that solves the $(\X,\oldD)$ BVP, that is, ${\oldD}(v,Af) = (v,f)_\vp$ for all $v\in\X$ and $f\in L^2(\Omega,\vp)$.
\end{thm}

Even in the case $\oldD$ is not strictly coercive, we can still gain information regarding weak solutions.
\begin{thm}\label{thm:weak solutions, info about kernel}
Let $\X$ be a closed subspace of $W^{1,2}(\Omega,\vp; X)$ that contains $W^{1,2}_0(\Omega,\vp; X)$. Let $\oldD$ be a Dirichlet form that is coercive
over $\X$. Define
\[
V = \{u\in \X: {\oldD}(v,u)=0 \text{ for all } v\in \X\}
\]
and
\[
W = \{u\in \X: {\oldD}(u,v)=0\text{ for all } v\in\X\}.
\]
Then $\dim V = \dim W <\infty$. Moreover,
if $f\in L^2(\Omega,\vp)$, there exists $u\in\X$ so that ${\oldD}(v,u) = (v,f)_\vp$ for all $v\in\X$ if and only if $f$ is orthogonal to $W$ in
$L^2(\Omega,\vp)$ in which case the solution is unique modulo $V$. In particular, if $V = W = \{0\}$, the solution always exists and is unique.
\end{thm}

In the case that $\oldD$ is self-adjoint,
we can prove that $L^2(\Omega,\vp)$ has a basis of eigenvectors. We will see in Proposition \ref{prop:weighted and unweighted derivs have equiv norm} that
$W^{1,2}(\Omega,\vp; X)$ embeds compactly in $L^2(\Omega,\vp)$. Thus,
\begin{thm}\label{thm:basis of eigenvectors}
Let $\X$ be a closed subspace of $W^{1,2}(\Omega,\vp; X)$ that contains $W^{1,2}_0(\Omega,\vp; X)$.
Suppose that $\oldD$ is a Dirichlet form that is coercive
over $\X$ and satisfies $\oldD = {\oldD}^*$. There exists an orthonormal basis $\{u_j\}$ of $L^2(\Omega,\vp)$ consisting of eigenfunctions for the $(\X,\oldD)$ BVP; that is, for each $j$, there
exists $u_j\in\X$ and a constant $\mu_j\in\R$ so that ${\oldD}(v,u_j) = \mu_j(v,u_j)_\vp$ for all $v\in\X$. Moreover, $\mu_j>-\lambda$ for all $j$ where $\lambda$ is the constant
in the coercive estimate (\ref{eqn:coercive estimate}), $\lim_{j\to\infty}\mu_j = \infty$, and $u_j\in C^\infty(\Omega)$ for all $j$.
\end{thm}

\subsection{Elliptic regularity -- estimates}

We can prove elliptic regularity in the interior of $\Omega$ in Section \ref{sec:interior est}.
\begin{thm}\label{thm: elliptic regularity, higher order, interior}
Let $\Omega\subset\R^n$ satisfy {(HI)}-{(HV)} for some $m\geq 2$. Let $L$ be defined by (\ref{eqn:L def'n}) and
$a_{jk},  b_j'\in C^{\ell+1}(\Omega)\cap W^{\ell+1,\infty}(\Omega)$ and $b_j,b \in C^\ell(\Omega)\cap W^{\ell,\infty}(\Omega)$ for some $0\leq\ell\leq m$.
Assume that $f\in W^{\ell,2}(\Omega,\vp; X)$ and $L$ is strongly elliptic.
Suppose that  $u\in W^{1,2}(\Omega,\vp; X)$ is a weak solution (i.e., ${\oldD}(v,u) = (v,f)_\vp$ for all $v\in W^{1,2}_0(\Omega,\vp; X)$) of the elliptic PDE
\[
Lu =f \quad\text{in }\Omega.
\]
Then $u\in W^{\ell+2,2}_{\loc}(\Omega,\vp; X)$ and if $V\subset\Omega$ is open and satisfies $\dist(V,\bd\Omega)>0$,
\begin{equation}\label{eqn:u in W^m+2 is bounded by f, u in W^m -- interior}
\| u \|_{W^{\ell+2,2}(V,\vp;X)} \leq C ( \| f \|_{W^\ell(\Omega,\vp; X)} + \|u\|_{L^2(\Omega,\vp)})
\end{equation}
where $C = C(\dist(V,\bd\Omega),C_{|\alpha|},\|a_{jk}\|_{C^{\ell+1}(\Omega)},\|b_j\|_{C^{\ell}(\Omega)}, \|b_j'\|_{C^{\ell+1}(\Omega)}, \|b\|_{C^\ell(\Omega)},n,\theta,\ell)$.
\end{thm}
Note that the inequality in Theorem \ref{thm: elliptic regularity, higher order, interior} is not an \emph{a priori} inequality. The meaning of (\ref{eqn:u in W^m+2 is bounded by f, u in W^m -- interior}) is that if the
right-hand side is finite, then $u\in W^{\ell+2,2}(V,\vp;X)$.

The case that $\Omega$ is bounded is not the only case for which we know the hypothesis that $u\in W^{1,2}(\Omega,\vp; X)$ is satisfied.
Indeed, we if combine Theorem \ref{thm:existence of weak solns -- strictly coercive}
and Theorem \ref{thm:weak solutions, info about kernel} with Theorem \ref{thm: elliptic regularity, higher order, interior} for $\ell=0$, we have the
following corollary.
\begin{cor} Let $L$, $V$, and $\Omega$ be as in Theorem \ref{thm: elliptic regularity, higher order, interior}. Let $\X$ be a closed subspace of $W^{1,2}(\Omega,\vp; X)$ that contains
$W^{1,2}_0(\Omega,\vp; X)$. Let $\oldD$ be the Dirichlet form corresponding to $L$ in $\X$. If any of the following conditions hold:
\begin{enumerate}\renewcommand{\labelenumi}{(\roman{enumi})}
\item $\oldD$ is strictly coercive,
\item $f \perp W = \{w\in\X: {\oldD}(w,v)=0 \text{ for all }v\in\X\}$ and $u$ is the weak solution that is orthogonal to $V$,
\end{enumerate}
then $u\in \X$ and hence in $W^{2,2}(V,\vp;X)$.
\end{cor}

We can also prove that elliptic regularity holds near the boundary for weak solutions of the partial differential equation $Lu=f$, in Sections \ref{sec:proof_sol'n regularity near boundary, higher order_0} and \ref{sec:proof_sol'n regularity near boundary, higher order_1}.
%
%
\begin{thm}\label{thm:sol'n regularity near boundary, higher order}Let $\Omega\subset\R^n$ satisfy {(HI)}-{(HVI)} with $m\geq 3$. Let $0\leq \ell\leq m-3$ and
the operator $L$ be defined by (\ref{eqn:L def'n})
where $a_{jk}, b_j' \in C^{\ell+1}(\bar\Omega)\cap W^{\ell+1,\infty}(\bar\Omega)$ and $b_j, b\in W^{\ell,\infty}(\Omega)$. Assume that $f\in W^{\ell,2}(\Omega,\vp; X)$
and $L$ is strongly elliptic.
Let $\X$ be a closed subspace of $W^{1,2}(\Omega,\vp; X)$ that contains $W^{1,2}_0(\Omega,\vp; X)$.
Suppose that  $u\in \X$ is a weak solution (i.e., ${\oldD}(v,u) = (v,f)_\vp$ for all $v\in \X$) of the elliptic PDE
\[
Lu =f \quad\text{in }\Omega.
\]
Then $u\in W^{\ell+2,2}(\Omega,\vp; X)$ and
\begin{equation}\label{eqn:u in W^m+2 is bounded by f, u in W^m, near bdy}
\| u \|_{W^{\ell+2,2}(\Omega,\vp; X)} \leq C ( \| f \|_{W^\ell(\Omega,\vp; X)} + \|u\|_{L^2(\Omega,\vp)})
\end{equation}
where $C = C(M_\ell,\|a_{jk}\|_{C^{\ell+1}(\Omega)},\|b_j\|_{C^{\ell+1}(\Omega)}, \|b_j'\|_{C^{\ell+1}(\Omega)}, \|b\|_{C^{\ell+1}(\Omega)},n,\theta, \|\rho\|_{C^{\ell+3}(\Omega)})$.
\end{thm}

\subsection{Boundary values of $L$-harmonic functions.}
We conclude the paper with a study of the boundary values of $L$-harmonic functions.
The goal is to show that $L$-harmonic functions (i.e., functions $u$ satisfying $Lu=0$) have unique boundary values in $W^{s-1/2,2}(\bd\Omega,\vp;T)$
when $u\in W^{s,2}(\Omega,\vp; X)$ and $s\geq 0$.

We first establish a simple but easily applicable uniqueness condition.
Let $L$ be a strongly elliptic second order operator. We would like to understand conditions on $L$ so that if $Lu=0$ and $u|_M=0$, then $u=0$.
Theorem \ref{thm:existence of weak solns -- strictly coercive} present one condition, and we will show the following in Section \ref{sec:traces_harmonic_functions}.
\begin{lem}\label{lem:D dominating 1 norms makes V=0}Let $\Omega\subset\R^n$ be a domain that satisfies {(HII)}.
Let $L$ be a strongly elliptic operator that has a Dirichlet form $\oldD$ so that for all $u\in W^{1,2}_0(\Omega,\vp; X)$ there exists a constant $c$ satisfying
\begin{equation}\label{eqn:D dominates 1 norm}
\Rre {\oldD}(u,u)\geq c \|\nabla_X u\|_{L^2(\Omega,\vp)}.
\end{equation}
If $Lu=0$ and $u\in W^{1,2}_0(\Omega,\vp; X)$, then $u\equiv0$.
\end{lem}

\begin{rem} If the operator $L$ is of the form $L = \sum_{j,k=1}^n X_j^* a_{jk} X_k + b$ where $b>0$, then $L$ satisfies (\ref{eqn:D dominates 1 norm}).
\end{rem}

With this restriction on $\oldD$, we can prove in Sections \ref{sec:proof_L times Tr is an isomorphism_2} and \ref{sec:proof_L times Tr is an isomorphism_1}:
\begin{thm}\label{thm:L times Tr is an isomorphism}
Let $\Omega\subset\R^n$ be a domain that satisfies {(HI)}-{(HVI)} for $m=2$.
Let $L$ be a strongly elliptic operator that has a Dirichlet form $\oldD$ which satisfies (\ref{eqn:D dominates 1 norm}). The map sending
\[
u \mapsto (Lu, \Tr u)
\]
is an isomorphism from
\[
W^{s,2}(\Omega,\vp; X) \to W^{s-2,2}(\Omega,\vp; X) \times W^{s-1/2,2}(\bd\Omega,\vp;T)
\]
for $1\leq s \leq m-1$.
\end{thm}

With an additional restriction on $L$, we can prove that $L$-harmonic functions in $L^2(\Omega)$ have boundary values if $s\geq 0$ in Section \ref{sec:proof_traces of L-harmonic functions}.
\begin{thm}\label{thm:traces of L-harmonic functions}
Let $\Omega\subset\R^n$ be a domain that satisfies {(HI)}-{(HVI)} for $m\geq3$. Let $L$ be  of the form
\begin{equation}\label{eqn:L has S normal}
L = \sum_{j=1}^n\Big( X_j^* X_j + b_j X_j + X_j^*b_j'\Big) + b,
\end{equation}
If $f\in W^{s,2}(\Omega,\vp; X)$ for $s\geq 0$ and $Lf=0$, then $\Tr f$ is well-defined and an element of
$W^{s-1/2,2}(\bd\Omega,\vp;T)$.
\end{thm}

%
%
\section{Facts for $W^{m,p}(\Omega,\vp; X)$, $1<p<\infty$}\label{sec:results on W^m,p}\label{sec:W^k,p spaces}

\subsection{The spaces $W^{-k,q}(\Omega,\vp; X)$ and $W^{k,p}(\Omega,\vp; X)^*$}
Let $1\leq p < \infty$ and $\frac 1p + \frac 1q =1$.
Fix $k\in\N$ and let $N(k)$ be the number of multiindices $\alpha$ where $|\alpha|\leq k$. As $k$ is fixed, we suppress the argument of $N$.
Let $\alpha^1,\dots,\alpha^N$ be an enumeration of such multiindices. For a vector $g$, we write $g = (g_1,\dots,g_N) = (g_\alpha)$ interchangeably.
For functions $g_1,\dots, g_N\in L^q(\Omega,\vp)$, there exists a
bounded linear functional $T_{g_1,\dots,g_n}$ on $W^{k,p}(\Omega,\vp; X)$ defined by
\[
T_{g_1,\dots,g_N} f  = \int_\Omega \Big( \sum_{j=1}^N X^{\alpha^j} f \overline{g_j} \Big)  \Big) e^{-\vp} dV.
\]
We can show that every functional on $W^{k,p}(\Omega,\vp; X)$ arises in this way.
\begin{prop} \label{prop:dual spaces}
For $\Omega\subset\R^n$,
let $1<p<\infty$ and $\frac 1p + \frac 1q =1$.
\begin{enumerate}\renewcommand{\labelenumi}{(\roman{enumi})}
\item The dual space $W^{-k,q}(\Omega,\vp; X) := W^{k,p}_0(\Omega,\vp; X)^*$ is the set
\[
W^{-k,q}(\Omega,\vp; X) = \Big\{ u\in \mathcal D'(\Omega) : u
= \sum_{|\alpha|\leq k} (-1)^{|\alpha|} D^{\alpha} g_{\alpha}, \ g_\alpha \in L^q(\Omega,\vp) \text{ for all }\alpha\Big\}.
\]
Moreover, the norm on $W^{-k,q}(\Omega,\vp; X)$ is given by
\begin{align*}
\| u \|_{W^{-k,q}(\Omega,\vp; X)} &:= \sup\{ |u(f)| : f\in W^{k,p}_0(\Omega,\vp; X), \| f \|_{W^{k,p}(\Omega,\vp; X)} =1 \} \\
&= \inf_{\mathcal F} \Big \{ \sum_{|\alpha|\leq k} \|g_\alpha\|_{L^q(\Omega,\vp)}^q \Big\}
\end{align*}
where $\mathcal F$ is the set of $N$-tuples $(g_1,\dots,g_N)\in L^q(\Omega,\vp)^{N}$ representing the functional $u$.

\item The dual space $W^{k,p}(\Omega,\vp; X)^*$ consists of $u\in \mathcal D'(\Omega)$ for which there exists a vector
$g = (g_\alpha) \in (L^q(\Omega,\vp))^N$ so that for all $f\in W^{k,p}(\Omega,\vp; X)$,
\[
u(f) = \sum_{|\alpha|\leq k} \big( X^\alpha f, g_\alpha\big)_\vp.
\]
Moreover, the norm on $W^{k,p}(\Omega,\vp; X)^*$
\[
\| u \|_{W^{k,p}(\Omega,\vp; X)^*} := \sup\{ |u(f)| : f\in W^{k,p}(\Omega,\vp; X), \| f \|_{W^{k,p}(\Omega,\vp; X)} =1 \}.
\]
\end{enumerate}
\end{prop}

\begin{proof} The proof is standard. See, for example, \cite[Sections 3.9, 3.12, 3.13]{AdFo03}
\end{proof}

\subsection{Approximation by $W^{k,p}(\Omega,\vp; X)$}

\begin{prop}\label{prop:density of C^infty_c}
Let $1\leq p < \infty$ and assume that $\bd\Omega$ satisfies (HI) for some $m\geq 1$. Let $1\leq \ell \leq m$ be an integer.
Then $C^\infty_c(\R^n)$ is dense in both $W^{\ell,p}(\Omega,\vp; X)$ and
$W^{\ell,p}(\Omega,\vp; D)$ in the sense that
if $\ep>0$ and $f\in W^{\ell,p}(\Omega,\vp; X)$, then there exists $\xi \in C^\infty_c(\R^n)$ so that
\[
\| \xi - f \|_{W^{\ell,p}(\Omega,\vp; X)} \leq \ep
\]
where $C$ is independent of $f$, $\vp$, and $\ep$.  Similarly,
if $\ep>0$ and $f\in W^{\ell,p}(\Omega,\vp; D)$, then there exists $\psi \in C^\infty_c(\R^n)$ so that
\[
\| \psi - f \|_{W^{\ell,p}(\Omega,\vp; D)} \leq \ep
\]
where $C$ is independent of $f$, $\vp$, and $\ep$.
\end{prop}

\begin{proof}
Let $\ep>0$  and $\chi_R$ be a smooth, nonnegative cut-off function so that $\chi_R \equiv 1$ on $B(0,R)$, $\chi_R \equiv 0$ off $B(0,2R)$,
and $|D^\alpha \chi_R| \leq C_{|\alpha|}/R^{|\alpha|}$ for $|\alpha|\geq 0$. For $R$ sufficiently large, it follows that
\[
\| (1-\chi_R)f \|_{W^{\ell,p}(\Omega,\vp; X)} \leq \ep.
\]
Let $g = \chi_R f$, and extend $g$ to be zero outside of $\Omega$.  It is enough to prove the result for $g$. Let $\Omega'$ be a $C^m$ domain satisfying
\[
  B(0,2R)\cap\Omega\subset\Omega'\subset B(0,3R)\cap\Omega.
\]
Since $\supp\chi_R \subset B(0,2R)$, $g = g 1_{\Omega'}$. Since $\Omega'$ is bounded, there exists
$C_{\Omega'}>0$ so that
\[
\|h\|_{W^{\ell,p}(\Omega',\vp;X)} \leq C_{\Omega'} \|h\|_{W^{\ell,p}(\Omega')}
\]
for any $h\in W^{\ell,p}(\Omega')$.  The function $\xi$ is constructed in the following manner. Extend $g$ to $\tilde g \in W^{\ell,p}(\R^n)$ following the technique of
\cite[Theorem 5.22]{AdFo03}. Since $g$ is identically zero in a neighborhood of $\partial\Omega'\cap\Omega$, this construction can be used to guarantee that $\tilde{g}$ is identically zero on $\Omega\backslash\Omega'$.  Form $\xi$ by mollifying $\tilde g$ in such a way that $\xi$ is also identically zero on $\Omega\backslash\Omega'$.  Since $\xi$ is constructed so that $\|g-\xi\|_{W^{\ell,p}(\Omega')}<\ep/{C_{\Omega'}}$, we have
\[
\|g-\xi\|_{W^{\ell,p}(\Omega,\vp;X)}
\leq \|g-\xi\|_{W^{\ell,p}(\Omega',\vp;X)} < \ep .
\]
\end{proof}

\subsection{Embeddings and compactness for $p=2$}
\begin{prop}\label{prop:compactness of W^1_0 into L^2}
Let $\Omega\subset\R^n$ satisfy (HII).
Then the embedding of $W^{1,2}_0(\Omega,\vp; X) \hookrightarrow L^2(\Omega,\vp)$ is compact.
\end{prop}

\begin{proof}We start by making a number of preliminary calculations.
Note that the formal adjoint $X_j^* = -\frac{\p}{\p x_j}$. This means
\begin{equation}\label{eqn:X_j + X_j^*}
(X_j + X_j^*)f = -\frac{\p\vp}{\p x_j}f
\end{equation}
and
\begin{equation}\label{eqn:[X_j,X_j^*]}
[X_j,X_j^*]f = -\frac{\p^2\vp}{\p x_j^2}f.
\end{equation}
Note that for $f\in C^\infty_c(\Omega)$,
\begin{equation}\label{eqn:X_j + X_j^* commutator}
\big( [X_j,X_j^*]f,f\big)_\vp = \big(X_j X_j^* f,f \big)_\vp - \big(X_j^* X_j f, f\big)_\vp
= \| X_j^* f\|_{L^2(\Omega,\vp)}^2-\|X_j f\|_{L^2(\Omega,\vp)}^2.
\end{equation}
Next, we see that for $\ep>0$, a small constant/large constant argument yields
\[
\| (X_j + X_j^*)f \|_{L^2(\Omega,\vp)}^2
\leq \Big(1+\frac 1{2\ep}\Big)\|X_j f\|_{L^2(\Omega,\vp)}^2 + (1+\ep)\|X_j^* f\|_{L^2(\Omega,\vp)}^2.
\]
Now set
\[
\Psi(x)  = |\nabla\vp(x)|^2 + (1+\ep)\triangle\vp(x).
\]
Consequently,
\begin{align}
\big( \Psi f, f\big)_\vp &= \sum_{j=1}^n \bigg[ \|(X_j+X^*_j)f\|_{L^2(\Omega,\vp)}^2  - (1+\ep) ([X_j,X_j^*]f,f)_\vp \bigg]\nn\\
&\leq\Big(2+\ep+\frac 1{2\ep}\Big) \sum_{j=1}^n  \|X_j f\|_{L^2(\Omega,\vp)}^2 . \label{eqn:X_j^* vs. X_j in L^2}
\end{align}
Since $C^\infty_c(\Omega)$ is dense in $W^{1,2}_0(\Omega,\vp; X)$, this inequality holds for all $f\in W^{1,2}_0(\Omega,\vp; X)$.

Let $\{f_k\}\subset W^{1,2}_0(\Omega,\vp; X)$ be a bounded sequence and set $M = \max_k \| f_k \|_{W^{1,2}(\Omega,\vp; X)}^2$.
Set $I(R) = \inf_{x\in\Omega\setminus B(0,R)}{|\Psi(x)|}$.
Since $I(R)>0$ for $R$ sufficiently large,
\begin{align}
\| f_k - f_j \|_{L^2(\Omega,\vp)}^2
&\leq \int_{B(0,R)\cap\Omega} | f_k - f_j(x)|^2 e^{-\vp}dV + \int_{\Omega\setminus B(0,R)} \frac{\Psi(x)}{I(R)}
| f_k - f_j(x)|^2 e^{-\vp}dV \nn \\
&\leq C_{\vp,R} \| f_k - f_j\|_{L^2(B(0,R))}^2 + \frac{C \| f_k - f_j\|_{W^{1,2}(\Omega,\vp; X)}^2}{I(R)} \nn \\
&\leq C_{\vp,R} \| f_k - f_j\|_{L^2(B(0,R))}^2 + C\frac{M}{I(R)}. \label{eqn:f_k-f_j good L^2 bound}
\end{align}
Fix an increasing sequence $R_j\to\infty$, so that $R_j$ satisfies $M/I(R_j) \leq 1/j$.
We may inductively construct a sequence of subsequences $f_{k^m_j}$ so that
\begin{enumerate}\renewcommand{\labelenumi}{(\roman{enumi})}
\item $f_{k^{m+1}_j}$ is a subsequence of $f_{k^m_j}$,
\item $\lim_{j\to\infty} f_{k^m_j} = f_{k^m}$ in $L^2(B(0,R_m)\cap\Omega,\vp)$, and
\item  $f_{k^m}|_{B(0,R_k)} = f_{k^\ell}$ if $\ell\leq m$.
\end{enumerate}
It is now easy to see from \eqref{eqn:f_k-f_j good L^2 bound} that
$f_{k^j_j}$ is a Cauchy sequence in $L^2(\Omega,\vp)$ and hence converges in $L^2(\Omega,\vp)$. Thus, $W^{1,2}_0(\Omega,\vp; X)$ embeds
compactly in $L^2(\Omega,\vp)$.
\end{proof}

\begin{cor}\label{cor:compactness of L^2 into W^-1}
The embedding $L^2(\Omega,\vp) \hookrightarrow W^{-1,2}(\Omega,\vp; X)$ is compact.
\end{cor}

\begin{cor}\label{cor:adjoint dominated in L^2}
Let $\Omega\subset\R^n$ satisfy (HII). There exists a constant $C>0$ so that
\[
\| \nabla f \|_{L^2(\Omega,\vp)}^2
\leq C \| \nabla_X f\|_{L^2(\Omega,\vp)}^2
\]
for $f\in W^{1,2}_0(\Omega,\vp; X)$.
\end{cor}

\begin{proof} By \eqref{eqn:X_j + X_j^* commutator} and \eqref{eqn:X_j^* vs. X_j in L^2},
\begin{align*}
\| \nabla f \|_{L^2(\Omega,\vp)}^2
= \| \nabla_X f\|_{L^2(\Omega,\vp)}^2 - \sum_{j=1}^n \big([X_j,X_j^*]f,f\big)_\vp
&\leq \| \nabla_X f\|_{L^2(\Omega,\vp)}^2 + \frac{1}{1+\ep}  \big(\Psi f,f\big)_\vp \\
&\leq C \| \nabla_X f\|_{L^2(\Omega,\vp)}^2.
\end{align*}
\end{proof}

\begin{prop}\label{prop:compactness of W^1_0(unweighted derivs) into L^2}
Let $\Omega\subset\R^n$ satisfy (HIII).
Then the embedding of $W^{1,2}_0(\Omega,\vp; D) \hookrightarrow L^2(\Omega,\vp)$ is compact.
\end{prop}

\begin{proof}The proof follows the lines of the proof of Proposition \ref{prop:compactness of W^1_0 into L^2} with
\[
\Theta(x)  = |\nabla\vp(x)|^2 - (1+\ep)\triangle\vp(x).
\]
replacing $\Psi(x)$.
\end{proof}

\begin{cor}\label{cor:compactness of L^2 into W^-1 (unweighted)}
Let $\Omega\subset\R^n$ satisfy (HIII).
Then the embedding $L^2(\Omega,\vp) \hookrightarrow W^{-1,2}(\Omega,\vp; D)$ is compact.
\end{cor}

\begin{cor}\label{cor:adjoint dominated in L^2 (unweighted)}
Let $\Omega\subset\R^n$ satisfy (HIII). Then there exists a constant $C>0$ so that
\[
\| \nabla_X f \|_{L^2(\Omega,\vp)}^2 \leq C\| \nabla f\|_{L^2(\Omega,\vp)}^2,
\]
for all $f\in W^{1,2}_0(\Omega,\vp; D)$.
\end{cor}

\begin{remark}Proposition \ref{prop:density of C^infty_c} and Corollaries \ref{cor:adjoint dominated in L^2} and \ref{cor:adjoint dominated in L^2 (unweighted)} allow us to define a number of equivalent ways to measure the $W^{k,2}_0(\Omega,\vp; X)$ norm (which we will use later on $\Omega_\ep'$)
Let $\psi\in W^k_0(\Omega,\vp; X)$ and
\[
Y_j = \frac 12(X_j -X_j^*) = \frac{\p}{\p x_j} - \frac 12 \frac{\p \vp}{\p x_j} = X_j + \frac 12 \frac{\p \vp}{\p x_j}
\text{ and }
\nabla_Y \psi = (Y_1 \psi,\dots, Y_n \psi).
\]
Note that
\begin{multline*}
\| Y_j \psi\|_{L^2(\Omega,\vp)}^2 = \Big( \frac{\p \psi}{\p x_j} - \frac 12 \frac{\p \vp}{\p x_j}\psi, \frac{\p \psi}{\p x_j} - \frac 12 \frac{\p \vp}{\p x_j}\psi\Big)_\vp \\
= \Big\| \frac{\p \psi}{\p x_j}\Big\|_{L^2(\Omega,\vp)}^2 + \frac 14 \Big\| \frac{\p\vp}{\p x_j} \psi\Big\|_{L^2(\Omega,\vp)}^2
- \Rre \Big( \frac{\p \psi}{\p x_j}, \frac{\p \vp}{\p x_j} \psi\Big)_\vp
\end{multline*}
and
\begin{multline*}
\| Y_j \psi\|_{L^2(\Omega,\vp)}^2 = \Big(X_j \psi + \frac 12 \frac{\p \vp}{\p x_j}\psi, X_j \psi + \frac 12 \frac{\p \vp}{\p x_j}\psi \Big)_\vp\\
= \big\| X_j \psi\big\|_{L^2(\Omega,\vp)}^2 + \frac 14 \Big\| \frac{\p\vp}{\p x_j} \psi\Big\|_{L^2(\Omega,\vp)}^2
+ \Rre \Big( \frac{\p \psi}{\p x_j}, \frac{\p \vp}{\p x_j} \psi\Big)_\vp.
\end{multline*}
Thus,
\[
\| Y_j \psi \|_{L^2(\Omega,\vp)}^2 = \frac 12 \Big( \Big\| \frac{\p \psi}{\p x_j} \Big\|_{L^2(\Omega,\vp)}^2 + \big\| X_j \psi\big\|_{L^2(\Omega,\vp)}^2\Big)
+ \frac 14 \Big\| \frac{\p\vp}{\p x_j} \psi\Big\|_{L^2(\Omega,\vp)}^2,
\]
so
\[
\| \nabla_Y \psi\|_{L^2(\Omega,\vp)}^2 \geq \frac 12 \Big( \| \nabla \psi\|_{L^2(\Omega,\vp)}^2 + \| \nabla_X \psi\|_{L^2(\Omega,\vp)}^2 \Big).
\]
It then follows that
\[
\| \nabla_Y \psi\|_{L^2(\Omega,\vp)}^2\sim \| \nabla \psi\|_{L^2(\Omega,\vp)}^2 \sim \| \nabla_X \psi\|_{L^2(\Omega,\vp)}^2.
\]
and consequently \emph{(HIV)} shows that for any $\ell$ so that $1\leq\ell\leq m$
\begin{equation}\label{eqn:Y,X,D equiv}
\| \psi\|_{W^{\ell,2}(\Omega,\vp; D)} \sim \| \psi\|_{W^{\ell,2}(\Omega,\vp; X)}^2 \sim \sum_{|\alpha|\leq \ell} \| Y^\alpha \psi\|_{L^2(\Omega,\vp)}^2 
\end{equation}
where the constants in $\sim$ depend on $\ell$, $n$, and $\vp$.
The reason that we introduced $Y_j$ is that $Y_j = e^{\frac 12\vp}\frac{\p}{\p x_j} e^{-\frac 12 \vp}$, so
\[
\| Y_j \psi \|_{L^2(\Omega,\vp)}^2 =  \int_{\Omega} \Big|e^{\frac 12\vp}\frac{\p}{\p x_j}\big( e^{-\frac 12 \vp} \psi\big) \Big|^2 e^{-\vp}\, dx
= \Big\| \frac{\p}{\p x_j} \big( e^{-\frac 12 \vp} \psi\big)  \Big\|_{L^2(\Omega)}.
\]
\end{remark}

%
%
\section{Sobolev spaces on $M$}\label{sec:traces on hypersurfaces}

As above with Proposition \ref{prop:dual spaces}, standard arguments yield
\begin{prop}
Let $1<p<\infty$ and $\frac 1p + \frac 1q =1$. Fix a nonnegative integer $k\leq m$ and let $N=N(k)$ be as in Proposition \ref{prop:dual spaces}.
The dual space to $W^{k,p}(M,\vp;T)$  consists of $u\in \mathcal D'(M)$ for which there exists a vector
$g = (g_\alpha) \in (L^q(M,\vp))^N$ so that for all $f\in W^{k,p}(M,\vp;T)$,
\[
u(f) = \sum_{|\alpha|\leq k} \big( T^\alpha f, g_\alpha\big)_\vp
= \sum_{|\alpha|\leq k} (-1)^{|\alpha|}\big( f, (T^\alpha)^*g_\alpha\big)_\vp
\]
where $T^\alpha$ is a tangential operator of order $|\alpha|$.
Moreover, the norm on $W^{-k,q}(M,\vp;T) := W^{k,p}(M,\vp;T)^*$
\begin{align*}
\| u \|_{W^{-k,q}(M,\vp;T)} :=& \sup\{ |u(f)| : f\in W^{k,p}(M,\vp;T), \| f \|_{W^{k,p}(M,\vp;T)} =1 \}\\
 =& \inf_{\mathcal F} \Big \{ \sum_{|\alpha|\leq k} \|g_\alpha\|_{L^q(M,\vp)}^q \Big\}
\end{align*}
where $\mathcal F$ is the set of $N$-tuples $(g_1,\dots,g_N)\in L^q(M,\vp)^{N}$ representing the functional $u$.
\end{prop}

\subsection{Approximations and embeddings for $W^{k,p}(M,\vp;T)$}\label{subsec:approx and embed on M}

When considering results on the boundary, we will generally need both (HII) and (HIII).  Adding these, it is helpful to observe that we have
\begin{equation}
\label{eq:combined_hypothesis}
  \lim_{\atopp{|x|\rightarrow\infty}{x\in\Omega}}|\nabla\varphi|=\infty.
\end{equation}
Conversely, (HIV) and \eqref{eq:combined_hypothesis} imply both (HII) and (HIII), since (HIV) with $k=2$ implies
\begin{equation}
\label{eq:laplace_estimate}
  |\Delta\varphi|\leq n|\nabla^2\varphi|\leq nC_2(1+|\nabla\varphi|).
\end{equation}

By classical results, we know that
$C^m_c(M)$ is dense in $W^{k,2}(M,\vp;T)$.

Let $B = (\oldb_{j\ell})$ be the matrix with bounded $C^{m-1}$ coefficients so that
\[
{\oldL}_j = \sum_{\ell=1}^n \oldb_{j\ell}\frac{\p}{\p x_\ell}.
\]
Since $\Ta_j = {\oldL}_j - {\oldL}_j\vp$, $T_j = (B \nabla_X)_j$.
Then $\Ta_\ell = \sum_{\ell'=1}^n {\oldb}_{\ell\ell'}X_{\ell'}$ implies
\[
\Ta_{\ell}^* = \sum_{\ell'=1}^n \Big( {\oldb}_{\ell\ell'}X_{\ell'}^* - \frac{\p {\oldb}_{\ell\ell'}}{\p x_{\ell'}}\Big).
\]

Using the formula for $\Ta_{\ell}^*$ and (\ref{eqn:X_j + X_j^*}), we observe that
\[
\Ta_j + \Ta_j^* = -\sum_{\ell=1}^n \Big( {\oldb}_{j\ell} \frac{\p \vp}{\p x_\ell} + \frac{\p {\oldb}_{j\ell}}{\p x_\ell}\Big)
= -(B\nabla\vp)_j - \sum_{\ell=1}^n  \frac{\p {\oldb}_{j\ell}}{\p x_\ell}.
\]
If $H\vp$ is the Hessian of $\vp$, then from (\ref{eqn:[X_j,X_j^*]}) it follows that
\begin{align*}
[\Ta_j,\Ta_j^*] &= \sum_{\ell,\ell'=1}^n \bigg( -{\oldb}_{j\ell}{\oldb}_{j\ell'} \frac{\p^2\vp}{\p x_\ell \p x_{\ell'}}
+ {\oldb}_{j\ell}\frac{\p {\oldb}_{j\ell'}}{\p x_\ell} X_{\ell'}^* + {\oldb}_{j\ell'} \frac{\p {\oldb}_{j\ell}}{\p x_{\ell'}} X_\ell - {\oldb}_{j\ell}\frac{\p^2 {\oldb}_{j\ell'}}{\p x_\ell \p x_{\ell'}}\bigg) \\
&= \sum_{\ell,\ell'=1}^n \bigg( -{\oldb}_{j\ell}{\oldb}_{j\ell'} \frac{\p^2\vp}{\p x_\ell \p x_{\ell'}}
- {\oldb}_{j\ell}\frac{\p {\oldb}_{j\ell'}}{\p x_\ell} \frac{\p\vp}{\p x_{\ell'}} - {\oldb}_{j\ell}\frac{\p^2 {\oldb}_{j\ell'}}{\p x_\ell \p x_{\ell'}} \bigg)\\
&= -\big(B (H\vp) B^T\big)_{jj} - \sum_{\ell,\ell'=1}^n\Big[  {\oldb}_{j\ell}\frac{\p {\oldb}_{j\ell'}}{\p x_\ell} \frac{\p\vp}{\p x_{\ell'}} + {\oldb}_{j\ell}\frac{\p^2 {\oldb}_{j\ell'}}{\p x_\ell \p x_{\ell'}}\Big].
\end{align*}

The key to the proof of Proposition \ref{prop:compactness of W^1_0 into L^2} was the construction of $\Psi$, an unbounded function so that
$(\Psi f, f)_{\vp}$ could be written in terms of inner products involving $\|\nabla_X f\|_{L^2(\Omega,\vp)}$ and $([X_j,X_j^*]f,f)_\vp$. Adapting the heuristic of
Proposition \ref{prop:compactness of W^1_0 into L^2}, we compute
\begin{multline*}
\sum_{j=1}^{n-1}\Big[ \big\|\big(T_j^* + T_j\big)f \big\|_{L^2(M,\vp)}^2 - (1+\ep)\big([T_j,T_j^*]f,f\big)_{M,\vp}\Big] \\
\hspace{-1.5in}= \sum_{j=1}^{n-1} \Bigg[ \Big\| -\big(B\nabla\vp\big)_j f - \sum_{\ell=1}^n \frac{\p {\oldb}_{j\ell'}}{\p x_\ell}f \Big\|_{L^2(M,\vp)}^2 \\
+ (1+\ep)\Big( \big(B(H\vp)B^T\big)_{jj}f + \sum_{\ell,\ell'=1}^n\Big[  {\oldb}_{j\ell}\frac{\p {\oldb}_{j\ell'}}{\p x_\ell} \frac{\p\vp}{\p x_{\ell'}} + {\oldb}_{j\ell}\frac{\p^2 {\oldb}_{j\ell'}}{\p x_\ell \p x_{\ell'}}\Big]  f, f \Big)_{M,\vp}\Bigg].
\end{multline*}
Therefore, the analog of $\Psi$ in Proposition \ref{prop:compactness of W^1_0 into L^2} is
\begin{multline}\label{eqn:Psi on M}
\Psi_M(x) = \sum_{j=1}^{n-1} \Bigg[\Big| \big( B\nabla\vp\big)_j + \sum_{\ell=1}^n \frac{\p {\oldb}_{j\ell}}{\p x_\ell}\Big|^2 \\
+ (1+\ep)\bigg([ \Tr\big(B(H\vp)B^T\big) + \sum_{\ell,\ell'=1}^n\Big[  {\oldb}_{j\ell}\frac{\p {\oldb}_{j\ell'}}{\p x_\ell} \frac{\p\vp}{\p x_{\ell'}} + {\oldb}_{j\ell}\frac{\p^2 {\oldb}_{j\ell'}}{\p x_\ell \p x_{\ell'}}\Big] \bigg) \Bigg].
\end{multline}
The matrix $B$ plays a critical role here. We observe that
\[
B \nabla = \begin{pmatrix} \tnabla_\oldL\vspace{.1in} \\  \p/\p\nu \end{pmatrix}
\]
where $\frac{\p}{\p\nu} = {\oldL}_n$ is the unit outward pointing normal.  Now, (HVI) and \eqref{eq:combined_hypothesis} tell us
\[
  \lim_{\atopp{|x|\rightarrow\infty}{x\in M}}|\tnabla_\oldL\varphi|=\infty,
\]
as well.  Using (HIV) and (HVI) to bound $H\varphi$, we have
\[
  \Psi_M(x)\geq|\tnabla_\oldL\varphi|^2-O(|\tnabla_\oldL\varphi|+1).
\]
Hence, we have the following analogue of (HII):
\begin{enumerate}
\item[BI.] There exists $\ep>0$ so that $\Psi_M$ defined by (\ref{eqn:Psi on M}) satisfies
\[
\lim_{\atopp{|x|\to\infty}{x\in M}} \Psi_M(x) = \infty.
\]
\end{enumerate}

\begin{prop}\label{prop:compactness of W^1 into L^2 on M}
Let $\Omega\subset\R^n$ satisfy \emph{(HI)}-\emph{(HV)} and $ \bd\Omega$ satisfy \emph{(BI)}. Then
the embedding  $W^{1,2}(M,\vp;T) \hookrightarrow L^2(M,\vp)$ is compact.
\end{prop}

\begin{proof}The proof follows the argument of Proposition \ref{prop:compactness of W^1_0 into L^2}.
\end{proof}
As earlier, we have the following corollary.

\begin{cor}\label{cor:adjoint dominated in L^2 on M}
Under the assumptions and  notation of  Proposition \ref{prop:compactness of W^1 into L^2 on M},
\[
\| \tnabla_{\Ta^*} f \|_{L^2(M,\vp)}^2
\leq C \| \tnabla_\Ta f\|_{L^2(M,\vp)}^2
\]
for some constant $C$ independent of $f$.
\end{cor}

A similar argument shows the following Rellich identity. Set
\begin{multline}\label{eqn:Theta on M}
\Theta_M = \sum_{j=1}^{n-1}\Bigg[ \Big| \big( B\nabla\vp\big)_\ell - \sum_{\ell=1}^n \frac{\p {\oldb}_{j\ell}}{\p x_\ell}\Big|^2 \\
- (1+\ep)\bigg( \Tr\big(B(H\vp)B^T\big) + \sum_{\ell,\ell'=1}^n\Big[ {\oldb}_{j\ell} \frac{\p {\oldb}_{j\ell'}}{\p x_\ell} \frac{\p \vp}{\p x_{\ell'}} +{\oldb}_{j\ell}\frac{\p^2 {\oldb}_{j\ell'}}{\p x_\ell \p x_{\ell'}}\Big] \bigg)\Bigg].
\end{multline}
As before, (HII)-(HIV) and (HVI) can be used to prove an analogue to (HIII):
\begin{enumerate}
\item[BII.] There exists $\ep>0$ so that $\Theta_M$ satisfies
\[
\lim_{\atopp{|x|\to\infty}{x\in M}} \Theta_M(x) = \infty.
\]
\end{enumerate}

\begin{prop}\label{prop:compactness of W^1 into L^2 on M, T^* version}
Let $\Omega\subset\R^n$ satisfy \emph{(HI)}-\emph{(HV)} and $ \bd\Omega$ satisfy \emph{(BII)}. Then
the embedding  $W^{1,2}(M,\vp;L) \hookrightarrow L^2(M,\vp)$ is compact.
\end{prop}

\begin{cor}\label{cor:adjoint dominated in L^2 on M, T^* version}
Under the assumptions and  notation of  Proposition \ref{prop:compactness of W^1_0 into L^2},
\[
\|\tnabla_\Ta f \|_{L^2(M,\vp)}^2
\leq C \|  f\|_{W^{1,2}(M,\vp;L)}^2
\]
for some constant $C$ independent of $f$.
\end{cor}

Our final comment on the consequences of  (HIV) and (HVI) is the following:
\begin{enumerate}
\item[BIII.] There exist constants $C_k$ so that
\[
|(\tnabla_\oldL)^k \vp| \leq C_k (1+|\tnabla_\oldL \vp|)
\]
for all $x\in M$ and $1\leq k \leq m$.
\end{enumerate}

Since we have shown that (BI)-(BIII) follow from (HI)-(HVI), we will suppress the individual boundary hypotheses and assume only (HVI) in the following.

\begin{cor}\label{cor:compactness of W^k into W^k-1}
Suppose that $\Omega\subset\R^n$ satisfies \emph{(HI)}-\emph{(HVI)}. Then if $0<k\leq m$,
\[
\| f \|_{W^{k}(M,\vp;T)} \sim \| f \|_{W^{k}(M,\vp;\oldL)}.
\]
\end{cor}

\begin{proof} The proof goes by induction. The $k=1$ case is the content of Corollary \ref{cor:adjoint dominated in L^2 on M} and
Corollary \ref{cor:adjoint dominated in L^2 on M, T^* version}. The higher $k$ follow from the $k=1$ case, the inductive hypothesis  and the fact that
$[T_j,T_j^*]$ is a function bounded by a multiple of $(1+|\nabla_\T\vp|)$.
\end{proof}

\begin{prop} \label{prop:interpolation of Sobolev norms (only integers) on M}
Let $\Omega\subset\R^n$ satisfy \emph{(HI)}-\emph{(HVI)}.
Then  there exists $K,K'>0$ depending on $n,m$ so that for any $\delta>0$,
$u\in W^{m,2}( M,\vp;T)$, and $0< j < m$,
\begin{align}
\sum_{|\alpha| =j} \| \Ta^\alpha u \|_{L^2( M,\vp)}^2 &\leq K \Big( \delta \sum_{|\beta| =m} \| \Ta^\beta u \|_{L^2( M,\vp)}^2 + \delta^{-j/(m-j)} \| u\|_{L^2( M,\vp)}\Big)
\label{eqn:interpolation of |alpha|=j, on M}\\
\| u\|_{W^{j,2}( M,\vp;T)} &\leq K' \big( \delta \| u\|_{W^{m,2}( M,\vp;T)} + \delta^{-j/(m-j)} \| u\|_{L^2( M,\vp;T)}\big) \label{eqn:interpolation of |alpha| |eq j, on M}\\
\| u\|_{W^{j,2}( M,\vp;T)} &\leq 2K' \| u\|_{W^{m,2}( M,\vp;T)}^{j/m} \| u\|_{L^2( M,\vp;T)}^{(m-j)/m} \label{eqn:j norm convex of m norm and 0 norm, on M}
\end{align}
\end{prop}

\begin{proof} Note that \eqref{eqn:interpolation of |alpha| |eq j, on M} follows
from repeated applications of \eqref{eqn:interpolation of |alpha|=j, on M}. Equation \eqref{eqn:j norm convex of m norm and 0 norm, on M}
follows from \eqref{eqn:interpolation of |alpha| |eq j, on M}
by choosing $\ep$ so that the two terms on the right-hand side are equal.

We first prove the result for
$m=2$, $j=1$. In this case,
\begin{multline*}
\| T_j u \|_{L^2(M,\vp)}^2 = ( T_j u, T_j u)_\vp = (T_j^* T_j u, u)_\vp \\
\leq \|T_j^* T_j u \|_{L^2(M,\vp)} \|u\|_{L^2(M,\vp)}
\leq \ep \|T_j^* T_j u \|_{L^2(M,\vp)}^2 + \frac1{4\ep} \|u\|_{L^2(M,\vp)}^2.
\end{multline*}
However, since $\lim_{\atopp{|x|\to\infty}{x\in M}} \big(\theta|\nabla\vp|^2 + \triangle \vp \big) = \infty$, it follows by Corollary
\ref{cor:adjoint dominated in L^2 on M} that
\[
\|T_j^* T_j u \|_{L^2(M,\vp)}^2 \leq C \| T_j u\|_{W^{1,2}(M,\vp;T)}^2.
\]
This proves the result for the case $m=2$, $j=1$. We can follow the argument of  \cite[Theorem 5.2]{AdFo03} to finish proof.
\end{proof}

\subsection{Approximation}
We can also prove a boundary version of the $L^2$ analog to \cite[Theorem 5.33]{AdFo03}, the Approximation Theorem for $\R^n$.

\begin{prop}\label{prop:approximation theorem on M}
Let $\Omega\subset\R^n$ satisfy \emph{(HI)}-\emph{(HVI)}.
There exists a constant $C = C(m,n)$ so that for $0<k\leq m$,
$v\in W^{k,2}( M,\vp;T)$,  and $0<\delta\leq 1$, there exists $v_\delta\in C^m( M)$ so that:
\[
\| v - v_\delta \|_{L^2( M,\vp)} \leq C\delta^k \sum_{|\alpha|=k} \| T^\alpha u \|_{L^2( M,\vp)}
\]
and
\[
\| v_\delta \|_{W^{j,2}( M,\vp;T)}
\leq C \begin{cases} \| v \|_{W^{k,2}( M,\vp;T)} & \text{if }j\leq k-1 \\ \delta^{k-j} \| v \|_{W^{k,2}( M,\vp;T)} &\text{if }k\leq j \leq m.
\end{cases}
\]
\end{prop}
Proposition \ref{prop:approximation theorem on M}  means that $ M$ has the approximation property.

\begin{proof} In this proof, we work locally and use the boundary  operators
$Y^b_j = {\oldL}_j - \frac12 {\oldL}_j(\vp)$. It follows from Corollary \ref{cor:compactness of W^k into W^k-1} that
\[
\|ve^{-\frac 12\vp} \|_{W^{k,2}(M)} = \sum_{|\alpha|\leq k} \| (Y^b)^\alpha v \|_{L^2(M,\vp)}
\sim \| v \|_{W^{k,2}(M,\vp;T)}.
\]
for $0\leq k \leq m$.  Then
\[
\|ve^{-\frac 12\vp} \|_{W^{k,2}(M)} \sim \sum_{j=1}^\infty\|v_je^{-\frac 12\vp} \|_{W^{k,2}(M)}
\]
where $v_j = \chi_{U_j} v$ and $\{\chi_{U_j}\}$ is a partition of unity subordinate to $\{U_j\}$.
By the classical theory, there exists $\psi_{\delta,j} \in C^{m+1}_c(M\cap U_j)$ so that
\[
\| v_je^{-\frac 12\vp} - \psi_{\delta,j}e^{-\frac 12\vp} \|_{L^2( M,\vp)} \leq C\delta^k \sum_{|\alpha|=k} \| \oldL^\alpha ( v_j e^{-\frac 12\vp} )\|_{L^2( M,\vp)}
\]
and
\[
 \| \psi_{\delta,j}e^{-\frac 12\vp} \|_{W^{\ell,2}(M)}
\leq C \begin{cases} \| v_j e^{-\frac 12\vp} \|_{W^{k,2}( M\cap U_j)} & \text{if }\ell\leq k-1 \\ \delta^{k-j} \| \psi \|_{W^{k,2}(M\cap U_j,\vp;X)} &\text{if }k\leq \ell \leq m.
\end{cases}
\]
Since the $M\cap U_j$ are of comparable surface area, the constant $C$ arising from the classical Approximation Theorem can be taken independent of $j$. Thus, the result follows by
summing in $j$ and observing that the decomposition $v = \sum_{j=1}^\infty v_j$ is locally finite.
\end{proof}

%
%
\section{Weighted Besov spaces on $\Omega$ and $M$}\label{sec:weighted Besov spaces on Omega}
We start with the following proposition. We initially prove a boundary version because we need to strengthen Proposition \ref{prop:density of C^infty_c} before we can prove an analog for $\Omega$.
This is an $L^2$ adaptation of the Approximation Theorem, \cite[Theorem 7.31]{AdFo03}.

\begin{prop}\label{prop:W^k in H on M}
Let $\Omega\subset\R^n$ satisfy \emph{(HI)}-\emph{(HVI)}. If $0<k<m$, then
\[
W^{k,2}(M,\vp;T) \in \H\big(k/m; L^2(M,\vp),W^{m,2}(M,\vp;T)\big)
\]
\end{prop}

\begin{proof} The argument is the same as \cite[Theorem 7.31]{AdFo03} with Proposition \ref{prop:interpolation of Sobolev norms (only integers) on M}
filling in for \cite[Theorem 5.2]{AdFo03} and Proposition \ref{prop:approximation theorem on M} with $W^{m,2}(M,\vp;T)$ replacing the Approximation Theorem
in \cite{AdFo03}.
\end{proof}
The importance of Proposition \ref{prop:approximation theorem on M} is that the Reiteration Theorem (see Theorem
\ref{thm:reiteration theorem}) holds for interpolation spaces generated from
the weighted $L^2$-Sobolev spaces.

\subsection{Real interpolation of boundary Sobolev spaces}
\label{sec:real_interpolation}
We are now ready to define our weighted Besov spaces.
\begin{defn}\label{defn:Besov spaces on M}
Let $0<s<\infty$, $1\leq p < \infty$, $1\leq q \leq \infty$ and $m$ be the smallest integer larger than $s$. We define the Besov space $B^{s;p,q}(M,\vp;T)$
to be the intermediate spaces between $L^p(M,\vp)$ and $W^{m,p}(M,\vp;T)$ corresponding to $\theta = s/m$, i.e.,
\[
B^{s;p,q}(M,\vp;T) = \big( L^p(M),W^{m,p}(M,\vp;T)\big)_{s/m,q;J}.
\]
We define the Besov space $B^{s;p,q}(M,\vp;\oldL)$
to be the intermediate spaces between $L^p(M,\vp)$ and $W^{m,p}(M,\vp;\oldL)$ corresponding to $\theta = s/m$, i.e.,
\[
B^{s;p,q}(M,\vp;\oldL) = \big( L^p(M,\vp),W^{m,p}(M,\vp;\oldL)\big)_{s/m,q;J}.
\]
\end{defn}
We will focus on the case $p=2$ since we only proved an $L^2$ Approximation Theorem. By Theorem \ref{thm:J method produces intermediate spaces},
$B^{s;2,q}(M,\vp;T)$ is a Banach space with interpolation norm
\[
\| u \|_{B^{s;2,q}(M,\vp;T)} = \big \| u ; \big(L^2(M,\vp), W^{m,2}(M,\vp;T)\big)_{s/m,q;J} \big\|.
\]
Also, $B^{s;2,q}(M,\vp;T)$ inherits density and approximation properties from $W^{m,2}(M,\vp;T)$. For example,
$\{ \psi\in C^\infty(M): \| \psi \|_{W^{m,2}(M,\vp;T)}<\infty\}$ is dense in $B^{s;2,q}(M,\vp;T)$.

Let $\Omega$ satisfy \emph{(HI)}-\emph{(HVI)}. Proposition \ref{prop:W^k in H on M} and the Reiteration Theorem  imply that
if $0\leq k < s < m$ and $s=(1-\theta)k + \theta m$, then
\[
B^{s;2,q}(M,\vp;T) = \big(W^{k,2}(M,\vp;T), W^{m,2}(M,\vp;T)\big)_{\theta,q;J}.
\]
More generally, if $0\leq k < s < m$ and
$s = (1-\theta)s_1+\theta s_2$ and $1 \leq q_1,q_2\leq\infty$, then
\begin{equation}\label{eqn:interpolation, q's change on M}
B^{s;2,q}(M,\vp;T) = \big( B^{s_1;2,q_1}(M,\vp;T), B^{s_2;2,q_2}(M,\vp;T)\big)_{\theta,q;J}.
\end{equation}

The following corollary is an immediate consequence of Proposition \ref{prop:W^k in H on M} and Lemma \ref{lem:embedding of J,K,H intermediate spaces}.
\begin{cor}\label{cor:sobo spaces are intermediate spaces}
\[
B^{m;2,1}(M,\vp;T) \hookrightarrow W^{m,2}(M,\vp;T) \hookrightarrow B^{m;2,\infty}(M,\vp;T).
\]
\end{cor}

\subsection{Proof of Lemma \ref{lem:traces exist for u in sobo spaces on M}}
\label{sec:proof_traces exist for u in sobo spaces on M}
\begin{proof}[Proof of Lemma \ref{lem:traces exist for u in sobo spaces on M}] We follow the outline of \cite[Lemma 7.40]{AdFo03}.
We may apply the Reiteration Theorem
to obtain
\[
B := B^{k-\frac 12;2,2}(M,\vp;T) = \big( W^{k-1,2}(M,\vp;T), W^{k,2}(M,\vp;T)\big)_{\theta,2;J}
\]
where $\theta = 1 - \frac 12 = \frac 12$.
From Theorem \ref{thm:discrete version J method}, we can apply the discrete version of the $J$-method and obtain that
$u\in B$ if and only if there exist $u_i \in  W^{k-1,2}(M,\vp;T) \cap W^{k,2}(M,\vp;T) = W^{k,2}(M,\vp;T)$ for
$i\in\Z$ so that
\[
\sum_{i\in\Z} u_i = u
\]
in $W^{k-1,2}(M,\vp;T) + W^{k,2}(M,\vp;T) = W^{k-1,2}(M,\vp;T)$ and such that
\[
\big\{ 2^{-i/2} \| u_i \|_{W^{k-1,2}(M,\vp;T)} \big\},\ \big\{ 2^{i/2} \| u_i \|_{W^{k,2}(M,\vp;T)} \big\} \in \ell^2.
\]

Let $\tilde\pi:\Omega_\ep'\to M$ be the map that sends $x\in \Omega_\ep'$ to the unique point $\tilde\pi(x)\in M$ obtained by flowing along ${\oldL}_n$. That is,  there exists $t = t_x$ such that
$x = e^{t{\oldL}_n}(\tilde\pi(x))$. The constant $\ep>0$ is small enough so that each point $x\in\Omega_\ep'$ can be uniquely represented by $x =(\tilde\pi(x),t_x)$. In this way, if
$U\in C^\infty_c(\Omega_\ep')$ and $x\in\Omega_\ep$, then
\[
U(x) = \int_{-\infty}^{t_x} \frac{d}{dt} U\big( e^{t{\oldL}_n}(\tilde\pi(x))\big)\, dt = \int_{-\infty}^{t_x} {\oldL}_n U\big( e^{t{\oldL}_n}(\tilde\pi(x))\big)\, dt.
\]

Let $\tilde\psi \in C^\infty_c(\R)$ be so that
\begin{enumerate}\renewcommand{\labelenumi}{(\roman{enumi})}
\item $\tilde\psi(t) =1$ on $[-1,1]$,
\item $\tilde\psi(t) = 0$ if $|t|\geq 2$,
\item $0 \leq \tilde\psi(t) \leq 1$ for all $t\in\R$,
\item and there exists $c_j\geq 0$ so that $|\tilde\psi^{(j)}(t)| \leq c_j$ for all $j\geq 1$ and $t\in\R$.
\end{enumerate}
Define $\tilde\psi_i(t) = \tilde\psi(t/2^i)$ and $\psi_i = \tilde\psi_{i+1}-\tilde\psi_i$. Then $\psi_i$ vanishes outside $(2^i,2^{i+2})\cup (-2^{i+2},-2^i)$ (and at the endpoints
in particular). Also, $\| \psi_i \|_{L^\infty(\R)} =1$ and $\| \psi_i' \|_{L^\infty(\R)} \leq 2^{-i}c_1$.

Let $U\in C^\infty_c(\Omega_\ep')$.
Define $U_i(x)$ by
\begin{multline*}
U_i(x) = U_i(\tilde\pi(x),t_x) = e^{\frac{\vp(x)}2}\int_{-\infty}^{t_x} \psi_i(t) {\oldL}_n\big( U e^{-\frac {\vp}2}\big)\Big|_{e^{t{\oldL}_n}(\tilde\pi(x))} \, dt \\
= e^{\frac{\vp(x)}2}\int_{-\infty}^{t_x} \psi_i(t) \frac{d}{dt} \big( U e^{-\frac {\vp}2}\big)\Big|_{e^{t{\oldL}_n}(\tilde\pi(x))} \, dt.
\end{multline*}
Next, for $x\in M$, define $u_i$ by $u_i(x) = U_i(x)$. Then
\[
u_i(x)  = e^{\frac {\vp(x)}2}\int_{-\infty}^{0} \psi_i(t) \frac{d}{dt} \big( U e^{-\frac {\vp}2}\big)\Big|_{e^{t{\oldL}_n}(\tilde\pi(x))} \, dt.
\]
By the support condition on $\psi_i$ and the Fundamental Theorem of Calculus,
\begin{equation}\label{eqn:u_i defn in integer trace lemma on M}
u_i(x) e^{-\frac {\vp(x)}2} = \int_{-2^{i+2}}^{-2^i} \psi_i(t)  \frac{d}{dt} \big( U e^{-\frac {\vp}2}\big)\Big|_{e^{t{\oldL}_n}(\tilde\pi(x))} dt
= -\int_{-2^{i+2}}^{-2^i} \psi_i'(t)   \big( U e^{-\frac {\vp}2}\big)\Big|_{e^{t{\oldL}_n}(\tilde\pi(x))}dt.
\end{equation}
Since $U$ has compact support, $U_i$ and consequently $u_i$ vanish for all $i\in\Z$ when $|x|$ is sufficiently large. Therefore, the support of $\Tr U$
is a compact set on which $\sum_{i\in\Z}u_i$ converges uniformly to $u = \Tr U$.  Also, if $|\alpha|\leq k-1$, then
\[
\oldL^\alpha_\T \big(u_i(x) e^{-\frac {\vp(x)}2}\big) = \int_{-2^{i+2}}^{-2^i} \psi_i(t) \oldL^\alpha_\T\Big|_x {\oldL}_n\big( U e^{-\frac {\vp}2}\big)\Big|_{e^{t{\oldL}_n}(x)}\, dt
\]
Recall that
\[
{\oldL}_j\Big|_x = \sum_{\ell=1}^n {\oldb}_{j\ell}(x) \frac{\p}{\p x_j} .
\]
On $\Omega_\ep'$, $\rho$ is bounded in the $C^{k+1}$ norm, so ${\oldb}_{j\ell}$ is bounded in the $C^k$ norm. Consequently, by Cauchy-Schwarz,
\[
\big|{\oldL}_\T^\alpha \big(u_i(x) e^{-\frac {\vp(x)}2}\big)\big|
\leq  (2^{i+2})^{1/2}C_{\|({\oldb}_{j\ell})\|_{C^k(\Omega_\ep)}} \bigg( \int_{-2^{i+2}}^{-2^i} \Big|\nabla^{|\alpha|+1} \big( Ue^{-\frac {\vp}2}\big) \big|_{e^{t{\oldL}_n(x)}} \Big|^2 \, dt\bigg)^{1/2}.
\]
For a fixed $t>0$, $t \in [2^i, 2^{i+2})$ for exactly two (adjacent) $i$. Set $Y_j f= e^{\frac{\vp}2} {\oldL}_j(f e^{-\frac{\vp}2}) = \frac12(T_j + {\oldL}_j)$.  By
Corollary \ref{cor:compactness of W^k into W^k-1}, $\| f \|_{W^{j,2}(M,\vp;T)} \sim \sum_{|\alpha|\leq j}\|Y_\T^\alpha f\|_{M,\vp}$. Moreover, the paths
$e^{t{\oldL}_n}$ foliate $\Omega_\ep$, so multiplying by $2^{-i/2}$, squaring, summing over $i$, and integrating yields
\begin{align*}
\sum_{i\in\Z} 2^{-i} &\| u_i \|_{W^{k-1,2}(M,\vp;T)}^2
\leq C\sum_{\atopp{i\in\Z}{|\alpha|\leq k-1}}  2^{-i} \| Y_\T^\alpha u_i  \|_{L^2(M,\vp)}^2
= C \sum_{\atopp{i\in\Z}{|\alpha|\leq k-1}} 2^{-i} \|{\oldL}_\T^\alpha \big(u_i e^{-\frac{\vp(x)}2}\big) \|_{L^2(M)}^2 \\&
=  C\sum_{|\alpha|\leq k-1} \int_{\Omega_\ep} \big| \nabla^{|\alpha|+1} \big(U(x)e^{-\frac{\vp(x)}2}\big) \big|^2\, dx
\leq C \| U \|_{W^{k,2}(\Omega_\ep,\vp;X)}
\end{align*}
where the last inequality follows from (\ref{eqn:Y,X,D equiv}).

Using the second equality in \eqref{eqn:u_i defn in integer trace lemma on M} and Cauchy-Schwarz, we have
\begin{align*}
\big| {\oldL}_\T^\alpha \big(u_i(x) e^{-\frac{\vp(x)}2}\big) \big|
&\leq 2^{-i} 2^{(i+2)/2} C \bigg( \int_{-2^{i+2}}^{-2^i} \big| {\oldL}_\T^\alpha\Big|_x \big(U\big(e^{t{\oldL}_n}(x)\big) e^{-\frac{\vp(e^{t{\oldL}_n}(x))}2}\big)\big|^2\, dt\bigg)^{1/2} \\
&\leq 2^{-i} 2^{(i+2)/2} C_{\|({\oldb}_{jk})\|_{C^k(\Omega_\ep')}} \bigg( \int_{-2^{i+2}}^{-2^i} \Big|\nabla^{|\alpha|} \big( Ue^{-\frac {\vp}2}\big) \big|_{e^{t{\oldL}_n(x)}} \Big|^2 \, dt\bigg)^{1/2}.
\end{align*}
Therefore,
\begin{align*}
\sum_{i\in\Z} 2^{i} \| u_i \|_{W^{k,2}(M,\vp;T)}^2
&\leq C\sum_{\atopp{i\in\Z}{|\alpha|\leq k}}  2^{i} \| Y_\T^\alpha u_i  \|_{L^2(M,\vp)}^2
= C \sum_{i\in\Z} 2^i \|{\oldL}_\T^\alpha \big(u_i e^{-\frac{\vp(x)}2}\big) \|_{L^2(M,\vp)}^2 \\
&= C_{\|({\oldb}_{jk})\|_{C^{k+1}(\Omega_\ep')}}\sum_{|\alpha|\leq k} \int_{\Omega_\ep} \big| \nabla^{|\alpha|}  \big( U(x)e^{-\frac {\vp(x)}2}\big)\big|^2 \, dx
\leq C \| U \|_{W^{k,2}(\Omega_\ep,\vp;X)},
\end{align*}
where the final inequality follows from (\ref{eqn:Y,X,D equiv}).

Together, these inequalities show that $\| u\|_{B^{k-1/2;2,2}(M,\vp;T)} \leq C \| U \|_{W^{m,2}(\Omega_\ep',\vp;X)}$ when
$U\in C^\infty_c(\Omega_\ep')$. Since $C^\infty_c(\Omega_\ep')$ is dense in $W^{k,2}_0(\Omega_\ep',\vp;X)$, the proof is complete.
\end{proof}

\subsection{Proof of Theorem \ref{thm:traces of normal derivatives}}
\label{sec:proof_traces of normal derivatives}
\begin{proof}[Proof of Theorem \ref{thm:traces of normal derivatives}] Let
$\ell'' = \ell+\ell'+1$. Set $B = B^{\ell + \frac 12;2,2}(M,\vp;T)$. By definition,
\[
B = \Big( L^2(M,\vp), W^{\ell'',2}(M,\vp;T)\Big)_\theta,\quad
\text{where } \theta = \frac{\ell+\frac 12}{\ell''}.
\]
From the discrete $J$-method (Theorem \ref{thm:discrete version J method}), $u\in L^2(M,\vp)$ belongs
to $B$ if and only if there exist $\{u_j\}_{j\in\Z}\subset W^{\ell'',2}(M,\vp;T)$ so that
$u = \sum_{j\in\Z}u_j$ where the sum converges in $L^2(M,\vp)$ and
$\{2^{-j\theta}J(2^j;u_j)\}\in\ell^2$. The latter condition means that there exists $K>0$ so that
\[
\sum_{j\in\Z} 2^{-\frac{2\ell+1}{\ell''}} \|u_j\|_{L^2(M,\vp)}^2 \leq K^2 \| u\|_B^2
\quad\text{and}\quad
\sum_{j\in\Z} 2^{\frac{2\ell'+1}{\ell''}} \| u_j\|_{W^{\ell'',2}(M,\vp;T)}^2 \leq K^2 \| u\|_B^2.
\]

Let $\tilde\psi\in C^\infty_c(\R)$ be the bump function from the proof of Lemma \ref{lem:traces exist for u in sobo spaces on M}.
Set
\[
\psi_j(t) = \tilde\psi(t/\delta^j)
\]
for $j\in\Z$ and $\delta>0$ to be decided later. Set $\eta(t) = \tilde\psi(2t/\ep)$. It follows that
$|\psi^{(k)}_j(t)| \leq c_k \delta^{-jk}$. Also, for $k\geq 1$,
\[
\supp \psi_j^{(k)} \subset [-2\delta^j,-\delta^j]\cup [\delta^j,2\delta^j].
\]

For $y\in\Omega_\ep'$, there exists a unique $x\in M$ and $t\in[-\ep,\ep]$ so that $y = e^{t{\oldL}_n}(x)$.
Set $\tilde\pi(y)=x$. Since $\|\rho\|_{C^{m}(\Omega_\ep')}<\infty$, it follows that the projection
$\|\tilde\pi\|_{C^{m-1}(\Omega_\ep')} <\infty$. Set
\[
U_j(y) = \frac{1}{\ell'!}e^{\frac 12\vp(y)} \eta\big(\rho(y)\big) \psi_j\big(\rho(y)) \big(\rho(y)\big)^{\ell'}
u_j\big(\tilde\pi(y)\big) e^{-\frac 12\vp(\tilde\pi(y))}.
\]
Since $\frac{\p}{\p \nu} = {\oldL}_n$, it is immediate that
\[
\Tr U_j = \Tr \frac{\p U_j}{\p\nu} = \cdots = \Tr \frac{\p^{\ell'-1}U_j}{\p\nu^{\ell'-1}} =0
\]
and
\[
\Tr \frac{\p^{\ell'} U_j}{\p\nu^{\ell'}} = u_j
\]
for all $j\in\Z$.

Thus, we only need to show that $U\in W^{\ell''}(\Omega_\ep',\vp;X)$ and is supported in $\Omega_\ep'$.
Since $\supp U_j\subset \Omega_\ep'$ for all $j$, it follows that $U$ is supported in
$\Omega_\ep'$. Note that if $f$ is a smooth function on $M$, then there exist functions
$c_{\alpha_1,\alpha_2}$, $1\leq \alpha_1,\alpha_2 \leq n-1$, that are bounded in $C^{m-2}(\Omega_\ep')$ so that
\[
{\oldL}_{\alpha_1}(f\circ \tilde\pi)(y) = \sum_{\alpha_2}^{n-1}c_{\alpha_1,\alpha_2} {\oldL}_{\alpha_2}f\big(\tilde\pi(y)).
\]
Also, by construction, ${\oldL}_n (f\circ\tilde\pi)(y)=0$. Since $U_j$ has  support in $\Omega_\ep'$,
(\ref{eqn:Y,X,D equiv}) shows that we may use the $Y_k$ operators (instead of the $X_k$'s) for differentiation.
Let $\gamma = (\gamma_1,\dots,\gamma_{\ell''})$ be a multiindex of length $\ell''$.
Set
\[
\gamma_T = \big|\{\alpha\in\gamma: \gamma_\alpha \text{ is tangential}\}\big|
\text{ and }
\gamma_N = \big|\{\alpha\in\gamma: \gamma_\alpha=n\}\big|.
\]
Set
\[
f_j(t) = \eta(t)\psi_j(t) t^{\ell'}
\quad\text{and}\quad
g_j(x) = u_j(x) e^{-\frac 12\vp(x)}
\]
for $x\in M$.
Observe that
\[
|f_j^{(\alpha_1)}(t)| \leq c_{\alpha_1} \delta^{j(\ell'-\alpha_1)}.
\]
The function $\eta$ does not affect the estimates -- derivatives of $\eta$ are supported where $|s|\in [\ep,2\ep]$
and the support of derivatives of $\eta$ and derivatives of $\psi_j$ cause $\eta\sim \delta^j$ (or else the
particular combination $\psi_j' \eta'$ is identically zero). By (\ref{eqn:Y,X,D equiv}) and the fact that $U$ is  supported in $\Omega_\ep'$, it is enough
to bound $\sum_{j\in\Z} \| Y^\gamma U_j(y)\|_{L^2(\Omega_\ep',\vp)}$ to show that
$U\in W^{\ell'',2}(\Omega_\ep',\vp;X)$.

Since $\rho$ is bounded in the $C^m$ norm and $c_{\alpha_1,\alpha_2}$ are bounded in the $C^{m-2}$ norm, there exists functions
$\sigma^\gamma_{\alpha,\beta}$ that are bounded on $\Omega_\ep'$ so that
\begin{align*}
Y^\gamma\big(U_j(y)\big) e^{-\frac 12\vp(y)}
= \oldL^\gamma\big(U_j(y)e^{-\frac12 \vp(y)}\big)
&= \oldL^\gamma\Big( f_j\big(\rho(y)\big) g_j\big(\tilde\pi(y)\big)\Big)  \\
&= \sum_{\alpha=0}^{\gamma_N} \sum_{|\beta|\leq \gamma_T} \sigma^\gamma_{\alpha,\beta}
f_j^{(\alpha)}\big(\rho(y)\big) \oldL^\beta g_j\big(\tilde\pi(y)\big).
\end{align*}
Thus,
\begin{align*}
\| Y^\gamma U_j(y)\|_{L^2(\Omega_\ep',\vp)}
&\leq C \sum_{\alpha=1}^{\gamma_N} \sum_{|\beta|\leq \gamma_T}
\int_{\Omega_\ep'} \big| f^{(\alpha)}_j\big(\rho(y)\big) \big|^2 \big| Y^\beta u_j\big(\tilde\pi(y)\big) \big|^2
e^{-\vp(\tilde\pi(y))}\, dy \\
&= C \sum_{\alpha=1}^{\gamma_N} \sum_{|\beta|\leq \gamma_T} \int_{\bd\Omega} \int_{-2\delta^j}^{2\delta^j}
\big| f^{(\alpha)}_j(t) \big|^2 \big| Y^\beta u_j(x) \big|^2 e^{-\vp(x)}\, dt\, d\sigma(x) \\
&\leq C \sum_{\alpha=1}^{\gamma_N} \sum_{|\beta|\leq \gamma_T} \int_{\bd\Omega}
\delta^{j(2\ell'-2\alpha+1)} \big| Y^\beta u_j(x)\big|^2 e^{-\vp(x)}\, d\sigma(x) \\
&\leq C \sum_{\alpha=1}^{\gamma_N}\delta^{j(2\ell'-2\alpha+1)}  \| u_j \|_{W^{\gamma_T,2}(\bd\Omega,\vp;T)}\\
&\leq C \big( \delta^{j(2\ell'+1)} + \delta^{j(2\ell'-2\gamma_N+1)}\big)
\| u_j \|_{W^{\gamma_T,2}(\bd\Omega,\vp;T)}.
\end{align*}
where $C$ is independent of $j$. Set $\delta = 2^{\frac{1}{\ell''}}$. This means
\begin{align}
\| Y^\gamma U(y)\|_{L^2(\Omega_\ep',\vp)}
&\leq \sum_{j\in\Z} \| Y^\gamma U(y)\|_{L^2(\Omega_\ep',\vp)} \nn\\
&\leq C \sum_{j\in\Z}\big( \delta^{j(2\ell'+1)} + \delta^{j(2\ell'-2\gamma_N+1)}\big)
\| u_j \|_{W^{\gamma_T,2}(\bd\Omega,\vp;T)}. \label{eqn:Y^gamma U in terms of u_j}
\end{align}
To check that the sum on the right hand side of \eqref{eqn:Y^gamma U in terms of u_j} is finite, observe that
\begin{align*}
\sum_{j\in\Z}\delta^{j(2\ell'+1)} \| u_j \|_{W^{\gamma_T,2}(\bd\Omega,\vp;T)}^2
&= \sum_{j\in\Z}2^{j\frac{2\ell'+1}{\ell''}} \| u_j \|_{W^{\gamma_T,2}(\bd\Omega,\vp;T)}^2\\
&\leq \sum_{j\in\Z}2^{j\frac{2\ell'+1}{\ell''}} \| u_j \|_{W^{\ell'',2}(\bd\Omega,\vp;T)}^2
\leq K \|u\|_B^2.
\end{align*}
To bound the remaining term in \eqref{eqn:Y^gamma U in terms of u_j}, we use
(\ref{eqn:interpolation of |alpha| |eq j, on M}) to bound
\[
\delta^{j(2\ell'-2\gamma_N+1)}\| u_j \|_{W^{\gamma_T,2}(\bd\Omega,\vp;T)}^2
\leq K' \delta^{j(2\ell'-2\gamma_N+1)}\big( \ep^2 \| u_j \|_{W^{\ell'',2}(\bd\Omega,\vp;T)}^2
+ \ep^{-2\frac{\gamma_T}{\ell''-\gamma_T}} \|u_j\|_{L^2(\bd\Omega,\vp)}^2\big).
\]
We require $\delta^{j(2\ell'-2\gamma_N+1)}\ep^2 = \delta^{j(2\ell'+1)}$. This means $\ep^2 = \delta^{2\gamma_Nj}$.
Since $\gamma_T + \gamma_N = \ell''$, it follows that
$\frac{\gamma_T}{\ell''-\gamma_T} = \frac{\ell''-\gamma_N}{\gamma_N}$ and
\[
\delta^{j(2\ell'-2\gamma_N+1)}\delta^{-2\gamma_N j \frac{\ell''-\gamma_N}{\gamma_N}}
= \delta^{-j(2\ell+1)}
\]
since $\ell+\ell'+1=\ell''$. Thus, since $\delta = 2^{\frac{1}{\ell''}}$,
\[
\delta^{j(2\ell'-2\gamma_N+1)}\| u_j \|_{W^{\gamma_T,2}(\bd\Omega,\vp;T)}^2
\leq K'\sum_{j\in\Z}\big( 2^{{\frac{2\ell'-1}{\ell'}}} \|u_j\|_{W^{\ell'',2}(\bd\Omega,\vp;T)}
+ 2^{-{\frac{2\ell-1}{\ell'}}} \|u_j\|_{L^2(\bd\Omega,\vp)}\big).
\]
Thus, $U\in W^{\ell'',2}(\Omega,\vp; X)$ and the proof is complete.
\end{proof}
The proof of Theorem \ref{thm:Trace Theorem for integer m on M} is now complete.
\begin{remark}\label{rem:certain normal derivs of extension are 0 on M}
If (for example) $\ell'=0$, then the formula for $U_j$ is
\[
U_j(y) =  e^{\frac 12\vp(y)} \eta\big(\rho(y)\big) \psi_j\big(\rho(y)\big) u_j\big(\tilde\pi(y)\big) e^{-\frac 12\vp(\tilde\pi(y))}
\]
Since ${\oldL}_n(\tilde\pi(y))=0$, we can compute
\[
(Y_n U_j(y)) e^{-\frac 12\vp(y)} = {\oldL}_n\big(U_j(y) e^{-\frac 12\vp(y)}\big)
= {\oldL}_n\rho(y) (\eta\psi_j)'\big(\rho(y)\big) u_j\big(\tilde\pi(y)\big) e^{-\frac 12\vp(\tilde\pi(y))}
\]
It follows from the support conditions on $\eta$ and $\psi$ that $\supp (\eta\psi_j)' \cap [(-\delta^j,\delta^j)\cap(-\ep,\ep)] = \emptyset$. Therefore
\[
\Tr (Y_n U_j)=0
\]
for all $j$. Similarly, $\Tr(Y_n^k U_j)=0$ for all $k$ for which $Y_n^k$ is defined. A similar result also holds if $\ell'>0$ since $\Tr\rho=0$.
\end{remark}

Now that Theorem \ref{thm:traces of normal derivatives} is proven, we can apply our trace and extension theorems to prove that a simple $(k,2)$-extension operator exists.
\begin{proof}[Proof of Theorem \ref{thm:simple extension operators exist}]
Let $1\leq k \leq m-1$ and $f\in W^{k,2}(\Omega,\vp;X)$.
We begin by constructing functions $u_1,\dots, u_k$ recursively. By Theorem \ref{thm:Trace Theorem for integer m on M},
$\Tr f \in B^{k-\frac 12;2,2}(M,\vp;T)$ and
\[
\| \Tr f \|_{B^{k-\frac 12;2,2}(M,\vp;T)} \leq K \| f\|_{W^{k,2}(\Omega,\vp; X)}.
\]
Next, by Theorem \ref{thm:traces of normal derivatives}, there exists a function $u_1\in W^{k,2}(\Omega_\ep',\vp;X)$
supported in $\Omega_\ep'$ and so that
\[
\Tr u_1 = \Tr f
\quad\text{and}\quad
\| u_1\|_{W^{k,2}(\Omega,\vp; X)} \leq K \| \Tr f \|_{B^{k-\frac 12;2,2}(M,\vp;T)}.
\]
Thus,
\[
\| u_1\|_{W^{k,2}(\Omega,\vp; X)} \leq K \| f\|_{W^{k,2}(\Omega,\vp; X)}.
\]
It now follows that $(f-u_1)\in W^{k,2}(\Omega,\vp; X)\cap W^1_0(\Omega,\vp; X)$.

Next, $T_n(f-u_1)\in W^{k-1,2}(\Omega,\vp; X)$ so by Theorem \ref{thm:Trace Theorem for integer m on M},
$\Tr(T_n(f-u_1)) \in B^{k-3/2;2,2}(M,\vp;T)$ and
\[
\| \Tr(T_n(f-u_1))\|_{B^{k-\frac 12;2,2}(M,\vp;T)} \leq K \| T_n(f-u_1) \|_{W^{k-1,2}(\Omega,\vp; X)}.
\]
By Theorem \ref{thm:traces of normal derivatives}, there exists a function $u_2\in W^{k,2}(\Omega_\ep',\vp; X)$
supported in $\Omega_\ep'$ and so that $\Tr u_2=0$
and $\Tr(\frac{\p u_2}{\p \nu}) = \Tr(T_n u_2) = \Tr( T_n(f-u_1))$ and
\[
\| u_2 \|_{W^{k,2}(\Omega,\vp; X)} \leq K \| \Tr(T_n(f-u_1))\|_{B^{k-\frac 32;2,2}(M,\vp;T)}.
\]
Thus, $(f-u_1-u_2)\in W^{k,2}(\Omega,\vp; X)\cap W^2_0(\Omega,\vp; X)$ and
\[
\| f - u_1-u_2 \|_{W^{k,2}(\Omega,\vp; X)} \leq K \| f \|_{W^{k,2}(\Omega,\vp; X)}.
\]

Iterating this process, we can show that there exist functions
\[
  u_1,\dots,u_k\in W^{k,2}(\Omega_\ep',\vp;X)
\]
supported in $\Omega_\ep'$ and so that
$(f-u_1-\cdots -u_j)\in W^{k,2}(\Omega,\vp; X)\cap W^j_0(\Omega,\vp; X)$ and
\[
\| f - u_1-\cdots -u_j \|_{W^{k,2}(\Omega,\vp; X)} \leq K \| f \|_{W^{k,2}(\Omega,\vp; X)}.
\]
Thus, $f-(u_1+\cdots+u_k) \in W^{k,2}_0(\Omega,\vp; X)$.

Next, we can write
\[
f = (u_1+\cdots+u_k) + (f-u_1-\cdots-u_k).
\]
The function $(f-u_1-\cdots-u_k)$ can be extended by 0 to produce a function in $W^{k,2}(\R^n,\vp;X)$ and $(u_1+\cdots+u_k)\in W^{k,2}(\R^n,\vp;X)$. Thus, we define
\[
Ef(x) = \begin{cases} f(x) & x\in \bar\Omega \\ u_1(x)+\cdots +u_k(x) & x\in \bar\Omega^c. \end{cases}
\]
\end{proof}

\subsection{Approximation of functions in $W^{k,2}(\Omega,\vp; X)$}
Using the Trace Theorem \ref{thm:traces of normal derivatives}, we are now in a position to improve Proposition \ref{prop:density of C^infty_c} for $p=2$ and relax the condition that
$f\in W^{\ell,2}_0(\Omega,\vp; X)$.
\begin{prop}\label{eqn: density of C^infty_c -- improved}
Assume that $\bd\Omega$ satisfies (HI)-(HVI) for some $m\geq 2$. Let $1\leq \ell \leq m-1$ be an integer.
Then $C^\infty_c(\R^n)$ is dense $W^{\ell,2}(\Omega,\vp; X)$  in the sense that
if $\ep>0$, then there exist $\psi\in C^\infty_c(\R^n)$ so that
\[
\| \psi - f \|_{W^{\ell,2}(\Omega,\vp; X)} \leq \ep
\]
and
\[
\| \psi\|_{W^{\ell,2}(\R^n,\vp;X)} \leq C \| f\|_{W^{\ell,2}(\Omega,\vp; X)}
\]
where $C$ is independent of $f$, $\vp$, and $\ep$.
\end{prop}

The function $\psi$ that we construct will actually satisfy $\supp \psi \subset \Omega_\ep'\cup\Omega$.

\begin{proof}The proof is a consequence of Theorem \ref{thm:simple extension operators exist}.
Let $\ep>0$. Constructing the functions $u_1,\dots, u_\ell$ as in the proof of Theorem \ref{thm:simple extension operators exist}.
Then for any $1\leq j \leq \ell$,
$(f-u_1-\cdots -u_j)\in W^{\ell,2}(\Omega,\vp; X)\cap W^j_0(\Omega,\vp; X)$ and
\[
\| f - u_1-\cdots -u_j \|_{W^{\ell,2}(\Omega,\vp; X)} \leq K \| f \|_{W^{\ell,2}(\Omega,\vp; X)}.
\]
Thus, $f-(u_1+\cdots+u_\ell) \in W^{\ell,2}_0(\Omega,\vp; X)$ so there exists $\xi\in C^\infty_c(\Omega)$ so that
\[
\| f - (u_1+\cdots+u_\ell+\xi) \|_{W^{\ell,2}(\Omega,\vp; X)} < \min\{\ep,\ep\| f\|_{W^{\ell,2}(\Omega,\vp; X)}\}.
\]
We can now take $\psi = u_1+\cdots + u_j +\xi$.
\end{proof}

\begin{prop} \label{prop:interpolation of Sobolev norms (only integers)}
Let $\Omega\subset\R^n$ have a $C^{m}$ boundary and satisfy (HI)-(HVI).  There exist $K,K'>0$ depending on $n,m$ so that
for each $\ep>0$, $u\in W^{m,2}(\Omega,\vp; X)$, and $0< j < m$,
\begin{align}
\sum_{|\alpha| =j} \| X^\alpha u \|_{L^2(\Omega,\vp)}^2 &\leq K \Big( \ep \sum_{|\beta| =m} \| X^\beta u \|_{L^2(\Omega,\vp)}^2 + \ep^{-j/(m-j)} \| u\|_{L^2(\Omega,\vp)}^2\Big)
\label{eqn:interpolation of |alpha|=j}\\
\| u\|_{W^{j,2}(\Omega,\vp; X)} &\leq K' \big( \ep \| u\|_{W^{m,2}(\Omega,\vp; X)} + \ep^{-j/(m-j)} \| u\|_{L^2(\Omega,\vp)}\big) \label{eqn:interpolation of |alpha| |eq j}\\
\| u\|_{W^{j,2}(\Omega,\vp; X)} &\leq 2K' \| u\|_{W^{m,2}(\Omega,\vp; X)}^{j/m} \| u\|_{L^2(\Omega,\vp)}^{(m-j)/m} \label{eqn:j norm convex of m norm and 0 norm}
\end{align}
\end{prop}

\begin{proof}  The proof of Proposition \ref{prop:interpolation of Sobolev norms (only integers)} is the same as the proof of Proposition \ref{prop:interpolation of Sobolev norms (only integers) on M} after using Theorem \ref{thm:simple extension operators exist} to reduce the problem to $\R^n$.
\end{proof}

We can also prove an $L^2$ analog to \cite[Theorem 5.33]{AdFo03}, the Approximation Theorem for $\R^n$.

\begin{prop}\label{prop:approximation theorem}
Let $\Omega\subset \R^n$ satisfy (HI)-(HVI) for some $m\in\N$.
If $0 <k \leq m$ then there exists a constant $C = C(m,n)$  so that for
$v\in W^{k,2}(\Omega,\vp; X)$  and $0<\ep\leq 1$, there exists $v_\ep \in C^\infty_c(\Omega)$ so that:
\[
\| v - v_\ep \|_{L^2(\Omega,\vp)} \leq C\ep^k \sum_{|\alpha|=k} \| X^\alpha v \|_{L^2(\Omega,\vp)}
\]
and
\[
\| v_\ep \|_{W^{j,2}(\Omega,\vp; X)}
\leq C \begin{cases} \| v \|_{W^{k,2}(\Omega,\vp; X)} & \text{if }j\leq k-1 \\ \ep^{k-j} \| v \|_{W^{k,2}(\Omega,\vp; X)} &\text{if }k\leq j \leq m.
\end{cases}
\]
\end{prop}
\begin{proof} The proof uses the same argument as the proof of Proposition \ref{prop:approximation theorem on M}.
\end{proof}
Proposition \ref{prop:approximation theorem}  means that $\Omega$ has the approximation property in the sense of \cite{AdFo03}. As a consequence of our improved approximation results, following the proof
of Proposition \ref{prop:W^k in H on M}, we can prove
\begin{prop}\label{prop:W^k in H}
If $\Omega\subset\R^n$ satisfies (HI)-(HVI), then
\[
W^{k,2}(\Omega,\vp; X) \in \H\big(k/m; L^2(\Omega,\vp),W^{m,2}(\Omega,\vp; X)\big).
\]
\end{prop}

\subsection{Besov spaces on $\Omega$ and Additional Trace Results}
\label{sec:Besov_spaces_and_Additional_Trace_Results}
We are now ready to define weighted Besov spaces on $\Omega$.
\begin{defn}\label{defn:Besov spaces}
Let $0<s<\infty$, $1\leq p < \infty$, $1\leq q \leq \infty$ and $\ell$ be the smallest integer larger than $s$. We define the Besov space $B^{s;p,q}(\Omega,\vp; X)$
to be the intermediate space between $L^p(\Omega)$ and $W^{\ell,p}(\Omega,\vp; X)$ corresponding to $\theta = s/\ell$, i.e.,
\[
B^{s;p,q}(\Omega,\vp; X) = \big( L^p(\Omega,\vp),W^{\ell,p}(\Omega,\vp; X)\big)_{s/\ell,q;J}.
\]
\end{defn}
We will focus on the case $p=2$ since we only proved an $L^2$ Approximation Theorem. By Theorem \ref{thm:J method produces intermediate spaces},
$B^{s;p,q}(\Omega,\vp; X)$ is a Banach space with interpolation norm
\[
\| u \|_{B^{s;p,q}(\Omega,\vp; X)} = \big \| u ; \big(L^p(\Omega,\vp), W^{\ell,p}(\Omega,\vp; X)\big)_{s/\ell,q;J} \big\|.
\]
Also, $B^{s;p,q}(\Omega,\vp; X)$ inherits density and approximation properties from $W^{\ell,p}(\Omega,\vp; X)$. For example,
$\{ \psi\in C^\infty(\Omega): \| \psi \|_{W^{\ell,p}(\Omega,\vp; X)}<\infty\}$ is dense in $B^{s;p,q}(\Omega,\vp; X)$.

For $\Omega$ that satisfies \emph{(HI)}-\emph{(HVI)}, Proposition \ref{prop:approximation theorem} and the Reiteration Theorem (Theorem \ref{thm:reiteration theorem}),
if $0\leq k < s < \ell$ and $s=(1-\theta)k + \theta \ell$, then
\[
B^{s;2,q}(\Omega,\vp; X) = \big(W^{k,2}(\Omega,\vp; X), W^{\ell,2}(\Omega,\vp; X)\big)_{\theta,q;J}.
\]
More generally, if $0\leq k < s < \ell$,
$s = (1-\theta)s_1+\theta s_2$, and $1 \leq q_1,q_2\leq\infty$, then
\begin{equation}\label{eqn:interpolation, q's change}
B^{s;2,q}(\Omega,\vp; X) = \big( B^{s_1;2,q_1}(\Omega,\vp; X), B^{s_2;2,q_2}(\Omega,\vp; X)\big)_{\theta,q;J}.
\end{equation}

We are now in a position to prove the following Trace Lemma.
\begin{lem}\label{lem:Besov functions have L^2 traces on M}
The trace operator $\Tr$  embeds
$B^{1/2;2,1}(\Omega_\ep'\cap\Omega,\vp;X)$ into $L^2(M)$.
\end{lem}

\begin{proof} Let $U$ be an element in $B =  B^{1/2;2,1}(\Omega_\ep'\cap\Omega,\vp;X)$. Without loss of generality, we may assume that $\|U\|_B \leq 1$.
By the discrete $J$-interpolation method, there exist functions $U_i$, $i\in\Z$, so that $U = \sum_{i\in\Z} U_i$ and
\[
\sum_{i\in\Z} 2^{-i/2} \|U_i\|_{L^2(\Omega_\ep'\cap\Omega,\vp)}\leq C
\qquad\text{and}\qquad
\sum_{i\in\Z} 2^{i/2} \| U_i\|_{W^{1,2}(\Omega_\ep'\cap\Omega,\vp;X)}\leq C
\]
for some constant $C$. As in the proof of Lemma \ref{lem:traces exist for u in sobo spaces on M}, we may assume that the functions $U_i$ are smooth and
at most finitely many are not identically zero. For any of these functions, we have, for $2^i\leq t \leq 2^{i+1}$ and $x\in M$,
\begin{align*}
\Big| e^{-\frac{\vp(x)}2}U_i(x)\Big|
&\leq \int_0^t \Big| \frac{d}{ds} e^{-\frac{\vp(e^{s{\oldL}_n}(x))}2}U_i\big(e^{s{\oldL}_n}(x)\big)\Big|\, ds +  \Big| e^{-\frac{\vp(e^{t{\oldL}_n}(x))}2}U_i\big(e^{t{\oldL}_n}(x)\big)\Big|  \\
&= \int_0^t \Big| {\oldL}_n\big( e^{-\frac{\vp(e^{s{\oldL}_n}(x))}2}U_i\big(e^{s{\oldL}_n}(x)\big)\big)\Big|\, ds + \Big| e^{-\frac{\vp(e^{t{\oldL}_n}(x))}2}U_i\big(e^{t{\oldL}_n}(x)\big)\Big| \\
&\leq \int_0^{2^{i+1}}\Big| {\oldL}_n\big( e^{-\frac{\vp(e^{s{\oldL}_n}(x))}2}U_i\big(e^{s{\oldL}_n}(x)\big)\big)\Big|\, ds + \Big| e^{-\frac{\vp(e^{t{\oldL}_n}(x))}2}U_i\big(e^{t{\oldL}_n}(x)\big)\Big|.
\end{align*}
Averaging $t$ over $[2^i,2^{i+1}]$, we now have the estimate
\[
\Big| e^{-\frac{\vp(x)}2}U_i(x)\Big|
\leq \int_0^{2^{i+1}} \Big| {\oldL}_n\big( e^{-\frac{\vp(e^{s{\oldL}_n}(x))}2}U_i\big(e^{s{\oldL}_n}(x)\big)\big)\Big|\, ds
+ \frac{1}{2^i} \int_{2^i}^{2^{i+1}} \Big| e^{-\frac{\vp(e^{t{\oldL}_n}(x))}2}U_i\big(e^{t{\oldL}_n}(x)\big)\Big|\, dt.
\]
By Cauchy-Schwarz,
\begin{multline*}
\big| e^{-\frac{\vp(x)}2}U_i(x)\big|
\leq 2^{(i+1)/2} \bigg(\int_0^{2^{i+1}}\big| Y_n U_i\big(e^{t{\oldL}_n}(x)\big)\big|^2 e^{-\vp(e^{t{\oldL}_n}(x))} \, dt\bigg)^{1/2} \\
+ 2^{-i/2}\bigg(\int_{2^i}^{2^{i+1}} | U_i\big(e^{t{\oldL}_n}(x)\big)|^2 e^{-\vp(e^{t{\oldL}_n}(x))}\, dt\bigg)^{1/2} := a_i(x) + b_i(x).
\end{multline*}
Observe that $\|a_i\|_{L^2(M)} \leq C 2^{i/2} \| U_i \|_{W^{1,2}(\Omega_\ep'\cap\Omega,\vp;X)}$ and
$\|b_i\|_{L^2(M)} \leq 2^{-i/2}\| U_i \|_{L^2(\Omega_\ep'\cap\Omega,\vp)}$.
Summing in $i$, we have
\[
\| U \|_{L^2(M,\vp)}
\leq \sum_{i\in\Z} \big\| e^{-\frac{\vp}2}U_i\big\|_{L^2(M)}
\leq C \Big( \sum_{i\in\Z} 2^{i/2} \| U_i\|_{W^{1,2}(\Omega_\ep'\cap\Omega,\vp;X)} + 2^{-i/2}\| U_i \|_{L^2(\Omega_\ep'\cap\Omega,\vp)}\Big) \leq C.
\]
\end{proof}

As a consequence of Theorem \ref{thm:Trace Theorem for integer m on M}
and Lemma \ref{lem:Besov functions have L^2 traces on M}, we can now prove our Trace Theorem for $L^2$ Besov spaces, Theorem \ref{thm:trace theorem on M}. Theorem \ref{thm:trace theorem on M}
is an $L^2$-analog  of \cite[Theorem 7.43]{AdFo03}.

\begin{proof}[Proof of Theorem \ref{thm:trace theorem on M}]
The proof of Theorem \ref{thm:trace theorem on M} follows from (\ref{eqn:interpolation, q's change on M}), Theorem \ref{thm:Trace Theorem for integer m on M}, Lemma
\ref{lem:Besov functions have L^2 traces on M} and the Exact Interpolation Theorem.
\end{proof}

We conclude our discussion of Sobolev space results with an extension of Proposition \ref{prop:compactness of W^1_0 into L^2} and Corollary
\ref{cor:adjoint dominated in L^2}, namely the proof of Proposition \ref{prop:weighted and unweighted derivs have equiv norm}.
\begin{proof}[Proof of Proposition \ref{prop:weighted and unweighted derivs have equiv norm}]
We will first show that $\| v\|_{W^{1,2}(\Omega,\vp;D)} \leq C \| v \|_{W^{1,2}(\Omega,\vp;X)}$ for some $C$ independent of $v$.
By Theorem \ref{thm:Trace Theorem for integer m on M}, $\Tr v \in  B^{1/2;2,2}(M,\vp;T)$, and there exists $v'\in W^{1,2}(\Omega_\ep',\vp;X)$ with
support in $\Omega_\ep'$ so that
\[
\Tr v' = \Tr v
\]
and
\[
\| v' \|_{W^{1,2}(\Omega,\vp;X)} \sim \| \Tr v'\|_{B^{1/2;2,2}(M,\vp,T)} \leq C \| v\|_{W^{1,2}(\Omega,\vp;X)} .
\]
Writing $v = (v-v')+v'$, we see that $v-v'\in W^{1,2}_0(\Omega,\vp;X)$. Since $W^{1,2}_0(\Omega,\vp;X)=W^{1,2}_0(\Omega,\vp;D)$ by  Corollary \ref{cor:adjoint dominated in L^2}
and Corollary \ref{cor:adjoint dominated in L^2 (unweighted)}, we estimate
\[
\| v- v'\|_{W^{1,2}(\Omega,\vp;D)} \leq C \| v- v'\|_{W^{1,2}(\Omega,\vp;X)} \leq C\| v\|_{W^{1,2}(\Omega,\vp;X)}.
\]
Thus, we need only to prove the result for $v'$. However, since $v'$ has compact support in $\Omega_\ep'$, we can  use  Corollary \ref{cor:adjoint dominated in L^2} to bound
\[
\|v'\|_{W^{1,2}(\Omega,\vp;D)} \leq \|v'\|_{W^{1,2}(\Omega_\ep',\vp;D)} \leq \|v'\|_{W^{1,2}(\Omega_\ep',\vp;X)}
\leq \| \Tr v'\|_{B^{1/2;2,2}(M,\vp,T)} \leq C \| v\|_{W^{1,2}(\Omega,\vp;X)},
\]
and the result is proved for $k=1$.

To show that $W^{k,2}(\Omega,\vp;X) = W^{k,2}(\Omega,\vp;D)$ for $k\geq 2$, we induct. $k=1$ is the base case. If we assume the norms are equivalent up to order $k-1$, then
let $|\alpha|=k$ and $|\beta|=k-1$ so that $X^\alpha = X^\beta X_j$ for some $j$. Then
\begin{align*}
\| D^\alpha v\|_{L^2(\Omega,\vp)} &= \|D^\beta D_j v\|_{L^2(\Omega,\vp)} \leq C \sum_{|\gamma|=k-1} \| X^\gamma D_j v\|_{L^2(\Omega,\vp)} \\
&\leq C \sum_{|\gamma|=k-1}\Big( \|X^\gamma X_j v\|_{L^2(\Omega,\vp)} + \Big\|X^\gamma \frac{\p \vp}{\p x_j} v\Big\|_{L^2(\Omega,\vp)} \Big).
\end{align*}
The first term is bounded by $\|v\|_{W^{k,2}(\Omega,\vp;X)}$. The second term can be estimated as follows:
\begin{align*}
\Big\|X^\gamma \frac{\p \vp}{\p x_j} v\Big\|_{L^2(\Omega,\vp)}
&\leq \sum_{\gamma_1+\gamma_2=\gamma} C\Big\|\Big(D^{\gamma_1} \frac{\p \vp}{\p x_j}\Big) X^{\gamma_2} v \Big\|_{L^2(\Omega,\vp)}
\end{align*}
By \emph{(HIV)}, the derivatives of $\vp$ are controlled by $|\nabla \vp|$ which in turn is controlled by $X_\ell+X_\ell^*$. Thus, the second term is also controlled by
$\|v\|_{W^{k,2}(\Omega,\vp;X)}$.

The proof that $W^{1,2}_0(\Omega,\vp;X)$ embeds compactly in $L^2(\Omega,\vp;X)$ is
contained in Proposition \ref{prop:compactness of W^1_0 into L^2} and its corollaries. We will show that this implies that $W^{k,2}_0(\Omega,\vp;X)$ embeds compactly in
$W^{k-1,2}_0(\Omega,\vp;X)$ by induction.

Assume that for some $2\leq k\leq m-1$, the result holds for $1\leq j \leq k-1$. Let $\{\psi_\ell\}$ be a bounded sequence of functions in $W^{k,2}_0(\Omega,\vp;X)$. Then $\{\psi_\ell\}$ is a bounded sequence of
functions in $W^{k-1,2}_0(\Omega,\vp;X)$ and hence there exists a subsequence (renamed to be $\psi_\ell$) converging in $W^{k-2,2}_0(\Omega,\vp;X)$. Moreover,
$\{\nabla_X \psi_\ell\}$ is a sequence of bounded vectors whose components are in $W^{k-1,2}_0(\Omega,\vp;X)$, so there exists a further subsequence, renamed $\psi_\ell$,
so that $\{\nabla_X \psi_\ell\}$ is Cauchy in $W^{k-2,2}_0(\Omega,\vp;X)$. It now follows that $\{\psi_\ell\}$ is Cauchy in $W^{k-1,2}_0(\Omega,\vp;X)$.

Next, let $\{f_\ell\}$ be a bounded sequence of functions in $W^{k,2}(\Omega,\vp;X)$.
Using the simple $(k,2)$-extension operator $E$ from Theorem \ref{thm:simple extension operators exist}, we extend $f_\ell$ to $Ef_\ell\in W^{k,2}_0(\R^n,\vp;X)$. By the previous
paragraph with $\R^n$ playing the role of $\Omega$, there exists a subsequence $Ef_{\ell_j}$ that converges in $W^{k-1,2}_0(\R^n,\vp;X)$. Since $Ef_{\ell_j}\big|_\Omega = f_{\ell_j}$, it follows
that $f_{\ell_j}$ converges in $W^{k-1,2}(\Omega,\vp;X)$.
\end{proof}

%
%
\section{Interior estimates -- the proof of Theorem \ref{thm: elliptic regularity, higher order, interior}}\label{sec:interior est}

\subsection{The $\ell=0$ case}
\begin{proof}[Proof of Theorem \ref{thm: elliptic regularity, higher order, interior} for $\ell=0$]  We follow the outline of \cite[\S6.3, Theorem 1]{Eva10}. Choose $W\subset\Omega$ so that $V\subset W$,
$\dist(V,\bd W) >0$, and $\dist(W,\bd\Omega)>0$. We first assume that $V$ and $W$ are bounded.
Let $\zeta\in C^\infty(\Omega)$ be a smooth cutoff so that $\zeta\big|_V =1$ and $\supp\zeta\subset W$.
Since $L$ is elliptic, by the classical theory $u\in W^{2,2}_{\loc}(\Omega,\vp; X)$.

Since $u$ is a weak solution of $Lu=f$, we have ${\oldD}(v,u) = (v,f)_\vp$ for all $v\in W^{1,2}_0(\Omega,\vp; X)$. Thus
\begin{equation}\label{eqn:weak solution, second order part isolated}
\sum_{j,k=1}^n \int_{\Omega} X_j v \overline{a_{jk}} \overline{X_k u} e^{-\vp}\, dx
= \int_{\Omega} v \overline{g} e^{-\vp}\, dx
\end{equation}
where
\begin{equation}\label{eqn:g defn}
g = f - \sum_{j=1}^n\big( b_j X_j u + b_j' X_j^*u + (X_j^* b_j')u\big) - bu.
\end{equation}
We would like to use \eqref{eqn:weak solution, second order part isolated} substituting $v = X_\ell^*(\zeta^2 X_\ell u)$. This is problematic as
$u\in W^{2,2}(W,\vp;X)$ and not thrice-differentiable. Instead, we let $u_\ep\in C^\infty_c(\Omega)$ be so that $u_\ep\to u$ in $W^{2,2}(W,\vp;X)$. We set
$v_\ep = X_\ell^*(\zeta^2 X_\ell u_\ep)$.
In this case, the left-hand side of
(\ref{eqn:weak solution, second order part isolated}) becomes
\[
A_\ep = \sum_{j,k=1}^n \big( X_j(X_\ell^* \zeta^2 X_\ell u_\ep), a_{jk} X_k u\big)_\vp,
\]
and the right-hand side becomes
\[
B_\ep  = \int_{\Omega} v_\ep \bar g e^{-\vp}\, dx
= \Big( X_\ell^* \zeta^2 X_\ell u_\ep, f - \sum_{j=1}^n\big( b_j X_j u + b_j' X_j^* u - \frac{\p b_j'}{\p x_j} u\big) - bu \Big)_\vp.
\]
Equation \eqref{eqn:weak solution, second order part isolated} now says that $A_\ep = B_\ep$.

Observe that
\begin{align}
A_\ep &= \sum_{j,k=1}^n \big( X_\ell^* X_j\zeta^2 X_\ell u_\ep, a_{jk} X_k u\big)_\vp +\sum_{j,k=1}^n ([X_j,X_\ell^*]\zeta^2 X_\ell u, a_{jk} X_k u\big)_\vp\nn\\
&= \sum_{j,k=1}^n \big(  X_j\zeta^2 X_\ell u_\ep, a_{jk} X_\ell X_k u\big)_\vp +\sum_{j,k=1}^n \Big[ ([X_j,X_\ell^*]\zeta^2 X_\ell u, a_{jk} X_k u\big)_\vp\\
&\qquad+ \Big(  X_j\zeta^2 X_\ell u_\ep, \frac{\p a_{jk}}{\p x_\ell} X_k u\Big)_\vp\Big]. \label{eqn: A_ep but can take limits}
\end{align}
Since no more than two derivatives  of $u_\ep$ are taken in $A_\ep$, we can let $\ep\to0$ and observe that $A=B$ where
\[
A = \sum_{j,k=1}^n \big(  X_j\zeta^2 X_\ell u, a_{jk} X_\ell X_k u\big)_\vp +\sum_{j,k=1}^n \Big[ ([X_j,X_\ell^*]\zeta^2 X_\ell u, a_{jk} X_k u\big)_\vp
+ \Big(  X_j\zeta^2 X_\ell u, \frac{\p a_{jk}}{\p x_\ell} X_k u\Big)_\vp\Big] 
\]
and
\begin{equation}\label{eqn:B, post limit}
B = \Big( X_\ell^* \zeta^2 X_\ell u, f - \sum_{j=1}^n\big( b_j X_j u + b_j' X_j^* u - \frac{\p b_j'}{\p x_j} u\big) - bu \Big)_\vp.
\end{equation}
We continue our investigation of $A$.
Observe that
\[
[X_j,X_\ell^*] = - \frac{\p\vp}{\p x_j\p x_\ell}
\quad\text{and}\quad
[X_\ell,  X_k]=\big([X_\ell,X_k]^*\big)^* = [X_k^*,X_\ell^*]^*=0.
\]
We have
\begin{align}
A &= \sum_{j,k=1}^n \big(  X_j\zeta^2 X_\ell u, a_{jk} X_k X_\ell  u\big)_\vp +\sum_{j,k=1}^n \Big[ ([X_j,X_\ell^*]\zeta^2 X_\ell u, a_{jk} X_k u\big)_\vp \nn\\
&+ \Big(  X_j\zeta^2 X_\ell u, \frac{\p a_{jk}}{\p x_\ell} X_k u\Big)_\vp + \big(  X_j\zeta^2 X_\ell u, a_{jk} \underbrace{[X_\ell, X_k]}_{=0}  u\big)_\vp \Big] \nn\\
&= \sum_{j,k=1}^n \big( \zeta^2 X_j X_\ell u, a_{jk} X_k X_\ell  u\big)_\vp +\sum_{j,k=1}^n \Big[ ([X_j,X_\ell^*]\zeta^2 X_\ell u, a_{jk} X_k u\big)_\vp \nn\\
&+ \Big(  X_j\zeta^2 X_\ell u, \frac{\p a_{jk}}{\p x_\ell} X_k u\Big)_\vp 
+ \big(  2\zeta \frac{\p\zeta}{\p x_j} X_\ell u, a_{jk} X_k X_\ell  u\big)_\vp \Big]. \label{eqn:A, post IBP}
\end{align}
The strong ellipticity condition implies that
\begin{equation}\label{eqn:A strong ellipticity}
 \sum_{j,k=1}^n \big( \zeta^2 X_j X_\ell u, a_{jk} X_k X_\ell  u\big)_\vp \geq \theta \| \zeta\nabla_X X_\ell u \|_{L^2(W,\vp)}.
\end{equation}
The remaining terms we bound as follows:
\[
\Big| \Big(  X_j\zeta^2 X_\ell u, \frac{\p a_{jk}}{\p x_\ell} X_k u\Big)_\vp + \big(  2\zeta \frac{\p\zeta}{\p x_j} X_\ell u, a_{jk} X_k X_\ell  u\big)_\vp\Big|\\
\leq  C_1 \| \zeta \nabla_X X_\ell u \|_{L^2(W,\vp)} \| \nabla_X u \|_{L^2(W,\vp)},
\]
where $C_1$ depends on $\|a_{jk}\|_{C^1(\Omega)}$ and $\|\zeta\|_{C^1(\Omega)}$. In particular, $C_1$ does not depend on $|\supp\zeta|$.
Next, using \emph{(HIV)}, Corollary \ref{cor:adjoint dominated in L^2} and the fact that
$\frac{\p\vp}{\p x_j} = -X_j^* - X_j$, we have
\begin{align*}
\Big| ([X_j,X_\ell^*]\zeta^2 X_\ell u, a_{jk} X_k u\big)_\vp \Big|
&\leq C_2 \big( (1+|\nabla\vp|) \zeta^2 |X_\ell u|, |a_{jk}| |X_k u|\big)_\vp \\
&\leq C_2' \| \zeta\nabla_X u \|_{L^2(W,\vp)}\big(\| \zeta\nabla_X u \|_{L^2(W,\vp)} + \| \zeta\nabla_X X_\ell u \|_{L^2(W,\vp)}\big),
\end{align*}
where $C_2'$ depends on $C_2$ and $\| a_{jk}\|_{L^\infty(\Omega)}$.
Thus, using (\ref{eqn:A strong ellipticity}) and the bounds on the error terms, we can bound (with $C = n^2(C_1+C_2')$)
\begin{align}
|A| &\geq \theta \| \zeta\nabla_X X_\ell u \|_{L^2(W,\vp)}^2
- C(\| \zeta\nabla_X X_\ell u \|_{L^2(W,\vp)}\| \zeta\nabla_X u \|_{L^2(W,\vp)} +  \|\nabla_X u \|_{L^2(W,\vp)}^2 \big) \nn\\
&\geq  \frac\theta2 \| \zeta\nabla_X X_\ell u \|_{L^2(W,\vp)}^2
- C_3 \|\nabla_X u \|_{L^2(W,\vp)}^2,  \label{eqn:good est for A}
\end{align}
where $C_3 = C_3(\|a_{jk}\|_{C^1(\Omega)},\|\zeta\|_{C^1(\Omega)},n,\theta)$.
We can bound $B$ with Cauchy-Schwarz and the small constant/large constant inequality. In particular, we can use Corollary \ref{cor:adjoint dominated in L^2}
to show that for some constant
$C_4>0$ where $C_4 = C_4(\|b_j\|_{L^\infty(\Omega)}, \|b_j'\|_{C^1(\Omega)}, \|b\|_{L^\infty(\Omega)},\|\zeta\|_{C^1(\Omega)},n)$, we have the estimate
\begin{align}
|B| &\leq  C_4 \|\zeta\nabla_X X_\ell u\|_{L^2(W,\vp)}\Big[ \|f\|_{L^2(W,\vp)} + \| u\|_{W^{1,2}(W,\vp;X)} \Big] \nn\\
&\leq \frac{\theta}4 \|\zeta\nabla_X X_\ell u\|_{L^2(W,\vp)}^2 + C_5\big(  \|f\|_{L^2(W,\vp)}^2 + \| u\|_{W^{1,2}(W,\vp;X)}^2\big),
 \label{eqn:B estimate}
\end{align}
where $C_5 = C_5(\|b_j\|_{L^\infty(\Omega)}, \|b_j'\|_{C^1(\Omega)}, \|b\|_{L^\infty(\Omega)},n,\theta)$.

Combining (\ref{eqn:good est for A}) and (\ref{eqn:B estimate}), we have shown that
\begin{equation}\label{eqn:W^2 est in terms of W^1 and L^2}
\| u \|_{W^{2,2}(V,\vp;X)}^2 \leq C_6\big(\| f\|^2_{L^2(W,\vp)} + \| u \|_{W^{1,2}(W,\vp;X)}^2 \big),
\end{equation}
where $C_6 = C_6(\|a_{jk}\|_{C^1(\Omega)},\|\zeta\|_{C^1(\Omega)},\|b_j\|_{L^\infty(\Omega)}, \|b_j'\|_{C^1(\Omega)}, \|b\|_{L^\infty(\Omega)},n,\theta)$.
We can improve the estimate \eqref{eqn:W^2 est in terms of W^1 and L^2} and replace $\| u \|_{W^{1,2}(W,\vp;X)}^2$ with $\| u \|_{L^2(W,\vp)}^2$.
Let $\eta\in C^\infty_c(\Omega)$ be a cutoff so that $\eta|_W =1$. Using (\ref{eqn:weak solution, second order part isolated}) and strong ellipticity, we estimate
\begin{align*}
\| \eta \nabla_X u \|_{L^2(\Omega,\vp)}^2
&\leq \frac 1{\theta} \sum_{j,k=1}^n \big(\eta^2 X_j u,a_{jk} X_k u\big)_\vp = \frac{1}{\theta} \big| \big(\eta^2 u, g\big)_\vp\big| \\
&\leq C_7 \|u\|_{L^2(\Omega,\vp)}\big( \|f\|_{L^2(\Omega,\vp)} + \|\eta\nabla_X u\|_{L^2(\Omega,\vp)} + \|u\|_{L^2(\Omega,\vp)}\big),
\end{align*}
where $C_7 = C_7(\|b_j\|_{L^\infty(\Omega)}, \|b_j'\|_{C^1(\Omega)}, \|b\|_{L^\infty(\Omega)},n,\theta)$.
Using a small constant/large constant argument, we have
\[
\| \nabla_X u \|_{L^2(W,\vp)} \leq \| \eta\nabla_X u\|_{L^2(\Omega,\vp)} \leq C_8\big(\|f\|_{L^2(\Omega,\vp)} + \|u\|_{L^2(\Omega,\vp)}\big),
\]
where $C_8 = C_8(\|b_j\|_{L^\infty(\Omega)}, \|b_j'\|_{C^1(\Omega)}, \|b\|_{L^\infty(\Omega)},n,\theta)$. Thus, we can refine (\ref{eqn:W^2 est in terms of W^1 and L^2}) by
\[
\| u \|_{W^{2,2}(V,\vp;X)}^2 \leq C\big(\| f\|_{L^2(\Omega,\vp)} + \| u \|_{L^2(\Omega,\vp)}^2 \big),
\]
where
$C = C(\dist(V,\bd\Omega),\|a_{jk}\|_{C^1(\Omega)},\|b_j\|_{L^\infty(\Omega)}, \|b_j'\|_{C^1(\Omega)}, \|b\|_{L^\infty(\Omega)},n,\theta)$. $\zeta$ and $\eta$ have disappeared from
$C$ as $\|\zeta\|_{C^1(\Omega)}$ depends only on $\dist(V,\bd\Omega)$. Thus, we can relax the boundedness condition on $V$ and let $V$ be as in the statement of the theorem.
\end{proof}

\subsection{$\ell\geq 1$ case}
\begin{proof}[Proof of Theorem \ref{thm: elliptic regularity, higher order, interior}, $\ell\geq 1$]
As with the $\ell=0$ case, that $u\in W^{\ell+2,2}_{\loc}(\Omega,\vp; X)$ follows from the classical theory.
We will establish \eqref{eqn:u in W^m+2 is bounded by f, u in W^m -- interior} by induction. The $\ell=0$ case has already been established.

Let $V\subset W\subset\Omega$ so that  $\dist(V,\bd W)>0$ and $\dist(W,\bd\Omega)>0$.

Assume  that \eqref{eqn:u in W^m+2 is bounded by f, u in W^m -- interior} holds for a nonnegative integer $\ell$,
for $a_{jk},  b_j'\in C^{\ell+2}(\Omega)\cap W^{\ell+2,\infty}(\Omega)$, for $b_j,b \in C^{\ell+1}(\Omega)\cap W^{\ell+1,\infty}(\Omega)$, and for $f\in W^{\ell+1,2}(\Omega,\vp; X)$.  Assume further 
that $u\in W^{1,2}(\Omega,\vp; X)$ is a weak solution of $Lu=f$ in $\Omega$. By the induction hypothesis, we have the estimate
\[
\| u \|_{W^{\ell+2}(W,\vp;X)} \leq C\big( \|f\|_{W^\ell(\Omega,\vp; X)} + \|u\|_{L^2(\Omega,\vp)}\big),
\]
where $C = C(\dist(W,\bd\Omega),\|a_{jk}\|_{C^{\ell+1}(\Omega)},\|b_j\|_{C^{\ell}(\Omega)}, \|b_j'\|_{C^{\ell+1}(\Omega)}, \|b\|_{C^\ell(\Omega)},n,\theta,\ell)$.
Let $\alpha$ be a multiindex of length $|\alpha|=\ell+1$. Let $\tilde v\in C^\infty_c(W)$ and set
\[
v = (X^\alpha)^*\tilde v
\quad\text{and}\quad
\tilde u = X^\alpha u.
\]
Since ${\oldD}(v,u)=(v,f)_\vp$, we plug in $v = (X^\alpha)^*\tilde v$ and compute
\begin{align*}
{\oldD}(v,u) &= \sum_{j,k=1}^n\big(X_j (X^\alpha)^*\tilde v, a_{jk} X_k u\big)_\vp + \sum_{j=1}^n\Big[ \big((X^\alpha)^*\tilde v,b_j X_j u\big)_\vp
+ \big(X_j (X^\alpha)^*\tilde v, b_j' u\big)_\vp\Big]\nn\\
& + \big((X^\alpha)^*\tilde v,bu\big)_\vp.
\end{align*}
Integrating by parts gives us
\begin{align*}
&{\oldD}(v,u) = \sum_{j,k=1}^n \big(X_j \tilde v, a_{jk} X_k X^\alpha u\big)_\vp + \sum_{j=1}^n\Big( \big( \tilde v,b_j X_j X^\alpha u\big)_\vp
+ \big( \tilde v, X_j^*( b_j' X^\alpha) u\big)_\vp\Big) + \big(\tilde v,bX^\alpha u\big)_\vp \nn\\
&+ \sum_{j,k=1}^n\bigg(\big(\tilde v, [X_j, (X^\alpha)^*]^* a_{jk} X_k u\big)_\vp
+ \sum_{\atopp{\beta\subset\alpha}{\beta\neq\alpha}} \big( \tilde v, (D^{\alpha-\beta}a_{jk})\, X_kX^\beta u \big)_\vp\bigg)\nn\\
&+ \sum_{j=1}^n \big(\tilde v, [X_j, (X^\alpha)^*]^* b_j' u\big)_\vp \nn\\
&+ \sum_{\atopp{\beta\subset\alpha}{\beta\neq\alpha}}\Bigg( \big(\tilde v, (D^{\alpha-\beta} b)\, X^\beta u \big)_\vp
+ \sum_{j=1}^n\bigg( \big(\tilde v, (D^{\alpha-\beta}b_j)\, X^\beta X_j u \big)_\vp + \big( \tilde v, X_j^*\big( (D^{\alpha-\beta} b_j')\, X^\beta u\big)\big)_\vp\bigg)\Bigg),
\end{align*}
so
\begin{equation}
{\oldD}(v,u) = {\oldD}(\tilde v,\tilde u) + \big(\tilde v, g\big)_\vp, \label{eqn:integration by parts to control higher order term -- interior ellipticity}
\end{equation}
where
\begin{align*}
g &= \sum_{j,k=1}^n\bigg(  [X_j, (X^\alpha)^*]^* a_{jk} X_k u + \sum_{\atopp{\beta\subset\alpha}{\beta\neq\alpha}} (D^{\alpha-\beta}a_{jk})\, X_kX^\beta u\bigg)
+ \sum_{j=1}^n [X_j, (X^\alpha)^*]^* b_j' u \\
&+ \sum_{\atopp{\beta\subset\alpha}{\beta\neq\alpha}}\Bigg( (D^{\alpha-\beta} b)\, X^\beta u
+ \sum_{j=1}^n\bigg(  (D^{\alpha-\beta}b_j)\, X^\beta X_j u +  X_j^*\big( (D^{\alpha-\beta} b_j')\, X^\beta u\big) \bigg)\Bigg).
\end{align*}

To compute $[X_j, (X^\alpha)^*]^*$, observe that
\begin{align*}
X_j (X^\alpha)^* &= X_j X_{\alpha_1}^*\cdots X_{\alpha_{\ell+1}}^*
= X_{\alpha_1}^* X_j X_{\alpha_2}^*\cdots X_{\alpha_{\ell+1}}^* + [X_j, X_{\alpha_1}^*]X_{\alpha_2}^*\cdots X_{\alpha_{\ell+1}}^* \\
&= \cdots = (X^\alpha)^*X_j + [X_j, X_{\alpha_1}^*] X_{\alpha_2}^*\cdots X_{\alpha_{\ell+1}}^* + \cdots + X_{\alpha_1}^*\cdots X_{\alpha_\ell}^*[X_j,X_{\alpha_{\ell+1}}^*] \\
&= (X^\alpha)^*X_j - \frac{\p\vp}{\p x_j \p x_{\alpha_1}} X_{\alpha_2}^*\cdots X_{\alpha_{\ell+1}}^* - \cdots - X_{\alpha_1}^*\cdots X_{\alpha_\ell}^*\frac{\p\vp}{\p x_j \p x_{\alpha_{\ell+1}}} \\
&= (X^\alpha)^*X_j  + \sum_{\atopp{\gamma\subset\alpha}{\gamma\neq\alpha}} c_{\alpha \gamma} (X^\gamma)^* D^{\alpha-\gamma} \frac{\p\vp}{\p x_j}
\end{align*}
for some constant $c_{\alpha\gamma}$
Thus,
\[
 [X_j, (X^\alpha)^*]^* = \sum_{\atopp{\gamma\subset\alpha}{\gamma\neq\alpha}} c_{\alpha \gamma} \Big( D^{\alpha-\gamma} \frac{\p\vp}{\p x_j} \Big)X^\gamma
\]
and we can rewrite
\begin{align*}
g &= \sum_{\atopp{\gamma\subset\alpha}{\gamma\neq\alpha}}c_{\alpha\gamma}  \bigg[
\sum_{j,k=1}^n   \Big( D^{\alpha-\gamma} \frac{\p\vp}{\p x_j} \Big) X^\gamma \big(a_{jk} X_k u\big)  + \sum_{j=1}^n  \Big( D^{\alpha-\gamma} \frac{\p\vp}{\p x_j} \Big)X^\gamma\big( b_j' u\big)\bigg] \\
& + \sum_{\atopp{\beta\subset\alpha}{\beta\neq\alpha}}\Bigg[ D^{\alpha-\beta} b\, X^\beta u +\sum_{j,k=1}^n D^{\alpha-\beta}a_{jk} X_kX^\beta u
+ \sum_{j=1}^n\bigg[  D^{\alpha-\beta}b_j\, X^\beta X_j u +  X_j^*\big( D^{\alpha-\beta} b_j'\, X^\beta u\big) \bigg]\Bigg] \\
&= \sum_{\atopp{\gamma\subset\alpha}{\gamma\neq\alpha}} \sum_{J\subset \gamma}c_{\alpha\gamma} \bigg[
\sum_{j,k=1}^n  \Big( D^{\alpha-\gamma} \frac{\p\vp}{\p x_j} \Big)D^{\gamma-J}a_{jk}\, X^JX_k u
+ \sum_{j=1}^n  \Big( D^{\alpha-\gamma} \frac{\p\vp}{\p x_j} \Big) D^{\gamma-J} b_j'\, X^J u\bigg] \\
& + \sum_{\atopp{\beta\subset\alpha}{\beta\neq\alpha}}\Bigg[ D^{\alpha-\beta} b\, X^\beta u +\sum_{j,k=1}^n D^{\alpha-\beta}a_{jk} X_kX^\beta u
+ \sum_{j=1}^n\bigg[  D^{\alpha-\beta}b_j\, X^\beta X_j u +  X_j^*\big( D^{\alpha-\beta} b_j'\, X^\beta u\big) \bigg]\Bigg] .
\end{align*}

Thus, if $\tilde f = X^\alpha f -  g$, then we can express ${\oldD}(v,u)=(v,f)_\vp = (\tilde v, X^\alpha f)_\vp$
as ${\oldD}(\tilde v,\tilde u) =(\tilde v,\tilde f)_\vp$. Since $\tilde v\in C^\infty_c(\Omega)$ is arbitrary, $\tilde u$ is a weak solution
for $L\tilde u = \tilde f$. Since
\[
\| \tilde f\|_{L^2(W,\vp)} \leq C'\big( \| f \|_{W^{\ell+1,2}(\Omega,\vp; X)} + \|g\|_{L^2(\Omega,\vp)}\big)
\]
it follows from the induction hypothesis that $\tilde f \in L^2(W)$ with
\[
\| \tilde f\|_{L^2(W,\vp)} \leq C\big( \| f \|_{W^{\ell+1,2}(\Omega,\vp; X)} + \|u\|_{L^2(\Omega,\vp)}\big)
\]

As a consequence of the $\ell=0$ case of Theorem \ref{thm: elliptic regularity, higher order, interior},
\[
\| \tilde u \|_{W^{2,2}(V,\vp;X)} \leq C( \| \tilde f\|_{L^2(W,\vp)} + \|\tilde u\|_{L^2(W,\vp)})
\leq C ( \| f \|_{W^{\ell+1,2}(\Omega,\vp; X)} + \| u\|_{L^2(\Omega,\vp)}).
\]
Since this inequality holds for any $\alpha$ of length $(\ell+1)$, it follows that $u\in W^{\ell+3,2}(V,\vp;X)$ and
\[
\| u \|_{W^{\ell+3}(W,\vp;X)} \leq C\big( \|f\|_{W^{\ell+1}(\Omega,\vp; X)} + \|u\|_{L^2(\Omega,\vp)}\big).
\]
\end{proof}

%
%
\section{Elliptic regularity at the boundary}\label{sec:elliptic regularity -- boundary}
The standard technique to prove elliptic regularity at the boundary is to work locally, rotate, and flatten the domain. Working on a weighted $L^2$ space complicates the matter and we instead
work with tangential and normal derivatives.

\subsection{Tangential Operators}

Recall that a first order differential operator $\Ta$ is a tangential operator if the first order component of $\Ta$ annihilates $\rho$.
By the hypotheses on $\Omega$, there exists $\ep>0$ so that on $\Omega_\ep'$, 
there exist first order differential operators $\Ta_1,\dots, \Ta_n$ so that
\[
\Ta_j = \sum_{\ell=1}^n {\oldb}_{j\ell} X_\ell,
\]
where $({\oldb}_{j\ell})$ is an orthogonal matrix, the components ${\oldb}_{j\ell}$
are bounded in $C^{m-1}(\Omega_\ep')$,
$\Ta_j$ is tangential for $1\leq j \leq n-1$, the first order part of $\Ta_n$ is the unit outward normal to the level curve of $\rho$ and
\begin{equation}\label{eqn:T_j's span}
\sum_{j=1}^n |\Ta_j f|^2 = |\nabla_X f|^2.
\end{equation}
From (\ref{eqn:T_j's span}), it is clear that if $k\leq m$, then
\begin{equation}\label{eqn:m norm of u equiv to tangential norms, normal norms, interior norms}
\| u \|_{W^{k,2}(\Omega,\vp; X)} \sim \sum_{|\alpha|\leq k}\|\Ta^\alpha u\|_{L^2(\Omega_\ep,\vp)} + \| u \|_{W^{k,2}(\Omega\setminus\overline{\Omega_\ep},\vp;X)}.
\end{equation}

By assumption $|d\rho|=1$ on $\bd\Omega$ so that if $\nu_\ep$ is the unit (outward) normal to $\{x\in \C^n : \rho(x)=-\ep\}$, then  $\nu_0^i = \frac{\p\rho(x)}{\p x_i}$. As an immediate
consequence of the Divergence Theorem,
\begin{equation}\label{eqn:divergence thm and rho}
\int_{\Omega} \frac{\p f}{\p x_j}\, dx = \int_{\bd\Omega} f \frac{\p\rho}{\p x_j}\, d\sigma
\end{equation}
where $d\sigma$ is the surface area measure on $\bd\Omega$.

\subsection{Proof of Theorem \ref{thm:sol'n regularity near boundary, higher order}, $\ell=0$ case}
\label{sec:proof_sol'n regularity near boundary, higher order_0}
We are now ready to prove the regularity of solutions of $Lu=f$ near $\bd\Omega$.
\begin{proof}[Proof of Theorem \ref{thm:sol'n regularity near boundary, higher order}, $\ell=0$ case]
Given Theorem \ref{thm: elliptic regularity, higher order, interior}, it is enough to show that $u\in W^{2,2}(\Omega_\ep,\vp;X)$. Let $V,W\subset\Omega_\ep'$
be smooth, bounded domains
so that $V\subset W$ and $\dist(V,\bd W) >0$. Let $\zeta\in C^\infty(\R^n)$ be a smooth cutoff so that $\zeta|_V =1$ and $\supp\zeta\subset W$.
Since $L$ is elliptic and $W$ is bounded, the classical theory yields $u\in W^{2,2}(\Omega_\ep \cap W,\vp;X)$.
By \emph{(HI)}, it is enough to work locally, i.e., we can assume that $\supp u$ is small enough that $T_1,\dots, T_n$ are well-defined on $\supp u$.

The function $u$ is a weak solution of $Lu=f$, so we have ${\oldD}(v,u) = (v,f)_\vp$ for all $v\in \X$. Thus $u$ satisfies the free boundary condition for $\X$ and equations
(\ref{eqn:weak solution, second order part isolated})  and \eqref{eqn:g defn} hold.
As in the proof of Theorem \ref{thm: elliptic regularity, higher order, interior},
we would like to use \eqref{eqn:weak solution, second order part isolated} substituting $v = X_{k}^*(\zeta^2 X_{k} u)$. This is problematic as
$u\in W^{2,2}(W,\vp;X)$ and not thrice-differentiable. Instead, we use Proposition \ref{prop:density of C^infty_c} which constructs
$u_\ep\in C^\infty_c(\R^n)$ so that $u_\ep\to u$ in $W^{2,2}(W\cap\Omega_\ep,\vp;X)$. Let $1\leq {k}\leq n-1$. Then set
$v_\ep = \Ta_{k}^*(\zeta^2 \Ta_{k} u_\ep)$.
In this case, the left-hand side of
(\ref{eqn:weak solution, second order part isolated}) becomes
\[
A_\ep := \sum_{j,{j'}=1}^n \big( X_j\Ta_{k}^*( \zeta^2 \Ta_{k} u_\ep), a_{j{j'}} X_{j'} u_\ep\big)_\vp,
\]
and the right-hand side becomes
\[
B_\ep  := \int_{\Omega} v_\ep \bar g e^{-\vp}\, dx
= \Big( \Ta_{k}^*( \zeta^2 \Ta_{k} u_\ep), f - \sum_{j=1}^n\big( b_j X_j u + b_j' X_j^* u - \frac{\p b_j'}{\p x_j} u\big) - bu \Big)_\vp.
\]
Equation \eqref{eqn:weak solution, second order part isolated} now says that $A_\ep = B_\ep$. Since $\Ta_{k}$ is tangential, $\Ta_{k}^*$ is also tangential and we compute
\begin{align}
A_\ep =& \sum_{j,{j'}=1}^n \Big( \Ta_{k}^* X_j(\zeta^2 \Ta_{k} u_\ep), a_{j{j'}} X_{j'} u\Big)_\vp +\sum_{j,{j'}=1}^n ([X_j,\Ta_{k}^*](\zeta^2 \Ta_{k} u_\ep), a_{j{j'}} X_{j'} u\big)_\vp\nn\\
=& \sum_{j,{j'}=1}^n \Big(  X_j(\zeta^2 \Ta_{k} u_\ep), a_{j{j'}} \Ta_{k} X_{j'} u\Big)_\vp \nn\\
&+\sum_{j,{j'}=1}^n \Big[ \Big( [X_j,\Ta_{k}^*](\zeta^2 \Ta_{k} u_\ep), a_{j{j'}} X_{j'} u\Big)_\vp
+ \sum_{{k}'=1}^n \Big(  X_j(\zeta^2 \Ta_{k} u_\ep), {\oldb}_{{k} {k}'} \frac{\p a_{j{j'}}}{\p x_{{k}'}} X_{j'} u\Big)_\vp\Big]. \label{eqn: A_ep bdy but can take limits}
\end{align}
Since no more than two derivatives  of $u_\ep$ are taken in $A_\ep$, we can let $\ep\to0$ and observe that $A=B$ where
\begin{multline*}
A = \sum_{j,{j'}=1}^n\Bigg[ \Big(  X_j(\zeta^2 \Ta_{k} u), a_{j{j'}} \Ta_{k} X_{j'} u\Big)_\vp \\
+\sum_{j,{j'}=1}^n \Big[ \Big( [X_j,\Ta_{k}^*](\zeta^2 \Ta_{k} u), a_{j{j'}} X_{j'} u\Big)_\vp
+ \sum_{{k}'=1}^n \Big(  X_j(\zeta^2 \Ta_{k} u), {\oldb}_{{k} {k}'} \frac{\p a_{j{j'}}}{\p x_{{k}'}} X_{j'} u\Big)_\vp\Big] \Bigg]
\end{multline*}
and
\begin{equation}\label{eqn:B bdy, post limit}
B = \Big( \Ta_{k}^*( \zeta^2 \Ta_{k} u), f - \sum_{j=1}^n\big( b_j X_j u+ b_j' X_j^* u - \frac{\p b_j'}{\p x_j} u\big) - bu \Big)_\vp.
\end{equation}
We continue our investigation of $A$. Observe that $\Ta_{k}^* = \sum_{{k}'=1}^n \big({\oldb}_{{k} {k}'}X_{{k}'}^* - \frac{\p {\oldb}_{{k} {k}'}}{\p x_{{k}'}}\big)$, so
\[
[X_j,\Ta_{k}^*] =  \sum_{{k}'=1}^n -{\oldb}_{{k}{k}'}\frac{\p^2\vp}{\p x_j\p x_{{k}'}} + \frac{\p {\oldb}_{{k} {k}'}}{\p x_j} X_{{k}'} - \frac{\p^2 {\oldb}_{{k}{k}'}}{\p x_{{k}'}\p x_j}
\]
and
\[
[\Ta_{k}, X_{j'}] = -\sum_{{k}'=1}^n \frac{\p {\oldb}_{{k}{k}'}}{\p x_{j'}} X_{{k}'}.
\]
We have
\begin{align}
A &= \sum_{j,{j'}=1}^n \Big(  X_j(\zeta^2 \Ta_{k} u), a_{j{j'}} X_{j'} \Ta_{k}  u\Big)_\vp +\sum_{j,{j'}=1}^n \Big[ \Big([X_j,\Ta_{k}^*](\zeta^2 \Ta_{k} u), a_{j{j'}} X_{j'} u\Big)_\vp \nn\\
&+ \Big(  X_j(\zeta^2 \Ta_{k} u), \frac{\p a_{j{j'}}}{\p x_{k}} X_{j'} u\Big)_\vp + \Big(  X_j(\zeta^2 \Ta_{k} u), a_{j{j'}} [\Ta_{k}, X_{j'}] u\Big)_\vp \Big] \nn\\
&= \sum_{j,{j'}=1}^n \big( \zeta^2 X_j \Ta_{k} u, a_{j{j'}} X_{j'} \Ta_{k}  u\big)_\vp +\sum_{j,{j'}=1}^n \Big[ \Big([X_j,\Ta_{k}^*](\zeta^2 \Ta_{k} u), a_{j{j'}} X_{j'} u\Big)_\vp \nn\\
&+ \Big(  X_j(\zeta^2 \Ta_{k} u), \frac{\p a_{j{j'}}}{\p x_{k}} X_{j'} u\Big)_\vp + \Big(  X_j(\zeta^2 \Ta_{k} u), a_{j{j'}} [\Ta_{k}, X_{j'}] u\Big)_\vp
+ \Big(  2\zeta \frac{\p\zeta}{\p x_j} \Ta_{k} u, a_{j{j'}} X_{j'} \Ta_{k}  u\Big)_\vp \Big]. \label{eqn:A bdy, post IBP}
\end{align}
The strong ellipticity condition implies that
\begin{equation}\label{eqn:A strong ellipticity, near bdy}
 \sum_{j,{j'}=1}^n \big( \zeta^2 X_j \Ta_{k} u, a_{j{j'}} X_{j'} \Ta_{k}  u\big)_\vp \geq \theta \| \zeta\nabla_X \Ta_{k} u \|_{L^2(W\cap\Omega_\ep,\vp)}^2.
\end{equation}
As in the proof of Theorem \ref{thm: elliptic regularity, higher order, interior},  the remaining terms of \eqref{eqn:A bdy, post IBP} are bounded by
\[
 C_2 \big(\| \zeta \nabla_X \Ta_{k} u \|_{L^2(W\cap\Omega_\ep,\vp)} \| \nabla_X u \|_{L^2(W\cap\Omega_\ep,\vp)} + \|\nabla_X u\|_{L^2(W\cap\Omega_\ep,\vp)}^2\big)
\]
where $C_2$ depends on $\|a_{j{j'}}\|_{C^1(\Omega)}$, $\|\rho\|_{C^3(\Omega)}$ and $\|\zeta\|_{C^1(\Omega)}$. In particular, $C$ does not depend on the size $\supp\zeta$.
Thus, using (\ref{eqn:A strong ellipticity, near bdy}) and the bounds on the error terms, we can bound
\begin{align}
|A| &\geq \theta \| \zeta\nabla_X \Ta_{k} u \|_{L^2(W\cap\Omega_\ep,\vp)}^2\nn\\
&- C_2(\| \zeta\nabla_X \Ta_{k} u \|_{L^2(W\cap\Omega_\ep,\vp)}\| \zeta\nabla_X u \|_{L^2(W\cap\Omega_\ep,\vp)} +  \|\nabla_X u \|_{L^2(W\cap\Omega_\ep,\vp)}^2 \big) \nn\\
&\geq  \frac\theta2 \| \zeta\nabla_X \Ta_{k} u \|_{L^2(W\cap\Omega_\ep,\vp)}^2
- C_3 \|\nabla_X u \|_{L^2(W\cap\Omega_\ep,\vp)}^2  \label{eqn:good est for A, near bdy}
\end{align}
where $C_3 = C_3(\|a_{j{j'}}\|_{C^1(\Omega)},\|\zeta\|_{C^1(\Omega)},n,\theta, \|\rho\|_{C^3(\Omega)})$.
We can bound $B$ with Cauchy-Schwarz and the small constant/large constant inequality. In particular, for a constant
\[
  C_4 = C_4(\|b_j\|_{L^\infty(\Omega)}, \|b_j'\|_{C^1(\Omega)}, \|b\|_{L^\infty(\Omega)},n,\|\rho\|_{C^2(\Omega)}),
\]
we have the estimate
\begin{align}
|B| &\leq  C_4 \|\zeta\nabla_X \Ta_{k} u\|_{L^2(W\cap\Omega_\ep,\vp)}\Big[ \|f\|_{L^2(W\cap\Omega_\ep,\vp)} + \| u\|_{W^{1,2}(W\cap\Omega_\ep,\vp;X)} \Big] \nn\\
&\leq \frac{\theta}4 \|\zeta\nabla_X \Ta_{k} u\|_{L^2(W\cap\Omega_\ep,\vp)}^2 + C_5\big(  \|f\|_{L^2(W\cap\Omega_\ep,\vp)}^2 + \| u\|_{W^{1,2}(W\cap\Omega_\ep,\vp;X)}^2\big)
 \label{eqn:B estimate, near bdy}
\end{align}
where $C_5 = C_5(\|b_j\|_{L^\infty(\Omega)}, \|b_j'\|_{C^1(\Omega)}, \|b\|_{L^\infty(\Omega)},n,\theta,\|\rho\|_{C^2(\Omega)})$.
Combining (\ref{eqn:good est for A, near bdy}) and (\ref{eqn:B estimate, near bdy}), it follows that
\begin{equation}\label{eqn:W^2 est in terms of W^1 and L^2, near bdy}
\| \tnabla_\Ta u\|_{W^{1,2}(V\cap\Omega_\ep,\vp;X)}^2 \leq C_6\big(\| f\|_{L^2(W\cap\Omega_\ep,\vp)}^2 + \| u \|_{W^{1,2}(W\cap\Omega_\ep,\vp;X)}^2 \big)
\end{equation}
where $C_6 = C_6(\|a_{j{j'}}\|_{C^1(\Omega)},\|\zeta\|_{C^1(\Omega)},\|b_j\|_{L^\infty(\Omega)}, \|b_j'\|_{C^1(\Omega)}, \|b\|_{L^\infty(\Omega)},n,\theta,\|\rho\|_{C^3(\Omega)})$.

We can improve the estimate \eqref{eqn:W^2 est in terms of W^1 and L^2, near bdy} and replace $\| u \|_{W^{1,2}(W,\vp;X)}^2$ with $\| u \|_{L^2(W,\vp)}^2$.
Let $\eta\in C^\infty_c(\Omega_\ep')$ be a cutoff so that $\eta|_{W\cap\Omega_\ep} =1$. Using (\ref{eqn:weak solution, second order part isolated}) and strong ellipticity, we estimate
\begin{align*}
\| \eta \nabla_X u \|_{L^2(\Omega,\vp)}^2
&\leq \frac 1{\theta} \sum_{j,{j'}=1}^n \big(\eta^2 X_j u,a_{j{j'}} X_{j'} u\big)_\vp = \frac{1}{\theta} \big| \big(\eta^2 u, g\big)_\vp\big| \\
&\leq C_7 \|u\|_{L^2(\Omega,\vp)}\big( \|f\|_{L^2(\Omega,\vp)} + \|\eta\nabla_X u\|_{L^2(\Omega,\vp)} + \|u\|_{L^2(\Omega,\vp)}\big),
\end{align*}
where $C_7 = C_7(\|b_j\|_{L^\infty(\Omega)}, \|b_j'\|_{C^1(\Omega)}, \|b\|_{L^\infty(\Omega)},n,\theta)$.

Using a small constant/large constant argument, we have
\[
\| \nabla_X u \|_{L^2(W\cap\Omega_\ep,\vp)} \leq \| \eta\nabla_X u\|_{L^2(\Omega,\vp)} \leq C_8\big(\|f\|_{L^2(\Omega,\vp)} + \|u\|_{L^2(\Omega,\vp)}\big)
\]
where $C_8 = C_8(\|b_j\|_{L^\infty(\Omega)}, \|b_j'\|_{C^1(\Omega)}, \|b\|_{L^\infty(\Omega)},n)$. Thus, we can refine (\ref{eqn:W^2 est in terms of W^1 and L^2, near bdy}) by
\[
\| \tnabla_\Ta u \|_{W^{1,2}(V,\Omega)}^2 \leq C\big(\| f\|_{L^2(\Omega,\vp)} + \| u \|_{L^2(\Omega,\vp)}^2 \big)
\]
where
$C = C(\|a_{j{j'}}\|_{C^1(\Omega)},\|b_j\|_{L^\infty(\Omega)}, \|b_j'\|_{C^1(\Omega)}, \|b\|_{L^\infty(\Omega)},n,\theta,\|\rho\|_{C^3(\Omega)})$. $\zeta$ and $\eta$ have disappeared from
$C$ as the bound depended on $\|\zeta\|_{C^1(\Omega)}$ but that bound depends on $\dist(V,W)$.
Thus, we can relax the boundedness condition on $V$ and let $V$ be as in the statement of the theorem.

It remains to bound $\| \Ta_n^2 u\|_{L^2(\Omega,\vp)}$. To do this, we recall that $\Ta_j = \sum_{j,{j'}=1}^n {\oldb}_{j{j'}} X_{j'}$ where $({\oldb}_{j{j'}})$ is an orthogonal matrix. Therefore, if
$({\oldb}_{j{j'}})^{-1} =  ({\oldb}^{j{j'}})$, then $X_{j'} = \sum_{{k}'=1}^n {\oldb}^{{j'}{k}'} \Ta_{{k}'}$. Therefore,
\[
\sum_{j,{j'}=1}^n (X_j^*)a_{j{j'}}X_{j'} = \sum_{j,{j'},{k},{k}'=1}^n \big(\Ta_{k}^* {\oldb}^{j{k}}a_{j{j'}}{\oldb}^{{j'}{k}'} \Ta_{{k}'}\big).
\]
Let $a_{{k}{k}'}^\T = \sum_{j,{j'}=1}^n\oldb^{j{k}}a_{j{j'}}{\oldb}^{{j'}{k}'}$. Since $({\oldb}^{j{k}})$ is an orthogonal matrix and the smallest eigenvalue of $a_{j{j'}}$ is $\theta$, it follows that
\[
\sum_{{k},{k}'=1}^n a_{{k}{k}'}^\T \xi_{k}\xi_{{k}'} \geq \theta|\xi|^2.
\]
Therefore, if $\xi = (0,...0,1)$, then we see that $a_{nn}^\T \geq \theta$. Consequently, using the fact that $Lu=f$, we see
\[
|\Ta_n^* \Ta_n u|
\leq C_9\bigg( |\nabla_X \tnabla_\Ta u| + |\nabla_X u|+ \sum_{j=1}^n \big( |b_j X_j u| +| X_j^*(b_j' u)|\big) + |bu| + |f| \bigg)
\]
where $C_9 = C_9(\|\rho\|_{C^2(\Omega)},\theta)$. Since the right-hand side is bounded
in $L^2(\Omega,\vp)$, $\Ta_n^* \Ta_n u\in L^2(\Omega,\vp)$. Finally, since
$\|\nabla_X^* \Ta_n u\|_{L^2(\Omega,\vp)}\geq C \|\Ta_n^2 u\|_{L^2(\Omega,\vp)}$, the proof of Theorem \ref{thm:sol'n regularity near boundary, higher order}
for the case $\ell=0$ is complete.
\end{proof}

\subsection{The $\ell\geq 1$ case}
\label{sec:proof_sol'n regularity near boundary, higher order_1}

Before proving the higher order case, we perform a quick computation regarding tangential and nontangential operators.
\begin{lem}\label{lem:commuting tan and nontan ops} Let $X$ be a first order differential operator with coefficients bounded by $\|\rho\|_{C^k(\Omega_\ep')}$ for $k\geq 1$ and let $\Ta_{\alpha_1},\dots, \Ta_{\alpha_\ell}$ be tangential
operators with coefficients bounded by $\|\rho\|_{C^1(\Omega_\ep')}$. If $\Ta^\alpha = \Ta_{\alpha_1}\cdots \Ta_{\alpha_\ell}$, then
\begin{enumerate}\renewcommand{\labelenumi}{\roman{enumi}.}
\item
\[
X \Ta^\alpha = \sum_{\beta\subset\alpha} \Ta^\beta X_\beta
\]
for first order operators $X_\beta$ with coefficients bounded by $\|\rho\|_{C^{k+\ell-|\beta|}(\Omega_\ep')}$.
\item With $X_\beta$ as in \emph{i.},
\[
[X,\Ta^\alpha] = \sum_{\beta\subsetneq\alpha} \Ta^\beta X_\beta
\]
\end{enumerate}
\end{lem}

\begin{proof}The proof is by induction on $\ell$.  When $\ell=1$ this is self-evident for any $k\geq 1$, since $X \Ta=[X,\Ta]+\Ta X$.  Observe that for $2\leq j\leq\ell$,
\[
X \Ta_{\alpha_1}\cdots\Ta_{\alpha_j} = \Ta_{\alpha_1} X \Ta_{\alpha_2}\cdots \Ta_{\alpha_j} + [X,\Ta_{\alpha_1}]  \Ta_{\alpha_2}\cdots \Ta_{\alpha_j}.
\]
If $X$ has coefficients bounded by $\|\rho\|_{C^{k}(\Omega_\ep')}$, then the commutator $[X,\Ta_{\alpha_{1}}]$ is a first order differential operator with coefficients bounded by $\|\rho\|_{C^{k+1}(\Omega_\ep')}$.  If we apply the induction hypothesis with $\ell=j-1$ to both terms, then \emph{(i)} is proved.

The proof of \emph{(ii)} follows from proof of \emph{(i)}.
\end{proof}

\begin{proof}[Proof of Theorem \ref{thm:sol'n regularity near boundary, higher order}, $\ell\geq 1$] This proof is loosely based on the proof of \cite[Theorem 7.29]{Fol95}.
By Theorem \ref{thm: elliptic regularity, higher order, interior} and the classical theory, we know that if $f\in W^{\ell,2}(\Omega,\vp; X)$ and $Lu=f$, then
$u\in W^{\ell+2,2}_{\loc}(\Omega,\vp; X)$. As with the $\ell=0$ case, we can restrict ourselves to $\Omega_\ep$
for $\ep>0$ suitably small. Let $V,W\subset\Omega_\ep'$ be bounded subsets and satisfy $V\subset W$, $\dist(V,\bd W)>0$ and $\dist(W,\bd\Omega_\ep')>0$.
Choose $\zeta\in C^\infty_c(W)$ with $\zeta|_V=1$.

We first induct on the number of tangential derivatives. The base case is already done. The induction hypothesis is that if $\ell\geq 1$ and
$|\beta|\leq \ell$, then there exists a constant $C$ that does not depend on $V$, the size of the support of $\zeta$, or $W$ so that
\begin{equation}\label{eqn:good tangential est}
\| T^\beta u \|_{W^{1,2}(W\cap\Omega_\ep,\vp;X)} \leq C \big( \| f \|_{W^{|\beta|-1,2}(W\cap\Omega_\ep,\vp;X)} + \|u\|_{L^2(\Omega_\ep,\vp)}\big).
\end{equation}

Let $\alpha$ be a multiindex of length $\ell+1$. Let $v\in W^{1,2}(\Omega_\ep,\vp;X)$. We start by showing
\begin{equation}\label{eqn:D(v,T^alpha bound)}
|{\oldD}(v,\Ta^\alpha\zeta u)| \leq C_1 \| v\|_{W^{1,2}(W\cap\Omega_\ep,\vp;X)}\big( \| f\|_{W^{\ell,2}(W\cap\Omega_\ep,\vp;X)} + \| u \|_{L^2(W\cap\Omega_\ep,\vp)}\big)
\end{equation}
where $C_1 = C_1(\|a_{jk}\|_{C^{\ell+1}(\Omega)},\|b_j\|_{C^{\ell+1}(\Omega)}, \|b_j'\|_{C^{\ell+1}(\Omega)}, \|b\|_{C^{\ell+1}(\Omega)},n,\|\zeta\|_{C^{\ell+1}(\Omega_\ep')},\|\rho\|_{C^{\ell+2}(\Omega_\ep')})$.

By Proposition \ref{prop:density of C^infty_c},
there exist $v_\delta\in C^\infty_c(\R^n)$ so that $v_\delta\to v$ in $W^{1,2}(\Omega_\ep,\vp;X)$. Therefore, since $\oldD$ involves at most first order derivatives,
\[
\lim_{\delta\to 0} {\oldD}(v_\delta,\Ta^\alpha\zeta u)  = {\oldD}(v,\Ta^\alpha\zeta u).
\]
We compute
\begin{multline}\label{eqn: D(v,T zeta u)}
{\oldD}(v_\delta,\Ta^\alpha\zeta u)
= \sum_{j,k=1}^n \big(X_j v_\delta, a_{jk} X_k \Ta^\alpha(\zeta u)\big)_\vp\\ + \sum_{j=1}^n\Big[ \big(v_\delta, b_j X_j \Ta^\alpha(\zeta u)\big)_\vp + \big( X_j v_\delta, b_j' \Ta^\alpha(\zeta u)\big)_\vp\Big]
+ \big(v_\delta, b\Ta^\alpha(\zeta u)\big)_\vp.
\end{multline}
We  examine each term separately
\begin{align}
&\big( X_j v_\delta, a_{jk} X_k T^\alpha(\zeta u)\big)_\vp
=  \big( (T^\alpha)^*X_j v_\delta, a_{jk} X_k(\zeta u) \big)_\vp + \big(X_j v_\delta,  [a_{jk}X_k,T^\alpha] (\zeta u)\big)_\vp \nn \\
&= \big( \zeta (T^\alpha)^*X_j v_\delta, a_{jk} X_ku \big)_\vp + \Big( X_j (T^\alpha)^*v_\delta, \frac{\p\zeta}{\p x_k} a_{jk} u \Big)_\vp+ \big(X_j v_\delta,  [a_{jk}X_k,T^\alpha] (\zeta u)\big)_\vp \nn \\
&=  \Big(  X_j \big(\zeta(T^\alpha)^*v_\delta\big), a_{jk} X_k u \Big)_\vp + \big( [\zeta (T^\alpha)^*,X_j] v_\delta, a_{jk} X_ku \big)_\vp
+ \Big( X_j (T^\alpha)^*v_\delta, \frac{\p\zeta}{\p x_k} a_{jk} u \Big)_\vp  \label{eqn: X_j v, X_k (T^alpha zeta u) IBP} \\
&\hspace{.45cm}+ \big(X_j v_\delta,  [a_{jk}X_k,T^\alpha] (\zeta u)\big)_\vp. \nn
\end{align}
Next,
\begin{align}
\big(v_\delta, b_j X_j T^\alpha(\zeta u)\big)_\vp
&= \big((T^\alpha)^* v_\delta, b_j  X_j (\zeta u)\big)_\vp + \big(v_\delta,  [b_j X_j, T^\alpha](\zeta u)\big)_\vp \nn\\
&=  \big(\zeta(T^\alpha)^* v_\delta, b_j  X_j  u\big)_\vp + \Big((T^\alpha)^* v_\delta, b_j  \frac{\p\zeta}{\p x_j} u\Big)_\vp+ \big(v_\delta,  [b_j X_j, T^\alpha](\zeta u)\big)_\vp.
\label{eqn:(v, b_j X_j T^alpha zeta u ) IBP}
\end{align}
Also,
\begin{align}
\big( X_j v_\delta, b_j'T^\alpha(\zeta u)\big)_\vp &= \big(\zeta (T^\alpha)^* X_j v_\delta, b_j' u\big)_\vp + \big(X_j v_\delta, [b_j',T^\alpha](\zeta u)\big)_\vp \nn\\
&= \Big( X_j \big(\zeta(T^\alpha)^* v_\delta),b_j'u\Big)_\vp + \big([\zeta(T^\alpha)^*,X_j] v_\delta,b_j'u\big)_\vp + \big(X_j v_\delta, [b_j',T^\alpha](\zeta u)\big)_\vp
\label{eqn: ( X_j v, b_j' T^alpha zeta u) IBP}
\end{align}
Plugging  \eqref{eqn: X_j v, X_k (T^alpha zeta u) IBP},  \eqref{eqn:(v, b_j X_j T^alpha zeta u ) IBP} and \eqref{eqn: ( X_j v, b_j' T^alpha zeta u) IBP} into \eqref{eqn: D(v,T zeta u)}, we see
\begin{align*}
{\oldD}(v_\delta,\Ta^\alpha\zeta u)
&= {\oldD}(\zeta (\Ta^\alpha)^* v_\delta, u) + E
= \big( \zeta (\Ta^\alpha)^* v_\delta, f\big)_\vp + E \nn \\
&= \big( \Ta_{\alpha_1}^* v_\delta, \Ta_{\alpha_2}\cdots \Ta_{\alpha_{\ell+1}} ( \zeta f)\big)_\vp  + E 
\end{align*}
where
\begin{align*}
E &= \sum_{j,k=1}^n\bigg[ \big( [\zeta (T^\alpha)^*,X_j] v_\delta, a_{jk} X_ku \big)_\vp
+ \Big( X_j (T^\alpha)^*v_\delta, \frac{\p\zeta}{\p x_k} a_{jk} u \Big)_\vp + \big(X_j v_\delta,  [a_{jk}X_k,T^\alpha] (\zeta u)\big)_\vp \bigg] \\
&+\sum_{j=1}^n \bigg[  \Big((T^\alpha)^* v_\delta, b_j  \frac{\p\zeta}{\p x_j} u\Big)_\vp+ \big(v_\delta,  [b_j X_j, T^\alpha](\zeta u)\big)_\vp
+  \big([\zeta(T^\alpha)^*,X_j] v_\delta,b_j'u\big)_\vp\bigg]\\
& + \sum_{j=1}^n\big(X_j v_\delta, [b_j',T^\alpha](\zeta u)\big)_\vp
+\big(v_\delta,[b,T^\alpha](\zeta u)\big)_\vp.
\end{align*}
Since $[\zeta,(\Ta^\alpha)^*]$ is tangential and $[\zeta,(\Ta^\alpha)^*]^* = [\Ta^\alpha,\zeta]$ is a differential operator of order $\ell$, by the induction hypothesis
\begin{equation}
|{\oldD}(v_\delta,\Ta^\alpha\zeta u)|  \leq C_2 \| v_\delta \|_{W^{1,2}(W\cap\Omega_\ep,\vp;X)} \| f \|_{W^{\ell,2}(W\cap\Omega_\ep,\vp;X)}  + |E|
\label{eqn:D(v, T u) est, post IBP}
\end{equation}
where $C_2 = C_2(\|\zeta\|_{C^\ell(\Omega_\ep')},\|\rho\|_{C^{\ell}(\Omega_\ep')})$.
We turn our attention to $E$. Using integration by parts, $[(T^\alpha)^*,X]^*=[X,T^\alpha]$ (formally), and Lemma \ref{lem:commuting tan and nontan ops} (with $k=1$, since the result is not improved when the coefficients of $X$ are smooth), it follows from the induction hypothesis that
\begin{align}
|E| &\leq C_3 \| v_\delta \|_{W^{1,2}(W\cap\Omega_\ep,\vp;X)} \sum_{|\beta|\leq\ell}\| \Ta^\beta u\|_{W^{1,2}(W\cap\Omega_\ep,\vp;X)}\nn \\
&\leq C_4 \| v_\delta \|_{W^{1,2}(W\cap\Omega_\ep,\vp;X)}\big( \| f\|_{W^{\ell-1,2}(W\cap\Omega_\ep,\vp;X)} +  \| u \|_{L^2(W\cap\Omega_\ep,\vp)} \big)
\label{eqn: E IBP}
\end{align}
where
\[
  C_4 = C_4(\|a_{jk}\|_{C^{\ell+1}(\Omega)},\|b_j\|_{C^{\ell+1}(\Omega)}, \|b_j'\|_{C^{\ell+1}(\Omega)}, \|b\|_{C^{\ell+1}(\Omega)},n,\|\zeta\|_{C^{\ell+1}(\Omega_\ep')},\|\rho\|_{C^{\ell+3}(\Omega_\eps')}).
\]
By plugging (\ref{eqn: E IBP}) into \eqref{eqn:D(v, T u) est, post IBP} and letting $\delta\to 0$, we observe that \eqref{eqn:D(v,T^alpha bound)} has been verified.

Since $\Ta^\alpha(\zeta u)\in\X$, we can set $v = \Ta^\alpha(\zeta u)$ in \eqref{eqn:D(v,T^alpha bound)} and use the coercive estimate \eqref{eqn:coercive estimate} to obtain
\begin{align*}
\| \Ta^\alpha (\zeta u) \|_{W^{1,2}(\Omega_\ep,\vp;X)}^2
&\leq C_5\big( |{\oldD}( \Ta^\alpha (\zeta u),\Ta^\alpha (\zeta u))| + \| \Ta^\alpha (\zeta u) \|_{L^2(\Omega_\ep,\vp)}^2\big) \\
&\leq C_6\| \Ta^\alpha (\zeta u)\|_{W^{1,2}(\Omega_\ep,\vp;X)} \big( \| f \|_{W^{\ell,2}(\Omega_\ep,\vp;X)}+ \| \Ta^\alpha (\zeta u) \|_{L^2(\Omega_\ep,\vp)}\big)\\&+
\| \Ta^\alpha (\zeta u) \|_{L^2(\Omega_\ep,\vp)}^2
\end{align*}
where
\[
  C_6 = C_6(\|a_{jk}\|_{C^{\ell+1}(\Omega)},\|b_j\|_{C^{\ell+1}(\Omega)}, \|b_j'\|_{C^{\ell+1}(\Omega)}, \|b\|_{C^{\ell+1}(\Omega)},n,\theta,
\|\zeta\|_{C^\ell(\Omega_\ep')},\|\rho\|_{C^{\ell+3}(\Omega_\eps')}).
\]
Applying a small constant/large constant argument and the induction hypothesis (\ref{eqn:good tangential est}) for $|\beta|\leq \ell$, we can finish the
proof of (\ref{eqn:good tangential est}) for $|\alpha|=\ell+1$.

We now need to lift the restriction that $\Ta^\alpha$ is tangential.
Without loss of generality, we may assume that $|\alpha|=\ell+2$ and $\Ta^\alpha = \Ta^\beta \Ta_n^\gamma$ where $\Ta^\beta$ is tangential. We will show that
there exists a constant $C_7$ so that
\begin{equation}\label{eqn:T^gamma u ok, up to m+2}
\| \Ta^\alpha u \|_{L^2(V\cap\Omega_\ep,\vp)} \leq C_7\big( \| f \|_{W^{\ell,2}(W\cap \Omega_\ep,\vp;X)} +  \| u \|_{L^2(W\cap \Omega_\ep,\vp)} \big)
\end{equation}
where
\[
  C_7 = C_7(\|a_{jk}\|_{C^{\ell+1}(\Omega)},\|b_j\|_{C^{\ell+1}(\Omega)}, \|b_j'\|_{C^{\ell+1}(\Omega)}, \|b\|_{C^{\ell+1}(\Omega)},n,\|\zeta\|_{C^{\ell+1}(\Omega_\ep')},\|\rho\|_{C^{\ell+3}(\Omega_\ep')}).
\]
The $\gamma=0$ case follows from \eqref{eqn:good tangential est}. Similarly, since the commutator $[\Ta_j,\Ta_n]$ is a first-order operator, we can write the $\gamma=1$ case
as
\[
\Ta^\alpha = \Ta_n \Ta^\beta + \text{lower order tangential terms}
\]
and the estimate again follows from \eqref{eqn:good tangential est}. We prove the $\gamma\geq 2$ case with an induction argument. Assume now that
\eqref{eqn:T^gamma u ok, up to m+2} holds for $\gamma = 0,\dots, J-1$ with $J\geq 2$. Assume that $|\gamma|=J$.
Redefine $\beta$ so that $T^\alpha = \Ta^\beta \Ta_n^2$. Note that $\Ta^\beta$ contains at most $(J-1)$ occurrences of $\Ta_n$.
Since $u\in W^{\ell+2}_{\loc}(\Omega)$ and $Lu=f$ in $\Omega$, we have $\Ta^\beta Lu = \Ta^\beta f$ a.e.\ in $\Omega$. We can write
\begin{multline*}
\Ta^\beta f = \Ta^\beta Lu\\
 = a_{nn}\Ta^\alpha u + \text{terms involving $\Ta_n$ at most $J-1$ times and of order at most $\ell+2$}.
\end{multline*}
Since $a_{nn}\geq\theta>0$, by the induction hypothesis and \eqref{eqn:T^gamma u ok, up to m+2}, it follows that
\[
\| \Ta^\gamma u \|_{L^2(V\cap\Omega_\ep,\vp)} \leq C_8\big( \| f\|_{W^{\ell,2}(W\cap\Omega_\ep,\vp;X)} + \| u \|_{L^2(W\cap\Omega_\ep,\vp)}\big).
\]
where $C_8 = C_8(\|a_{jk}\|_{C^{\ell+1}(\Omega)},\|b_j\|_{C^{\ell+1}(\Omega)}, \|b_j'\|_{C^{\ell+1}(\Omega)}, \|b\|_{C^{\ell+1}(\Omega)},n,\theta, \|\rho\|_{C^{\ell+3}(\Omega_\ep')})$.

Since the constant $C_8$ does not depend on the size of $V$, the estimate holds for all $V$ and hence
\[
\| u \|_{W^{\ell+2,2}(\Omega_\ep,\vp;X)} \leq C_8\big( \| f\|_{W^{\ell,2}(\Omega_\ep,\vp;X)} + \| u \|_{W^{1,2}(\Omega_\ep,\vp;X)}\big).
\]
\end{proof}

%
%
\section{Traces of $L$-harmonic functions}
\label{sec:traces_harmonic_functions}
In this section, we wish to show that $L$-harmonic functions (i.e., functions $u$ so that $Lu=0$) have unique boundary values in $W^{s-1/2,2}(\bd\Omega,\vp;T)$
when $u\in W^{s,2}(\Omega,\vp; X)$ and $s\geq 0$. 

We first establish a simple but easily applicable uniqueness condition by proving Lemma \ref{lem:D dominating 1 norms makes V=0}.
\begin{proof}[Proof Lemma \ref{lem:D dominating 1 norms makes V=0}]Since $Lu=0$ and $u\in W^{1,2}_0(\Omega,\vp; X)$, it follows that
\[
\Rre {\oldD}(u,u) = \Rre(u,Lu)_\vp = 0.
\]
Since $\Rre {\oldD}(u,u)\geq c \|\nabla_X u\|_{L^2(\Omega,\vp)}$, it follows that $\nabla_X u=0$. By Corollary \ref{cor:adjoint dominated in L^2},
$\|\nabla u\|_{L^2(\Omega,\vp)} \les \|\nabla_X u\|_{L^2(\Omega,\vp)} =0$. Therefore, $\nabla u=0$ and $u$ is constant (on each component of $\Omega$). Since $u|_M=0$, $u\equiv 0$.
\end{proof}

\subsection{The $s\geq 2$ case in Theorem \ref{thm:L times Tr is an isomorphism}}
\label{sec:proof_L times Tr is an isomorphism_2}
\begin{lem}\label{lem:traces of L-harmonic functions for s geq 3/2}
Let $\Omega\subset\R^n$ be a domain that satisfies {(HI)}-{(HVI)} with $m\geq 3$.
Let $L$ be a strongly elliptic operator that has a Dirichlet form $\oldD$ that satisfies \eqref{eqn:D dominates 1 norm} for all  $u\in W^{1,2}_0(\Omega,\vp; X)$. Let
$2\leq k \leq m-1$ be an integer. Then there is a one-to-one correspondence between $B^{k-1/2;2,2}(M,\vp;T)$ and $W^{k,2}(\Omega,\vp; X) \cap \ker L$ with norm equivalence.
 \end{lem}

\begin{proof} Assume that $U\in W^{k,2}(\Omega,\vp; X)$ and $LU=0$. Since $U\in W^{k,2}(\Omega,\vp; X)$, Theorem \ref{thm:Trace Theorem for integer m on M}
implies that $\Tr U \in B^{k-1/2;2,2}(M,\vp;T)$. Since $L$ satisfies the hypotheses of Lemma \ref{lem:D dominating 1 norms makes V=0}, $U$ is the unique function
in $W^{k,2}(\Omega,\vp; X) \cap \ker L$ with boundary value $\Tr U$.

Now assume that $u\in B^{k-1/2;2,2}(M,\vp;T)$. By  Theorem \ref{thm:Trace Theorem for integer m on M}, there exists a function $\tilde U \in W^{k,2}(\Omega,\vp; X)$
with boundary value $u$ and
\[
\| \tilde U \|_{W^{k,2}(\Omega,\vp; X)} \leq C \|  u \|_{ B^{k-1/2;2,2}(M,\vp;T)}.
\]
Since $ k\geq 2$, $L\tilde U \in W^{k-2,2}(\Omega,\vp; X)$. By Theorem \ref{thm:weak solutions, info about kernel}, there exists
$U_0\in W^{1,2}_0(\Omega,\vp; X)$ so that ${\oldD}(v,U_0) = (v,L\tilde U)_\vp$ for all $v\in W^{1,2}_0(\Omega,\vp; X)$. Since $L$ satisfies \eqref{eqn:D dominates 1 norm},
$U_0$ is unique. By Theorem \ref{thm:sol'n regularity near boundary, higher order}, $U_0\in W^{k,2}(\Omega,\vp; X)$.
Moreover, the mapping
\[
  L:W^{k,2}(\Omega,\vp; X)\cap W^{1,2}_0(\Omega,\vp; X) \to W^{k-2,2}(\Omega,\vp; X)
\]
is a bijective linear mapping, so the Open Mapping Theorem (or, more directly,
its corollary the Bounded Inverse Theorem) prove that its inverse is continuous, i.e.,
\[
\| U_0 \|_{W^{k,2}(\Omega,\vp; X)} \leq C \| L\tilde U \|_{W^{k-2,2}(\Omega,\vp; X)}
\leq C \| \tilde U \|_{W^{k,2}(\Omega,\vp; X)} \leq C \| u \|_{ B^{k-1/2;2,2}(M,\vp;T)}.
\]
Let $U = \tilde U - U_0$. Then $LU = 0$ and $\Tr U = \Tr \tilde U = u$ and
\[
\| U  \|_{W^{k,2}(\Omega,\vp; X)} \leq C\| u \|_{ B^{k-1/2;2,2}(M,\vp;T)}.
\]
\end{proof}

\subsection{The case $s=1$ in Theorem \ref{thm:L times Tr is an isomorphism}}
\label{sec:proof_L times Tr is an isomorphism_1}
We use the arguments in \cite{Tay96} for the following.
\begin{thm}\label{thm:regularity at the ground level}
Let $L$ be a strongly elliptic operator and $S$ be a first order operator with bounded coefficients. Set
\[
Au = Lu + Su.
\]
There exists a constant $C>0$ so that for all $u\in W^{1,2}_0(\Omega,\vp; X)$,
\[
\| u \|_{W^{1,2}(\Omega,\vp; X)}^2 \leq C \| Au \|_{W^{-1,2}(\Omega,\vp; X)}^2 + C \|u\|_{L^2(\Omega,\vp)}^2.
\]
\end{thm}

\begin{proof}
Observe that for any $\eps>0$
\[
|(u,Su)_\vp| \leq C \|u\|_{L^2(\Omega,\vp)} \| u\|_{W^{1,2}(\Omega,\vp; X)} \leq \frac C2 \Big( \eps\| u \|_{W^{1,2}(\Omega,\vp; X)}^2 + \frac 1\eps \|u\|_{L^2(\Omega,\vp)}^2\Big).
\]
Therefore,
\[
\Rre (u,Au)_\vp \geq \frac\theta2 \| u \|_{W^{1,2}(\Omega,\vp; X)}^2 - C' \|u\|_{L^2(\Omega,\vp)}^2
\text{ for all }u\in W^{1,2}_0(\Omega,\vp; X),
\]
so
\[
\| u \|_{W^{1,2}(\Omega,\vp; X)}^2 \leq C\Rre(u,Au)_\vp + C' \|u\|_{L^2(\Omega,\vp)}^2.
\]
Also,
\[
\Rre(u,Au)_\vp \leq C \| Au\|_{W^{-1,2}(\Omega,\vp; X)} \| u \|_{W^{1,2}(\Omega,\vp; X)}
\leq \frac{C\eps}{2}  \| u \|_{W^{1,2}(\Omega,\vp; X)}^2 + \frac{C}{2\eps} \| Au\|_{W^{-1,2}(\Omega,\vp; X)}^2.
\]
Putting our inequalities together and choosing $\ep>0$ small enough so that we can absorb the $C\ep  \| u \|_{W^{1,2}(\Omega,\vp; X)}^2$ term, we see that
\[
\| u \|_{W^{1,2}(\Omega,\vp; X)}^2 \leq C \| Au \|_{W^{-1,2}(\Omega,\vp; X)}^2 + C \|u\|_{L^2(\Omega,\vp)}^2.
\]
\end{proof}

We next show  that
$L: W^{1,2}_0(\Omega,\vp; X)\to W^{-1,2}(\Omega,\vp; X)$ is continuous, injective and has a bounded inverse.

We first assume that $L$ gives rise to a strictly elliptic Dirichlet form over $W^{1,2}_0(\Omega,\vp; X)$. Then
\[
\Rre (v,Lu)_\vp = \Rre {\oldD}(v,u) \geq C \| v\|_{W^{1,2}(\Omega,\vp; X)}\| u\|_{W^{1,2}(\Omega,\vp; X)}.
\]
Consequently, $L:W^{1,2}_0(\Omega,\vp; X) \to W^{-1,2}(\Omega,\vp; X)$ and
\[
\| Lu \|_{W^{-1,2}(\Omega,\vp; X)} \geq C \| u \|_{W^{1,2}(\Omega,\vp; X)}.
\]
Therefore, $L:W^{1,2}_0(\Omega,\vp; X) \to W^{-1,2}(\Omega,\vp; X)$  has closed range. If $L$ is not surjective, there exists
a nonzero $v^*\in W^{-1,2}(\Omega,\vp; X)$
so that $v^* \perp \Ran(L)$. By the Riesz Representation Theorem, we can therefore choose $v\in W^{1,2}_0(\Omega,\vp; X)$ so that $v^*(v)\neq 0$ and $w^*(v)=0$ for $w^*\in\Ran(L)$.
In this case
\[
0 = (v, Lu)_\vp \text{ for all }u\in W^{1,2}_0(\Omega,\vp; X).
\]
Setting $u=v$ forces $v=0$ (and hence $v^*=0$ as well).
Therefore, $L$ is surjective. We also know that $L$ is injective as a consequence of Lemma \ref{lem:D dominating 1 norms makes V=0}.
Consequently, the inverse to $L$ exists, call it $G$. Then $G:W^{-1,2}(\Omega,\vp; X)\to W^{1,2}_0(\Omega,\vp; X)$. As $L^2(\Omega,\vp)\hookrightarrow
W^{-1,2}(\Omega,\vp; X)$ compactly, $G:L^2(\Omega,\vp)\to W^{1,2}_0(\Omega,\vp; X)$ compactly.

We now investigate the equation
\[
Au =f
\]
where $f\in W^{-1,2}(\Omega,\vp; X)$, $u\in W^{1,2}_0(\Omega,\vp; X)$, and
$A = L+S$ as in Theorem \ref{thm:regularity at the ground level}. We continue to assume that $L$ has a strictly elliptic Dirichlet form over
$W^{1,2}_0(\Omega,\vp;X)$. If $u\in W^{1,2}_0(\Omega,\vp; X)$, then there exists $v\in W^{-1,2}(\Omega,\vp; X)$ so that
\[
u = Gv
\]
If $Au=f$, then
\[
f = AGv = (L+S)Gv = (I+SG)v.
\]
We know that $SG:W^{-1,2}(\Omega,\vp; X)\to L^2(\Omega,\vp)$ and $L^2(\Omega,\vp)\hookrightarrow W^{-1,2}(\Omega,\vp; X)$ is compact. Therefore
$I+SG : W^{-1,2}(\Omega,\vp; X)\to W^{-1,2}(\Omega,\vp; X)$ is a compact perturbation of the identity. The Fredholm alternative implies that the map
$I+SG$ is therefore surjective if and only if it is injective. Lemma \ref{lem:D dominating 1 norms makes V=0} supplies a condition that guarantees injectivity.

Since the difference between a strongly elliptic operator and a strongly elliptic operator that gives rise to a strictly elliptic Dirichlet form is the addition of a multiple of the identity,
the case of relevance is $S = \lambda I$
for some $\lambda\in\R$.  If $Lu=v\neq 0$, then
\[
  (L+\lambda I)u=(L+\lambda I)Gv=(I+\lambda G)v\neq 0
\]
since $I+\lambda G$ is injective.  We have  therefore proved the following.
\begin{prop} \label{prop:W^1,2_0 to W^-1 isom}
Let $L$ be a strongly elliptic operator that has a Dirichlet form that satisfies (\ref{eqn:D dominates 1 norm}). Then the map
\[
L: W^{1,2}_0(\Omega,\vp; X) \to W^{-1,2}(\Omega,\vp; X)
\]
is an isomorphism with norm equivalence.
\end{prop}

With regard to the the norm equivalence, it follows immediately that $\| Lu \|_{W^{-1,2}(\Omega,\vp; X)} \leq \| u \|_{W^{1,2}(\Omega,\vp; X)}$. The reverse inequality follows from
the Bounded Inverse Theorem.  We are now in a position to improve Lemma \ref{lem:traces of L-harmonic functions for s geq 3/2}.
\begin{lem}\label{lem:traces of L-harmonic functions for s = 1/2}
Let $\Omega\subset\R^n$ be a domain that satisfies {(HI)}-{(HVI)} for $m=2$.
Let $L$ be a strongly elliptic operator that has a Dirichlet form $\oldD$ which satisfies (\ref{eqn:D dominates 1 norm}).
There is a one-to-one correspondence between $B^{1/2;2,2}(M,\vp;T)$ and $W^{1,2}(\Omega,\vp; X) \cap \ker L$ with norm equivalence.
\end{lem}

\begin{proof} We already know that $\Tr :  W^{1,2}(\Omega,\vp; X) \to B^{1/2;2,2}(M,\vp;T)$ is continuous. Now let $f\in  B^{1/2;2,2}(M,\vp;T)$. By
Theorem \ref{thm:Trace Theorem for integer m on M}, there exists $F\in W^{1,2}(\Omega,\vp; X)$ so that $\Tr F = f$ and
\[
\| F \|_{W^{1,2}(\Omega,\vp; X)} \leq C \| f \|_{B^{1/2;2,2}(\Omega,\vp; X)}.
\]
Solving $Lu=0$ in $\Omega$ and $\Tr u = f$ is equivalent to finding $v\in W^{1,2}_0(\Omega,\vp; X)$ where $Lv = -LF$ because we could then set $u=F+v$ and it would follow
from Proposition \ref{prop:W^1,2_0 to W^-1 isom} that
\begin{align*}
\| u \|_{W^{1,2}(\Omega,\vp; X)} \leq \| F \|_{W^{1,2}(\Omega,\vp; X)} + \| v \|_{W^{1,2}(\Omega,\vp; X)}
&\leq C\Big( \| f \|_{B^{1/2:2,2}(\Omega,\vp; X)} + \| Lv \|_{W^{-1,2}(\Omega,\vp; X)}\Big) \\
&\leq C\| f \|_{B^{1/2:2,2}(\Omega,\vp; X)}.
\end{align*}
However, $-LF \in W^{-1,2}(\Omega,\vp; X)$ so such a $v$ exists by Proposition \ref{prop:W^1,2_0 to W^-1 isom}.
\end{proof}
Combining our results, we can prove Theorem \ref{thm:L times Tr is an isomorphism}.
\begin{proof}[Proof of Theorem \ref{thm:L times Tr is an isomorphism}]
 Let $f\in W^{s-2,2}(\Omega,\vp; X)$ and $g\in W^{s-1/2,2}(\bd\Omega,\vp;T)$. By Theorem \ref{thm:weak solutions, info about kernel} and
Theorem \ref{thm:sol'n regularity near boundary, higher order}, there
exists a unique $u_1\in W^{s,2}(\Omega,\vp; X)\cap W^{1,2}_0(\Omega,\vp; X)$ so that $Lu_1=f$. If
$G:W^{s-2,2}(\Omega,\vp; X) \to W^{s,2}(\Omega,\vp; X) \cap W^{1,2}_0(\Omega,\vp; X)$ is the inverse to
$L$, then $G$ is continuous, i.e., there exists a constant $C$ so that $\| G f\|_{W^{s,2}(\bd\Omega,\vp;T)} \leq C \| f \|_{W^{s-2,2}(\Omega,\vp; X)}$. Plugging in $f = Lu_1$, we see that
$\| u_1\|_{W^{s,2}(\Omega,\vp; X)}\les \| f\|_{W^{s-2}(\Omega,\vp; X)}$.
Also, by
Lemma \ref{lem:traces of L-harmonic functions for s = 1/2} and Lemma \ref{lem:traces of L-harmonic functions for s geq 3/2}, there exists a unique $u_2\in W^{s,2}(\Omega,\vp; X)$ so
that $Lu_2=0$ and $\Tr u_2=g$. Also, $u_2$ satisfies  $\| u_2\|_{W^{s,2}(\Omega,\vp; X)}\les \| g\|_{W^{s-1/2}(\bd\Omega,\vp;T)}$.
Thus, $u = u_1+u_2$ is the unique function in $W^{s,2}(\Omega,\vp; X)$ so that
\[
\begin{cases} Lu =f &\text{in }\Omega \\
\Tr u =g &\text{on }\bd\Omega
\end{cases}
\]
and
\[
\| u \|_{W^{s,2}(\Omega,\vp; X)} \leq C \Big( \| f \|_{W^{s-2}(\Omega,\vp; X)} + \| g \|_{W^{s-1/2}(\bd\Omega,\vp;T)}\Big)
\]
for a constant $C$ independent of $u$, $f$, and $g$.

In the reverse direction, let $u\in W^{s,2}(\Omega,\vp; X)$. There exists a unique $u_1$ so that $u_1\in W^{s,2}(\Omega,\vp; X)\cap W^{1,2}_0(\Omega,\vp; X)$,
$Lu = Lu_1$, and
\[
\| u_1 \|_{W^{s,2}(\Omega,\vp; X)} \les \| Lu_1 \|_{W^{s-2,2}(\Omega,\vp; X)} \leq \| u \|_{W^{s,2}(\Omega,\vp; X)}.
\]
If $u_2 = u-u_1$, then $u = u_1+u_2$, $Lu_2=0$, and we have already established that
$\| u_2 \|_{W^{s,2}(\Omega,\vp; X)} \sim \| \Tr u_2 \|_{W^{s-1/2}(\Omega,\vp; X)}$. Thus, we have a unique decomposition $u = u_1+u_2$ and
\begin{multline*}
\| u \|_{W^{s,2}(\Omega,\vp; X)} \leq \| u_1 \|_{W^{s,2}(\Omega,\vp; X)}  + \| u_2 \|_{W^{s,2}(\Omega,\vp; X)}
\les \| Lu_1 \|_{W^{s-2,2}(\Omega,\vp; X)} + \| \Tr u_2 \|_{W^{s-1/2}(\Omega,\vp; X)} \\
\les  \| u \|_{W^{s,2}(\Omega,\vp; X)} + \| u_2 \|_{W^{s,2}(\Omega,\vp; X)} \les \| u \|_{W^{s,2}(\Omega,\vp; X)}
\end{multline*}
where the last inequality uses the fact that $u_2 = u-u_1$.
\end{proof}

\subsection{Proof of Theorem \ref{thm:traces of L-harmonic functions}}
\label{sec:proof_traces of L-harmonic functions}
In this subsection, we prove that functions $f\in L^2(\Omega,\vp)$ that are $L$-harmonic have traces in $B^{-1/2;2,2}(\Omega,\vp; X)$. Our motivation for the trace definition is
from \cite{BoCh}.
If we define the operator $S$ by
\[
S = -\sum_{j,k=1}^n \overline{a_{jk}} \frac{\p \rho}{\p x_k} X_j, 
\]
then for $v\in W^{2,2}(\Omega,\vp; X) \cap W^{1,2}_0(\Omega,\vp; X)$ and $\psi\in W^{2,2}(\Omega,\vp; X)$
\[
\big(L^*v, \psi\big)_\vp = \int_{\bd\Omega} \Tr Sv \overline{\Tr \psi}\, e^{-\vp}\, d\sigma + {\oldD}(v,\psi) = \int_{\bd\Omega} \Tr Sv \overline{\Tr \psi}\, e^{-\vp}\, d\sigma  + \big(v,L \psi\big)_\vp.
\]
If $L\psi =0$, then
\begin{equation}\label{eqn:trace formula def'n}
\big(L^*v, \psi\big)_\vp = \int_{\bd\Omega} \Tr Sv\, \overline{\Tr \psi}\, e^{-\vp}\, d\sigma.
\end{equation}
Since $v \in W^{2,2}(\Omega,\vp; X)$, we have $Sv \in W^{1,2}(\Omega,\vp; X)$ and $\Tr Sv \in B^{1/2;2,2}(\bd\Omega,\vp;T)$.

We would like to show a partial converse to the argument, i.e., that if  $\vartheta\in B^{1/2;2,2}(\Omega,\vp; X)$, then $\vartheta = \Tr Sv$ for some
$v\in W^{1,2}_0(\Omega,\vp; X)\cap W^{2,2}(\Omega,\vp; X)$.

Our goal is to show that if $f\in L^2(\Omega,\vp)$ and $Lf=0$, then there exists a well-defined $g\in B^{-1/2;2,2}(\bd\Omega,\vp;T)$ so that
$\Tr f = g$. Equation (\ref{eqn:trace formula def'n}) is the key. Motivated by Theorem \ref{thm:traces of normal derivatives}, we investigate operators $L$ of the form in \eqref{eqn:L has S normal}.
To define an element  $g \in B^{-1/2;2,2}(\bd\Omega,\vp;T)$, it suffices to determine the action of $g$ on elements $\psi \in B^{1/2;2,2}(\bd\Omega,\vp;T)$.
Let  $f\in L^2(\Omega,\vp)$ satisfy $Lf=0$, and let
$\psi \in B^{1/2;2,2}(\bd\Omega,\vp;T)$. From Theorem \ref{thm:traces of normal derivatives},
there exists a (nonunique) element $v\in W^{2,2}(\Omega,\vp; X)\cap W^{1,2}_0(\Omega,\vp; X)$ so that
\[
\frac{\p v}{\p \nu} = \psi \text{ on }\bd\Omega.
\]
Define $ \Tr f$ by
\begin{equation}\label{eqn:Tr f defined}
\langle \Tr f , \psi \rangle := (L^*v,f)_\vp.
\end{equation}
Observe that
\[
|(L^*v,f)_\vp| \les \| v \|_{W^{2,2}(\Omega,\vp; X)} \|f\|_{L^2(\Omega,\vp)} \les \| \psi \|_{B^{1/2;2,2}(\bd\Omega,\vp;T)} \|f\|_{L^2(\Omega,\vp)}.
\]
That $\Tr f$ is well-defined follows from approximating $f$ by functions in $W^{2,2}(\Omega,\vp; X)$ and following the argument that leads to
(\ref{eqn:trace formula def'n}).
In particular, if $\eta_j\to f$ in $L^2(\Omega,\vp)$ and $\eta_j\in W^{2,2}(\Omega,\vp; X)$, then
$L\eta_j\to Lf=0$ in $W^{-2,2}(\Omega,\vp; X)$. We need to show that $L\eta_j\to 0$ in $L^2(\Omega,\vp)$ so we can achieve (\ref{eqn:trace formula def'n}).
$C^\infty_c(\Omega)$ is dense in $L^2(\Omega,\vp)$, and
if $\zeta\in C^\infty_c(\Omega)$, then
\[
( L\eta_j, \zeta)_\vp = (\eta_j, L^* \zeta)_\vp \longrightarrow (f, L^*\zeta)_\vp = (Lf,\zeta)_\vp
\]
where the last equality follows from the pairing of $f$ as a distribution against the test function $\zeta$.

Thus $\Tr f$ is a well-defined element of $B^{-1/2;2,2}(\bd\Omega,\vp;T)$. The use of the name trace is appropriate because if
$f\in L^2(\Omega,\vp)\cap \ker L$ and
has enough regularity so that (\ref{eqn:trace formula def'n}) applies, then the two definitions of $\Tr f$ agree.
Thus we have proven Theorem \ref{thm:traces of L-harmonic functions}.

\appendix
%
%
\section{Background on interpolation -- the real method}

\subsection{The Bochner Integral}
Our discussion of the real interpolation method closely follows \cite{AdFo03}.

\subsection{$L^q$-spaces} Let $X$ be $\mathbb{R}$ or $\mathbb{C}$.
For $1\leq q \leq \infty$, let $L^q(a,b;d\mu(t))$ be the space of functions $f:(a,b)\to X$ such that the norm
\[
\| f; L^q(a,b;d\mu(t),X) \| = \begin{cases} \displaystyle \Big( \int_a^b \|f(t)\|^q_X\, d\mu(t) \Big)^{1/q} & 1 \leq q < \infty \\
\displaystyle \esssup_{a<t<b} \|f(t)\|_X & q=\infty \end{cases}
\]
is finite.

We focus on the special case where $d\mu = dt/t$. We denote $L^q(a,b;d\mu) = L^q_*$.

Let $X_0$ and $X_1$ be two Banach spaces that are continuously imbedded on a Hausdorff topological vector space $\mathfrak X$ and whose
intersection is nontrivial. Such a pair  of Banach spaces $\{X_0,X_1\}$ is called an \bfem{interpolation pair}, and we now turn to the construction of
Banach spaces $X$ suitably intermediate between $X_0$ and $X_1$. It is often the case that $X_1 \imbed X_0$, e.g.,
$X_0 = L^p(\Omega,\vp)$ and $X_1 = W^{m,p}(\Omega,\vp; X)$.

Let $\|\cdot \|_{X_j}$ denote the norm in $X_j$, $j=0,1$. The spaces $X_0 \cap X_1$ and $X_0 + X_1 = \{ u = u_0+u_1: u_0\in X_0,\ u_1 \in X_0\}$
are Banach spaces with norms
\[
\| u \|_{X_0\cap X_1} = \max \{ \|u_0\|_{X_0}, \|u_1 \|_{X_1}\}
\]
and
\[
\| u\|_{X_0+X_1} = \inf\{ \|u_0\|_{X_0} + \|u_1\|_{X_1}: u  = u_0+u_1,\ u_0\in X_0,\ u_1\in X_1\},
\]
respectively. Note that $X_0\cap X_1 \imbed X_j \imbed X_0+X_1$. We say that a Banach space $X$ is \bfem{intermediate} between
$X_0$ and $X_1$ if
\[
X_0\cap X_1 \imbed X \imbed X_0+X_1.
\]

\subsection{The $J$ and $K$ norms}
For a fixed $t>0$, set
\[
J(t;u) = \max\{ \|u\|_{X_0}, t\|u\|_{X_1}\}
\]
and
\[
K(t;u) = \inf\{ \|u_0\|_{X_0} + t \|u_1\|_{X_1} : u = u_0 + u_1,\ u_0\in X_0,\ u_1\in X_1 \}.
\]

\begin{defn}[The $K$-method]\label{defn:K method}If $0\leq \theta\leq 1$ and $1\leq q \leq \infty$, then we define
\[
(X_0,X_1)_{\theta,q;K} = \{u \in X_0+X_1 : t^{-\theta} K(t;u)\in L^q_* = L^q(0,\infty; dt/t)\}.
\]
\end{defn}

In fact,
\begin{thm}[Theorem 7.10, \cite{AdFo03}]\label{thm:K method produces intermediate spaces}
If and only if either $1\leq q < \infty$ and $0<\theta<1$ or $q=\infty$ and $0\leq \theta\leq 1$, then the space $(X_0,X_1)_{\theta,q;K}$ is a nontrivial
Banach space with norm
\[
\| u \|_{\theta,q;K} = \| t^{-\theta} K(t;u) : L^q_* \|.
\]
Furthermore,
\[
\|u\|_{X_0+X_1} \leq \frac{\|u\|_{\theta,q;K}}{\|t^{-\theta}\min\{1,t\};L^q_*\|} \leq \|u\|_{X_0\cap X_1},
\]
and there hold the embeddings
\[
X_0\cap X_1\imbed (X_0,X_1)_{\theta,q;K} \imbed X_0+X_1,
\]
and $(X_0,X_1)_{\theta,q;K}$ is an intermediate space between $X_0$ and $X_1$.
\end{thm}

\begin{defn}[The $J$-method]\label{defn:J method}If $0\leq \theta\leq 1$ and $1\leq q \leq \infty$, then we define
\begin{multline*}
(X_0,X_1)_{\theta,q;J} = \bigg\{u \in X_0+X_1 : u = \int_0^\infty f(t)\frac{dt}t,\ f\in L^1(0,\infty;dt/t,X_0+X_1)\\ \text{ having values in }
 X_0\cap X_1 \text{ and such that } t^{-\theta} J(t;f)\in L^q_* = L^q(0,\infty; dt/t) \bigg\}.
\end{multline*}
\end{defn}

In fact,
\begin{thm}[Theorem 7.13 \cite{AdFo03}]
\label{thm:J method produces intermediate spaces}
If either $1\leq q < \infty$ and $0<\theta<1$ or $q=\infty$ and $0\leq \theta\leq 1$, then the space $(X_0,X_1)_{\theta,q;J}$ is a
Banach space with norm
\[
\| u \|_{\theta,q;J} = \inf_{f\in S(u)} \| t^{-\theta} J\big(t;f(t)\big) : L^q_* \|
\]
where
\[
S(u) = \Big\{f\in L^1(0,\infty;dt/t, X_0+X_1): u = \int_0^\infty f(t)\, \frac{dt}t\Big\}.
\]
Furthermore,
\[
\|u\|_{X_0+X_1} \leq \|t^{-\theta} \min\{1,t\};L^{q'}_*\| \|u\|_{\theta,q;J} \leq \|u\|_{X_0\cap X_1},
\]
where $\frac 1q + \frac 1{q'} =1$. Consequently, there hold the embeddings
\[
X_0\cap X_1\imbed (X_0,X_1)_{\theta,q;J} \imbed X_0+X_1,
\]
and $(X_0,X_1)_{\theta,q;J}$ is an intermediate space between $X_0$ and $X_1$.
\end{thm}

It is very useful to have a discrete version of the $J$ method.
\begin{thm}[Theorem 7.15, \cite{AdFo03}]
\label{thm:discrete version J method}
An element $u\in X_0+X_1$ belongs to $(X_0,X_1)_{\theta,q;J}$ if and only if $u= \sum_{j=-\infty}^\infty u_j$ where the series converges in $X_0+X_1$
and the sequence $\{2^{-j\theta}J(2^j;u_j)\}\in\ell^q$. In this case,
\[
\inf \Big\{ \|2^{-j\theta}J(2^j;u_j);\ell^q \| : u = \sum_{j=-\infty}^\infty u_j \big\}
\]
is a norm on $(X_0,X_1)_{\theta,q;J}$ equivalent to   $\|u\|_{\theta,q;J}$.
\end{thm}

If $0<\theta<1$, the $J$ and $K$ interpolations are equivalent. In fact,
\begin{thm}[Theorem 7.16, \cite{AdFo03}]
\label{thm:comparison of J and K}
If $0<\theta<1$ and $1\leq q \leq \infty$, then
\[
(X_0,X_1)_{\theta,q;J} = (X_0,X_1)_{\theta,q;K},
\]
the two spaces having equivalent norms.
\end{thm}

%
%
\subsection{An important class of intermediate spaces}
\begin{defn}\label{defn:funny h interpolation space}
Let $\{X_0,X_1\}$ be an interpolation pair of Banach spaces.
We say that $X \in \H(\theta;X_0,X_1)$ if there exist constants $C_1, C_2>0$ so that for all $u\in X$ and $t>0$,
\[
C_1 K(t;u) \leq t^{\theta} \|u\|_X \leq C_2 J(t;u)
\]
\end{defn}

\begin{lem}\label{lem:embedding of J,K,H intermediate spaces}
Let $0\leq\theta\leq 1$ and let $X$ be an intermediate space between $X_0$ and $X_1$. Then
 $X \in \H(\theta;X_0,X_1)$ if and only if $(X_0,X_1)_{\theta,1;J} \imbed X\imbed (X_0,X_1)_{\theta,\infty;K}$.
\end{lem}

\begin{cor}[Corollary 7.20, \cite{AdFo03}]
\label{cor:interpolation spaces live in H}
If $0<\theta<1$ and $1\leq q \leq \infty$, then
\[
(X_0,X_1)_{\theta,q;J} = (X_0,X_1)_{\theta,q;K} \in \H(\theta;X_0,X_1).
\]
Moreover, $X_0 \in \H(0;X_0,X_1)$ and $X_1 \in \H(1;X_0,X_1)$.
\end{cor}

The importance of the class $\H(\theta;X_0,X_1)$ is made clear from the following theorem (which is part of Theorem 7.21, \cite{AdFo03}).
\begin{thm}[The Reiteration Theorem]\label{thm:reiteration theorem} Let $0\leq \theta_0 < \theta_1 \leq 1$ and let $X_{\theta_0}$ and
$X_{\theta_1}$ be intermediate spaces between $X_0$ and $X_1$. For $0\leq \lambda \leq 1$, let
$\theta = (1-\lambda)\theta_0 + \lambda \theta_1$.
If $X_{\theta_i}\in \H(\theta_i;X_0,X_1)$ for $i=0,1$ and if either $0<\lambda<1$ and $1\leq q \leq \infty$ or $0\leq \lambda\leq 1$ and $q=\infty$, then
\[
(X_0,X_1)_{\theta,q;J}= (X_{\theta_0},X_{\theta_1})_{\lambda,q;J} = (X_{\theta_0},X_{\theta_1})_{\lambda,q;K} =(X_0,X_1)_{\theta,q;K}.
\]
\end{thm}

\subsection{Interpolation Spaces}
Let $P = \{X_0,X_1\}$ and $Q = \{Y_0,Y_1\}$ be two interpolation pairs of Banach spaces. Let $T\in\mathcal B(X_0+X_1,Y_0+Y_1)$ satisfy
$T\in \mathcal B(X_i,Y_i)$, $i=1,2$, with norm at most $M_i$. That is,
\[
\| T u_i \|_{Y_i} \leq M_i \| u_i\|_{X_i}
\]
for all $u_i\in X_i$, $i=1,2$.

If $X$ and $Y$ are intermediate spaces for $\{X_0,X_1\}$ and $\{Y_0,Y_1\}$, respectively, then we call $X$ and $Y$
\bfem{interpolation spaces of type $\theta$} for $P$ and $Q$, where $0\leq \theta\leq 1$ if every such linear operator $T$ maps $X$ to $Y$ with
norm $M$ satisfying
\begin{equation}\label{eqn:interpolation inequality}
M \leq C M_0^{1-\theta} M_1^\theta
\end{equation}
where $C$ is independent of $T$ and $C\geq 1$. If we can take $C=1$ in (\ref{eqn:interpolation inequality}), then we say that the interpolation
spaces $X$ and $Y$ are \bfem{exact}.

\begin{thm}[The Exact Interpolation Theorem, Theorem 7.23 \cite{AdFo03}]\label{thm:exact interpolation}Let
$P = \{X_0,X_1\}$ and $Q = \{Y_0,Y_1\}$ be two interpolation pairs.
\begin{enumerate}\renewcommand{\labelenumi}{(\roman{enumi})}
\item If either  $0<\theta<1$ and $1\leq q \leq \infty$ or $0\leq \theta\leq 1$ and $q=\infty$, then the intermediate spaces
$(X_0,X_1)_{\theta,q;K}$ and $(Y_0,Y_1)_{\theta,q:K}$ are exact interpolation spaces of type $\theta$ for $P$ and $Q$;

\item If either  $0<\theta<1$ and $1\leq q \leq \infty$ or $0\leq \theta\leq 1$ and $q=\infty$, then the intermediate spaces
$(X_0,X_1)_{\theta,q;J}$ and $(Y_0,Y_1)_{\theta,q:J}$ are exact interpolation spaces of type $\theta$ for $P$ and $Q$;
\end{enumerate}
\end{thm}

\bibliographystyle{alpha}
\bibliography{mybib9-5-12}

\end{document}